\documentclass[12pt]{amsart}

\newif\ifPDF
\ifx\pdfoutput\undefined\PDFfalse
\else \ifnum \pdfoutput > 0 \PDFtrue
        \else \PDFfalse
        \fi
\fi

\usepackage[centertags]{amsmath}
\usepackage{amsfonts}
\usepackage{mathrsfs}
\usepackage{textcomp}
\usepackage{amssymb}
\usepackage{amsthm}
\usepackage{newlfont}
\usepackage[all]{xy}

\ifPDF
  \usepackage[pdftex]{color, graphicx}
  \usepackage[pdftex, bookmarks, colorlinks]{hyperref}
  \hypersetup{colorlinks=false}

\else
  \usepackage{color}
  \usepackage[dvips]{graphicx}
  \usepackage[dvips]{hyperref}
\fi

\usepackage[margin=1in]{geometry}

\newtheorem{thm}{Theorem}[section]
\newtheorem{cor}[thm]{Corollary}
\newtheorem{lem}[thm]{Lemma}
\newtheorem{prop}[thm]{Proposition}
\theoremstyle{definition}
\newtheorem{defn}[thm]{Definition}
\theoremstyle{remark}
\newtheorem{rem}[thm]{Remark}
\newtheorem{example}[thm]{Example}
\numberwithin{equation}{section}

\newcommand{\norm}[1]{\left\Vert#1\right\Vert}
\newcommand{\abs}[1]{\left\vert#1\right\vert}

\newcommand{\Real}{\mathbb R}
\newcommand{\Int}{\mathbb Z}
\newcommand{\Comp}{\mathbb C}

\newcommand{\eps}{\varepsilon}

\begin{document}


\title[Comparison radius and mean dimension]{Comparison radius and mean topological dimension: Rokhlin property, comparison of open sets, and subhomogeneous C*-algebras}

\author{Zhuang Niu}
\address{Department of Mathematics, University of Wyoming, Laramie, WY, 82071, USA}
\email{zniu@uwyo.edu}


\thanks{The research is partially supported by a Simons Collaboration Grant (Grant \#317222) and partially supported by an NSF grant (DMS-1800882)}
\date{\today}


\begin{abstract}
Let $(X, \Gamma)$ be a free minimal dynamical system, where $X$ is a compact separable Hausdorff space and $\Gamma$ is a discrete amenable group. It is shown that,  if $(X, \Gamma)$ has a version of Rokhlin property (uniform Rokhlin property) and if $\mathrm{C}(X)\rtimes\Gamma$ has a Cuntz comparison on open sets, then the comparison radius of the crossed product C*-algebra $\mathrm{C}(X) \rtimes  \Gamma$ is at most half of the mean topological dimension of $(X, \Gamma)$. 

These two conditions  are shown to be  satisfied if $\Gamma = \Int$ or if $(X, \Gamma)$ is an extension of a free Cantor system and $\Gamma$ has subexponential growth. The main tools being used are Cuntz comparison of diagonal elements of a subhomogeneous C*-algebra and small subgroupoids.
\end{abstract}

\maketitle

\setcounter{tocdepth}{1}
\tableofcontents

\section{Introduction}

Consider a topological dynamical system $(X, \Gamma)$, where $X$ is a compact Hausdorff space and $\Gamma$ is a discrete amenable group. The mean (topological) dimension of $(X, \Gamma)$, denoted by $\mathrm{mdim}(X, \Gamma)$, was introduced by Gromov (\cite{Gromov-MD}), and then was developed and studied systematically by Lindenstrauss and Weiss (\cite{Lindenstrauss-Weiss-MD}). It is a numerical invariant, taking value in $[0, +\infty]$, to measure the complexity of $(X, \Gamma)$ in terms of dimension growth with respect to partial orbits.

On the other hand, the concept of dimension growth also appears in the classification theory of C*-algebras, and it is used to provide a condition for a certain C*-algebra to be classified by its Elliott invariant.  For example, 
consider a unital inductive system of C*-algebras $$A_1 \to A_2 \to \cdots \to\varinjlim A_i,$$ where  each $A_i = \mathrm{M}_{n_i}(\mathrm{C}(X_i))$ is the algebra of $\mathrm{M}_{n_i}(\Comp)$-valued continuous functions on some compact separable Hausdorff space $X_i$;  
its dimension growth is defined as $$\liminf_{i\to\infty}\frac{\mathrm{dim}(X_i)}{n_i},$$ where $\mathrm{dim}(X_i)$ is the topological covering dimension of $X_i$. 
If the dimension growth is $0$, then the isomorphism classes of the limit C*-algebras $\varinjlim A_i$ are classified by their Elliott invariant (\cite{Gong-AH}, \cite{EGL-AH}, \cite{Toms-SDG}). (See \cite{GLN-TAS}, \cite{EGLN-ASH} \cite{EN-K0-Z}, and \cite{EGLN-DR} for further developments of the classification of C*-algebras in this direction.)

For a general abstract C*-algebra, it might not have an obvious inductive limit decomposition in terms of homogeneous or subhomogeneous C*-algebras as the example above. In such cases, comparison radius, introduced by Toms (\cite{RC-Toms}) and denoted by $\mathrm{rc}(A)$, plays a role as the dimension growth of such a C*-algebra $A$. 

Let us start with some comparison theory of C*-algebras. Let $a, b$ be positive elements of a matrix algebra over $A$. Then $a$ is said to be Cuntz sub-equivalent to $b$ if there are sequences $(x_n), (y_n)$ in a matrix algebra over $A$ such that $$x_nby_n \to a,\quad\mathrm{as}\quad n\to\infty.$$ In the case that $a, b$ are projections, the Cuntz sub-equivalence relation recovers the classical Murray-von Neumann sub-equivalence relation; and moreover, if $A=\mathrm{C}(X)$, $a$ is Cuntz sub-equivalent to $b$ if and only if the vector bundle over $X$ induced by $a$ is isomorphic to a sub-bundle of the the vector bundle induced by $b$.

If $a$ is Cuntz sub-equivalent to $b$, then the rank function induced by $a$ is always dominated by the rank function induced by $b$, just like the case above, the rank of the vector bundle induced by $a$ is at most the rank of the vector bundle induced by $b$. The converse is not true in general, as there are many examples of a pair of vector bundles such that the bundle with smaller rank is not  isomorphic to a sub-bundle of the other one.   

However, if the gap between the ranks of two (complex) vector bundles is sufficiently large (larger than $\frac{1}{2}\mathrm{dim(X)}$), then the converse is true; that is,  the vector bundle with small dimension is isomorphic to a sub-bundle of the vector bundle with larger dimension.

The comparison radius is the infimum of all the gaps of the ranks so that the converse does hold (see Definition \ref{definition-RC}), and a typical example is that the comparison radius of $\mathrm{M}_n(\mathrm{C}(X))$ is at most $\frac{1}{2} \frac{\mathrm{dim}(X)}{n}$, that is,  half of the dimension ratio of $\mathrm{M}_n(\mathrm{C}(X))$. Thus, in general, the comparison radius is regarded as a version of dimension growth for an abstract C*-algebra.

For the given topological dynamical system $(X, \Gamma)$, the natural C*-algebra to be considered is the crossed product C*-algebra $\mathrm{C}(X) \rtimes \Gamma$. It would be interesting to compare the mean dimension of  $(X, \Gamma)$ (the dynamical dimension growth) with the comparison radius of $\mathrm{C}(X) \rtimes \Gamma$ (the C*-algebraic dimension growth), and this is the main motivation of this paper. Indeed, Phillips and Toms conjectured that the comparison radius of $\mathrm{C}(X) \rtimes \Gamma$ should equal half of the mean dimension of $(X, \Gamma)$.

Many results in this direction have been obtained: For free minimal $\Int$-actions, if $X$ is finite dimensional (hence the dynamical system has mean dimension zero), it was shown in \cite{TW-Dym-1} that the algebra $\mathrm{C}(X)\rtimes\Int$ has finite nuclear dimension; therefore the C*-algebra has strict comparison of positive elements, and  the radius of comparison is zero. Still with the assumption that $X$ is finite dimensional, this result was generalized to free minimal actions by $\Int^d$ in \cite{Szabo-Z} or by a group with comparison property in \cite{KS-comparison}.

Without assuming $X$ to be finite dimensional, for a free minimal $\Int$-action with zero mean dimension, it was shown in \cite{EN-MD0} that the crossed product C*-algebra absorbs the Jiang-Su algebra (and therefore has finite nuclear dimension by \cite{ENST-ASH}). In particular, this implies that the comparison radius is $0$. For  free minimal $\Int$-actions without zero mean dimension, Phillips showed in \cite{NCP-fmd} that the radius of comparison of $\mathrm{C}(X) \rtimes \Int$ is at most $1+36\mathrm{mdim}(X, \sigma, \Int)$.

The argument in \cite{TW-Dym-1}, \cite{EN-MD0}, or \cite{NCP-fmd} relies on the Putnam's orbit-cutting algebra (or the large sub-algebra) $A_y$;  and in the case of zero mean dimension, the argument in \cite{EN-MD0} also heavily  depends on the small boundary property (which is equivalent to mean dimension zero in the case of $\mathbb Z$-actions). 

However, beyond the $\mathbb Z$-action case, it is not clear in general how to construct large sub-algebras;  moreover, once the dynamical system does not have mean dimension zero, the small boundary property does not hold anymore. 

To deal with these difficulties, in this paper we consider the following properties of the dynamical systems and the crossed product C*-algebras: 
\begin{itemize}
\item {\em Uniform Rokhlin Property (URP)}. The topological dynamical system $(X, \Gamma)$ is said to have Uniform Rokhlin Property (URP) if there exist disjoint towers with shapes arbitrarily invariant and all level sets are open, such that the complement of the towers has arbitrarily small orbit capacity (see Definition \ref{UPR}). 
\item {\em Cuntz comparison of open sets (COS)}.  Consider the C*-algebra $\mathrm{C}(X) \rtimes\Gamma$. Note that any open set $E\subseteq X$ represents a Cuntz equivalence class of $\mathrm{C}(X)$ and hence a Cuntz equivalence class of $\mathrm{C}(X) \rtimes\Gamma$. Denote it by $[E]$. Then the dynamical system $(X, \Gamma)$ is said to have Cuntz comparison of open sets (COS) if there are $\lambda\in \Real^+$ and $m\in \mathbb N$ such that for any open sets $$E, F\subseteq X\quad \mathrm{with}\quad \mu(E) < \lambda \mu(F)$$ for all ergodic measures $\mu$, one has that $$[E] \leq m[F]$$ in the Cuntz semigroup of $\mathrm{C}(X) \rtimes\Gamma$. (See Definition \ref{definition-comp}.)
\end{itemize}

A consequence of the (URP) is that the crossed product C*-algebra $\mathrm{C}(X)\rtimes\Gamma$ can be (weakly) tracially approximated by homogeneous C*-algebras with dimension ratio at most the mean dimension of $(X, \Gamma)$ (see Theorem \ref{T-SDG}); roughly speaking, any element of $\mathrm{C}(X)\rtimes\Gamma$ can be approximately decomposed as the (not necessary orthogonal) sum of an element in a homogeneous C*-algebra with dimension ratio at most $\mathrm{mdim}(X, \Gamma)$ and an element which is uniformly small under all traces (i.e., under all ergodic measures). 

This tracial approximation property, together with Cuntz comparison of open sets, implies that the radius of comparison of the crossed product C*-algebra is at most half of the mean dimension:
\theoremstyle{theorem}
\newtheorem*{thmA}{Theorem}
\begin{thmA}[Theorem \ref{main-thm}]
Let $(X, \Gamma)$ be a free and minimal topological dynamical system satisfying the (URP) and (COS). 
Then, 
\begin{equation}\label{rc-md-comp}
\mathrm{rc}(\mathrm{C}(X)\rtimes \Gamma) \leq \frac{1}{2}\mathrm{mdim}(X, \Gamma).
\end{equation}
\end{thmA}

All minimal $\Int$-actions have the (URP)  (see Lemma \ref{LRT-Z}); and if $(X, \Gamma)$ is an extension of a free minimal system with small boundary property, it has the (URP) (see Corollary \ref{URP-ext-SBP}).

To investigate when $\mathrm{C}(X)\rtimes\Gamma$ has the (COS), small subgroupoids of $X \rtimes \Gamma$ are considered. These are the open and relatively compact subgroupoids of $X \rtimes \Gamma$, and they can be regarded as a local version of the orbit-cutting subalgebra $A_y$. It is well known that the C*-algebra of a small subgroupoid is subhomogeneous (i.e., the dimensions of its irreducible representations are uniformly bounded). 

It turns out that the C*-algebra of a small subgroupoid is rather special:  it is a recursive subhomogeneous C*-algebra with diagonal maps (see Theorem \ref{diag-SHA}); and with a revised argument of \cite{EN-MD0},  the Cuntz class of diagonal elements of such a C*-algebra are shown to be determined by their ranks (see Theorem \ref{DSH-comp}), provided that the dimensions of the irreducible representations are sufficiently large, but regardless of the topological dimension of its spectrum. 

This comparison property of diagonal elements leads to a Cuntz comparison of open sets for $(X, \Gamma)$: if $\Gamma$ is amenable and $X \rtimes\Gamma$ has small subgroupoids with each orbit arbitrarily invariant, then the dynamical system $(X, \Gamma)$ has the Cuntz-comparison property on open sets (see Corollary \ref{orb-comp-measure}). Such small subgroupoids always exist if 
\begin{itemize}
\item  $\Gamma=\Int$, or
\item  $(X, \Gamma)$ is an extension of a free Cantor system and $\Gamma$ has subexponential growth. 
\end{itemize}
Therefore, in these cases, the dynamical system $(X, \Gamma)$ has the (COS) (see Corollary \ref{comp-open-sets-Z} and Corollary \ref{comp-open-set-ext}); in particular, the estimate \eqref{rc-md-comp} holds (see Corollary \ref{main-thm-Z} and Corollary \ref{main-thm-G}).

\begin{rem}
In \cite{Niu-MD-Zd}, using the adding-one-dimension and going-down argument of \cite{GLT-Zk}, the (COS) and (URP) (and hence the estimate \eqref{rc-md-comp}) are obtained for arbitrary minimal and free $\Int^d$-actions. Hence \eqref{rc-md-comp} holds for all minimal free $\mathbb Z^d$-actions.

In \cite{Niu-MD-Z-absorbing}, still under the assumption of the (COS) and (URP), it is shown that $\mathrm{mdim}(X, \Gamma)=0$ implies that the C*-algebra $\mathrm{C}(X)\rtimes\Gamma$ is classified by its Elliott invariant. 


\end{rem}

\section{Notation and preliminaries}

\subsection{Topological Dynamical Systems}

\begin{defn}
A topological dynamical system $(X, \Gamma)$ consists of a separable compact Hausdorff space $X$, a discrete group $\Gamma$, and a homomorphism $\Gamma \to \mathrm{Homeo}(X)$, where $\mathrm{Homeo}(X)$ is the group of homeomorphisms of $X$, acting on $X$ from the right. In this paper, we frequently omit the word topological, and just refer it as a dynamical system. 

The dynamical system $(X, \Gamma)$ is said to be free if $x\gamma=x$ implies $\gamma=e$, where $x\in X$ and $\gamma\in\Gamma$.

A closed set $Y \subseteq X$ is said to be invariant if $$Y\gamma = Y,\quad \gamma \in\Gamma,$$ and the dynamical system $(X, \Gamma)$ is said to be minimal if $\varnothing$ and $X$ are the only invariant closed subsets.

If $\Gamma = \Int$, then $\sigma$ is induced by a single homeomorphism of $X$, which is still denoted by $\sigma$. In this case, the dynamical system is denoted by $(X, \sigma)$.
\end{defn}

\begin{rem}
In the case $\Gamma = \Int$, it is well known that $(X, \sigma)$ is minimal if and only if one of the following holds:
\begin{enumerate}
\item if $Y \subseteq X$ is closed and $\sigma(Y) \subseteq Y$, then $Y=X$ or $Y=\varnothing$;
\item for any $x\in X$, the forward orbit $\{x, \sigma(x), \sigma^2(x)... \}$ is dense;
\item for any $x\in X$, the orbit $\{..., \sigma^{-1}(x), x,  \sigma(x), ... \}$ is dense;
\item for any non-empty open set $U\subseteq X$, there is $n\in\mathbb N$ such that $$U\cup\sigma^{-1}(U)\cup\cdots\cup\sigma^{-n}(U)=X.$$
\end{enumerate}
\end{rem}

\begin{defn}
A Borel measure $\mu$ on $X$ is invariant if for any Borel set $E\subseteq X$, one has $$\mu(E) = \mu(E\gamma),\quad \gamma\in\Gamma.$$ Denote by $\mathcal M_1(X, \Gamma)$ the set of all invariant Borel probability measures on $X$. It is a Choquet simplex under the weak* topology.
\end{defn}

\begin{defn}\label{defn-amenable}
Let $\Gamma$ be a (countable) discrete group. Let $K\subseteq\Gamma$ be a finite set and let $\delta>0$. Then a finite set $F \subseteq \Gamma$ is said to be $(K, \eps)$-invariant if $$\frac{\abs{FK\Delta F}}{\abs{F}} < \eps.$$

The group $\Gamma$ is amenable if there is a sequence  $(\Gamma_n)$ of finite subsets of $\Gamma$ such that for any $(K, \eps)$, the set $\Gamma_n$ is $(K, \eps)$-invariant if $n$ is sufficiently large. The sequence $(\Gamma_n)$ is called a F{\o}lner sequence.

The $K$-interior of a finite set $F\subseteq\Gamma$ is defined as $$\mathrm{int}_K(F) = \{\gamma\in F: \gamma K \subseteq F\}.$$ Note that $$\abs{F\setminus\mathrm{int}_K(F)} \leq\abs{K}\abs{FK\setminus F}\leq \abs{K}\abs{FK\Delta F} ,$$ and hence for any $\eps>0$, if $F$ is $(K, \frac{\eps}{\abs{K}})$-invariant, then $$\frac{\abs{F\setminus\mathrm{int}_K(F)}}{\abs{F}} < \eps.$$

\end{defn}

\begin{defn}[see \cite{Lindenstrauss-Weiss-MD}]
Consider a topological dynamical system $(X, \Gamma)$, where $\Gamma$ is amenable, and let $E\subseteq X$. The orbit capacity of $E$ is defined by
$$\mathrm{ocap}(E):=\lim_{n\to\infty}\frac{1}{\abs{\Gamma_n}}\sup_{x\in X}\sum_{\gamma\in \Gamma_n} \chi_E(x\gamma),$$
where $(\Gamma_n)$ is a F{\o}lner sequence, and $\chi_E$ is the characteristic function of $E$.
The limit always exists and is independent from the choice of the F{\o}lner sequence $(\Gamma_n)$.
\end{defn}
\begin{rem}\label{ocp-prop}
Orbit capacity has the following properties:
\begin{enumerate}
\item If $E$ is a closed set, then $\mathrm{ocap}(E) \leq \eps$ if and only if $\mu(E) \leq \eps$, $\mu \in\mathcal M_1(X, \Gamma)$.
\item The orbit capacity is semicontinuous in the sense that for any closed set $D \subseteq X$ and any $\eps>0$, there is an open neighbourhood $U\supseteq D$ such that $\mathrm{ocap}(U) < \mathrm{ocap}(D) + \eps$.
\end{enumerate}
\end{rem}

\begin{defn}[see \cite{Gromov-MD} and \cite{Lindenstrauss-Weiss-MD}]
Let $\mathcal U$ be an open cover of $X$. Define
$$D(\mathcal U)=\min\{\mathrm{ord}(\mathcal V): \textrm{$\mathcal V$ is an open cover of $X$ and $\mathcal V \preceq U$}\},$$ where $$\mathrm{ord}(\mathcal V)=-1 + \sup_{x\in X}\sum_{V\in\mathcal V} \chi_V(x),$$ and $\mathcal V \preceq \mathcal U$ means that, for any $V\in\mathcal V$, there is $U\in\mathcal U$ with $V\subseteq U$.

Consider a topological dynamical system $(X, \Gamma)$, where $\Gamma$ is a discrete amenable group. The mean topological dimension is defined by
$$\mathrm{mdim}(X, \Gamma):=\sup_{\mathcal U}\lim_{n\to\infty}\frac{1}{\abs{\Gamma_n}}D(\bigvee_{\gamma\in\Gamma_n} \gamma^{-1}(\mathcal U)),$$
where $\mathcal U$ runs over all finite open covers of $X$, $(\Gamma_n)$ is a F{\o}lner sequence (the limit is independent from the choice of $(\Gamma_n)$), and $\alpha \vee \beta$ denotes the open cover $$\{U \cap V: U\in \alpha,\ V\in\beta\}$$ for any open covers $\alpha$ and $\beta$.

Note that in the case $\Gamma = \Int$, one has
$$\mathrm{mdim}(X, \sigma)=\sup_{\mathcal U}\lim_{n\to\infty}\frac{1}{n}D(\mathcal U \vee \sigma^{-1}(\mathcal U) \vee \cdots\vee \sigma^{-n+1}(\mathcal U)).$$

\end{defn}

\begin{rem}
It is shown in \cite{Dou-MD} that, if $\Gamma$ is a countable infinite amenable group, then each number in $[0, +\infty]$ can be realized as the mean dimension  of a minimal dynamical system $(X, \Gamma)$.
\end{rem}

\subsection{Crossed product C*-algebras}
Consider a topological dynamical system $(X, \Gamma)$. 
The (full) crossed product C*-algebra $A=\mathrm{C}(X)\rtimes\Gamma$ is defined to be the universal C*-algebra
$$\textrm{C*}\{f, u_\gamma;\ u_\gamma fu_\gamma^*=f(\cdot\gamma) = f\circ\gamma,\ u_{\gamma_1}u^*_{\gamma_2} = u_{\gamma_1\gamma_2^{-1}},\ u_e=1,\ f\in \mathrm{C}(X),\ \gamma, \gamma_1, \gamma_2 \in \Gamma\}.$$
The C*-algebra $A$ is nuclear (Corollary 7.18 of \cite{Williams}) if $\Gamma$ is amenable. If, moreover, $(X, \Gamma)$ is minimal and topologically free, the C*-algebra $A$ is simple (Theorem 5.16 of \cite{Effros-Hahn} and  Th\'{e}or\`{e}me 5.15 of \cite{ZM-prod}), i.e., $A$ has no non-trivial two-sided ideals.  

\subsection{Comparison for positive elements of a C*-algebra}

\begin{defn}
Let $A$ be a C*-algebra, and let $a, b\in A^+$. We say that $a$ is Cuntz sub-equivalent to $b$, denoted by $a \precsim b$, if there are $x_n$, $y_n$, $n=1, 2, ...$, such that $$\lim_{n\to\infty} x_nby_n = a,$$ and we say that $a$ is Cuntz equivalent to $b$ if $a\precsim b$ and $b \precsim a$.

Denote by $\mathrm{M}_n(A)$ the C*-algebra of $n\times n$ matrices over $A$. Regard $\mathrm{M}_n(A)$ as the upper-left conner of $\mathrm{M}_{n+1}(A)$, denote by $$\mathrm{M}_\infty(A) = \bigcup_{n=1}^\infty \mathrm{M}_n(A),$$ the algebra of all finite matrices over $A$, and denote by $\mathrm{Tr}$ the standard (unbounded) trace of $\mathrm{M}_\infty(A)$.

Let $\tau: A \to\Comp$ be a tracial state. Then define the rank function
$$\mathrm{d}_\tau(a):=\lim_{n\to\infty}(\tau\otimes\mathrm{Tr})(a^{\frac{1}{n}})=\mu_{\tau\otimes\mathrm{Tr}}(\mathrm{sp}(a)\cap(0, +\infty)),\quad a\in \mathrm{M}_\infty(A)^+,$$ where $\mu_{\tau\otimes\mathrm{Tr}}$ is the Borel measure induced by ${\tau\otimes\mathrm{Tr}}$ on $\mathrm{sp}(a)$, the spectrum of $a$.  


It is well known that for any $a, b\in \mathrm{M}_n(A)^+$ for some $n\in\mathbb N$, if $a\precsim b$, then
$$\mathrm{d}_\tau(a) \leq \mathrm{d}_\tau(b).$$

\end{defn}
\begin{example}
Consider $f\in \mathrm{C}(X)^+$ and let $\mu$ be a Borel probability measure on $X$. Then
$$\mathrm{d}_{\tau_\mu} = \mu(f^{-1}(0, +\infty)),$$
where $\tau_\mu$ is the trace of $\mathrm{C}(X)$ defined by $$\tau_\mu(f) = \int f\mathrm{d}\mu,\quad f\in\mathrm{C}(X).$$

Let $f, g\in\mathrm{C}(X)$ be positive functions, and define the open sets $$E = f^{-1}(0, +\infty)\quad\mathrm{and}\quad F=g^{-1}(0, +\infty).$$ Then $f \precsim g$ if, and only if, $E \subseteq F$. That is, the Cuntz equivalence classes of $f$ and $g$ are determined by their open supports. 

For each open set $E \subseteq X$, pick a continuous function $\varphi_{E}: X \to [0, +\infty)$ satisfying
\begin{equation}\label{defn-phi-E}
 E = \varphi_E^{-1}(0, +\infty).
\end{equation} 
For instance, one can pick $\varphi_E(x) = d(x, X\setminus E)$, where $d$ is a compatible metric on $X$. The notation $\varphi_E$ will be used throughout the paper. Note that the Cuntz equivalence class of $\varphi_E$ is independent of the choice of the individual function $\varphi_E$.
\end{example}

\begin{defn}\label{defn-p-cut}
Let $a\in A^+$, where $A$ is a C*-algebra, and let $\eps>0$. Define
$$(a - \eps)_+ =f(a) \in A,$$
where $f(t)=\max\{t-\eps, 0\}$.
\end{defn}
The following is a frequently used fact on Cuntz comparison.
\begin{lem}[Section 2 of \cite{RorUHF2}]
Let $a, b$ be positive elements of a C*-algebra $A$. Then $a \precsim b$ if and only if $(a-\eps)_+ \precsim b$ for all $\eps >0$.
\end{lem}

\begin{defn}[Definition 6.1 of \cite{RC-Toms}]\label{definition-RC}
The radius of comparison of a unital C*-algebra $A$, denoted by $\mathrm{rc}(A)$, is the infimum of the set of real numbers $r > 0$ such that  if $a, b\in \mathrm{M}_\infty(A)^+$ satisfy
$$\mathrm{d}_\tau(a) + r < \mathrm{d}_\tau(b),\quad \tau\in\mathrm{T}(A),$$ then $a \precsim b$, where $\mathrm{T}(A)$ is the simplex of tracial states. (In \cite{RC-Toms}, the radius of comparison is defined in terms of quasitraces instead of traces; but since all the algebras considered in this paper are nuclear, by \cite{Haagtrace}, any quasitrace actually is a trace.)
\end{defn}
\begin{example}
Let $X$ be a compact Hausdorff space. Then 
\begin{equation}\label{comp-dim-control}
\mathrm{rc}(\mathrm{M}_n(\mathrm{C}(X)))\leq \frac{1}{2}\frac{\mathrm{dim}(X) - 1}{n},
\end{equation} 
where $\mathrm{dim(X)}$ is the topological covering dimension of $X$ (a lower bound of $\mathrm{rc}(\mathrm{C}(X))$ in terms of cohomological dimension is given in \cite{EN-RC}).
\end{example}
The main purpose of this paper is to investigate whether the dynamical version of \eqref{comp-dim-control} holds, that is, whether one has
$$\mathrm{rc}(\mathrm{C}(X) \rtimes \Gamma) \leq \frac{1}{2}\mathrm{mdim}(X, \Gamma).$$

\section{Uniform Rokhlin property and structure of $\mathrm{C}(X)\rtimes\Gamma$}

Let us first consider a Rokhlin property for a dynamical system $(X, \Gamma)$. It turns out that this Rokhlin property implies that the crossed-product C*-algebra $\mathrm{C}(X) \rtimes \Gamma$ can be weakly tracially approximated by (not necessary unital) homogeneous C*-algebras.

\subsection{Uniform Rokhlin Property}
\begin{defn}\label{UPR}
A topological dynamical system $(X, \Gamma)$, where $\Gamma$ is a discrete amenable group, is said to have Uniform Rokhlin Property (URP) if for any $\eps>0$ and any finite set $K\subseteq \Gamma$, there exist closed sets $B_1, B_2, ..., B_S \subseteq X$ and $(K, \eps)$-invariant sets $\Gamma_1, \Gamma_2, ..., \Gamma_S \subseteq \Gamma$ such that
$$B_s\gamma,\quad \gamma\in \Gamma_s,\ s=1, ..., S, $$
are mutually disjoint and
$$\mathrm{ocap}(X\setminus\bigsqcup_{s=1}^S\bigsqcup_{\gamma\in \Gamma_s}B_s\gamma) < \eps.$$
\end{defn}

In fact, in the definition of the (URP), the base sets $B_1, ..., B_S$ can also be assumed to be open:

\begin{lem}
A topological dynamical system $(X, \Gamma)$  has (URP) if, and only if, it satisfies Definition \ref{UPR} but with $B_1$, ..., $B_S$ being open sets, instead.
\end{lem} 
\begin{proof}
Let $(K, \eps)$ be given. Assume there exist closed sets $B_1, B_2, ..., B_S \subseteq X$ and $(K, \eps)$-invariant sets $\Gamma_1, \Gamma_2, ..., \Gamma_S \subseteq \Gamma$ such that
$$B_s\gamma,\quad \gamma\in \Gamma_s,\ s=1, ..., S, $$
are mutually disjoint and
$$\mathrm{ocap}(X\setminus\bigsqcup_{s=1}^S\bigsqcup_{\gamma\in \Gamma_s}B_s\gamma) < \eps.$$ Since  each $B_s$, $s=1, ..., S$, is compact, one can choose an open neighbourhood $B_s'$ of $B_s$ such that $$B'_s\gamma,\quad \gamma\in \Gamma_s,\ s=1, ..., S$$
are mutually disjoint. Then $$\mathrm{ocap}(X\setminus\bigsqcup_{s=1}^S\bigsqcup_{\gamma\in \Gamma_s}B'_s\gamma) \leq \mathrm{ocap}(X\setminus\bigsqcup_{s=1}^S\bigsqcup_{\gamma\in \Gamma_s}B_s\gamma) < \eps,$$ and thus satisfies Definition \ref{UPR} with open base sets.

For the converse, assume there exist open sets $B_1, B_2, ..., B_S \subseteq X$ and $(K, \eps)$-invariant sets $\Gamma_1, \Gamma_2, ..., \Gamma_S \subseteq \Gamma$ such that
$$B_s\gamma,\quad \gamma\in \Gamma_s,\ s=1, ..., S, $$
are mutually disjoint and
$$\mathrm{ocap}(X\setminus\bigsqcup_{s=1}^S\bigsqcup_{\gamma\in \Gamma_s}B_s\gamma) < \eps.$$ 

Consider the closed set $X\setminus\bigsqcup_{s=1}^S\bigsqcup_{\gamma\in \Gamma_s} B_s\gamma$. It has an open neighbourhood $U$ such that $\mathrm{ocap}(U) < \eps$ (see Remark \ref{ocp-prop}), and there is a closed subset $B_s'\subseteq B_s$, $s=1, ..., S$, such that $$X\setminus\bigsqcup_{s=1}^S\bigsqcup_{\gamma\in \Gamma_s}B'_s\gamma \subseteq U.$$ Indeed, $B_s'$ can be chosen as $\bigcup_{\gamma\in\Gamma_s} (U^c\cap (B_s\gamma))\gamma^{-1}$ (note that $U^c$ is closed in $X$, $U^c \subseteq \bigsqcup_{s=1}^S\bigsqcup_{\gamma\in \Gamma_s}B_s\gamma$, and each $B_s\gamma$ is open; so each set $U^c\cap (B_s\gamma)$ is closed in $X$). Then the closed sets $B_s'$ satisfy Definition \ref{UPR} (with base sets closed).
\end{proof}

\begin{thm}[Theorem 5.5 of \cite{KS-comparison}]\label{KS-tower}
If $(X, \Gamma)$ is free and has small boundary property, then $(X, \Gamma)$ has the (URP).
\end{thm}

\begin{rem}
In the case that $\Gamma=\Int^d$, it follows from Theorem 1.10.1 of \cite{GY-ETDS-2011}, that if $(X, \Int^d)$ is an extension of free dynamical system with the small boundary property, then it has the Topological Rokhlin Property in the sense of 1.9 of \cite{GY-ETDS-2011} (actually, the proof of 1.10.1 of \cite{GY-ETDS-2011} shows that the dynamical system $(X, \Int^d)$ has the (URP)).
\end{rem}

\begin{rem}\label{KS-AFM}
The level sets of the towers obtained in \cite{KS-comparison} actually have arbitrarily small diameter, and this is referred by the authors  as the property of almost finiteness in measure (Definition 3.5 of \cite{KS-comparison}). This property is shown to be equivalent to the small boundary property (Theorem 5.5 of \cite{KS-comparison}). In the case that $X$ is a Cantor set, the result is also obtained in \cite{DH-tiling}.
\end{rem}

It follows from Corollary 3.4 of \cite{Lind-MD} that any free minimal $\mathbb Z$-action has the (URP):

\begin{lem}\label{LRT-Z}
Let $\sigma: X \to X$ be a homeomorphism, and assume that $(X, \sigma)$ is an extension of a free minimal system. Then, for any $\eps > 0$, any $N\in \mathbb N$, there is a closed set $B\subseteq X$ such that
$$\sigma^i(B),\quad i=0, ..., N-1$$
are disjoint and
$$\mathrm{ocap}(X\setminus\bigsqcup_{i=0}^{N-1}\sigma^i(B)) < \eps.$$ In particular, $(X, \sigma)$ has the (URP) (with $S=1$).
\end{lem}
\begin{proof}
Let $\eps>0$ and $N\in\mathbb N$ be arbitrary.
It follows from Corollary 3.4 of \cite{Lind-MD} that there is a continuous function $n: X \to [0, +\infty)$ such that 
$$\mathrm{ocap}\{x: n(\sigma(x)) \neq n(x) + 1\ \mathrm{or}\ n(x)\notin \Int\} \leq \frac{\eps}{2N(N(2N+1)+1)}.$$

Consider the level sets $$E'_k=n^{-1}(k),\quad k=0, 1, 2, ...$$ and the (open) set
$$F=\sigma^{-N+1}(F_0) \cup \sigma^{-N+2}(F_0) \cup \cdots \cup \sigma^{N-2}(F_0) \cup \sigma^{N-1}(F_0),$$ where $$F_0=\{x: n(\sigma(x)) \neq n(x) + 1\}.$$ 
Note that 
$$n(\sigma^k(x)) = n(x) + k,\quad x\in X\setminus F,\ -N+1\leq k \leq N-1,$$
\begin{equation}\label{small-F}
\mathrm{ocap}(F) \leq 2N\cdot\mathrm{ocap}(F_0)\leq \delta:=\frac{\eps}{N(2N+1)+1}.
\end{equation}
and
\begin{equation}\label{small-non-integer}
\mathrm{ocap}(X\setminus \bigsqcup_{k=0}^\infty E_k') \leq \frac{\eps}{2N(N(2N+1)+1)} < \delta.
\end{equation}

%
%

Define $$E_l:=\bigsqcup_{i=0}^\infty E'_{iN+l},\quad l=0, ..., N-1,$$ and define
$$B=(E_0\setminus F) \cap \sigma^{-1}(E_1\setminus F)\cap \cdots\cap\sigma^{-N+1}(E_{N-1}\setminus F).$$
For each $l=0, 1, ..., N-1$, since $E'_{iN+l}$, $i=0, 1, ...,$ is eventually empty, the set $E_l$ is closed, and hence the set $B$ is closed as well. Since
$$B\subseteq E_0,\ \sigma{B}\subseteq E_1,\ ...,\ \sigma^{N-1}(B) \subseteq E_{N-1},$$
and
$E_0$, $E_1$, ..., $E_{N-1}$ are mutually disjoint, one has that
$$B,\ \sigma(B), ..., \sigma^{N-1}(B)$$
are mutually disjoint. Thus, $B$, $\sigma(B)$, ..., $\sigma^{N-1}(B)$ form a tower. 

Let us estimate $\mathrm{ocap}(X\setminus\bigsqcup_{l=0}^{N-1} \sigma^l(B))$. Note that
\begin{eqnarray*}
E_{0} \setminus B & = & E_{0}\setminus ((E_{0}\setminus F) \cap \sigma^{-1}(E_{1}\setminus F)\cap \sigma^{-2}(E_{2}\setminus F)\cap\cdots\cap\sigma^{-N+1}(E_{N-1}\setminus F)) \\
& = & (E_{0}\cap F) \cup (E_{0}\setminus\sigma^{-1}(E_{1}\setminus F))\cup\cdots\cup  (E_{0}\setminus\sigma^{-N+1}(E_{N-1}\setminus F)) \\
& = & (E_{0}\cap F) \cup \sigma^{-1} (\sigma(E_{0})\setminus (E_{1}\setminus F)) \cup \cdots \cup \sigma^{-N+1} (\sigma^{N-1}(E_{0})\setminus (E_{N-1}\setminus F)) \\
& = & (E_{0}\cap F) \cup \sigma^{-1} (\sigma((E_{0}\setminus F)\cup(E_{0}\cap F))\setminus (E_{1}\setminus F)) \cup \cdots \cup \\
& & \sigma^{-N+1} (\sigma^{N-1}((E_{0}\cup F)\cup(E_{0}\cap F))\setminus (E_{N-1}\setminus F)) \\
& \subseteq & (E_{0}\cap F) \cup \sigma^{-1} (E_{1}\cup\sigma(E_{0}\cap F)\setminus (E_{1}\setminus F)) \cup \cdots \cup \\
& & \sigma^{-N+1} (E_{N-1}\cup\sigma^{N-1}(E_{0}\cap F)\setminus (E_{N-1}\setminus F))  \\
& \subseteq & (E_{0}\cap F) \cup \sigma^{-1} ((E_{1}\cap F)\cup\sigma(E_{0}\cap F)) \cup \cdots \cup \\
& & \sigma^{-N+1} ((E_{N-1}\cap F)\cup\sigma^{N-1}(E_{0}\cap F)),
\end{eqnarray*}
and hence by \eqref{small-F},
\begin{eqnarray*}
\mathrm{ocap}(E_0\setminus B) &\leq &\mathrm{ocap}(E_{0}\cap F) + (\mathrm{ocap}(E_{1}\cap F) +\mathrm{ocap}(E_{0}\cap F)) + \cdots + \\
& & (\mathrm{ocap}(E_{N-1}\cap F) + \mathrm{ocap}(E_{0}\cap F)) \leq 2N\delta.
\end{eqnarray*}
Therefore, for any $l=1, 2, ..., N-1$, one has
\begin{eqnarray*}
\mathrm{ocap}(E_l \setminus \sigma^{l}(B)) & = & \mathrm{ocap}(\sigma^{-l}(E_l) \setminus B) \\
& = & \mathrm{ocap}(\sigma^{-l}((E_l\setminus F)\cup(E_l\cap F)) \setminus B) \\
& \leq & \mathrm{ocap}((E_0\cup \sigma^{-l}(E_l\cap F)) \setminus B) \\
& \leq & \mathrm{ocap}((E_0\setminus B) \cup \sigma^{-l}(E_l\cap F))\\
& \leq & (2N+1)\delta,
\end{eqnarray*}
and thus, together with \eqref{small-non-integer},
\begin{eqnarray*}
\mathrm{ocap}(X\setminus\bigsqcup_{l=0}^{N-1} \sigma^l(B)) & = & \mathrm{ocap}((X\setminus \bigsqcup_{l=0}^{N-1} E_l) \cup (\bigsqcup_{l=0}^{N-1} E_l ) \setminus\bigsqcup_{l=0}^{N-1} \sigma^l(B)) \\
& \leq & \mathrm{ocap}(X\setminus \bigsqcup_{k=0}^{\infty} E'_k) + \sum_{l=0}^{N-1} \mathrm{ocap} ({E}_l\setminus \sigma^l(B))  \\
&\leq & \delta + N(2N+1)\delta < \eps,
\end{eqnarray*}
as desired.
\end{proof}

\begin{lem}\label{URP-ext}
If a dynamical system $(X, \Gamma)$ is an extension of a dynamical system $(Y, \Gamma)$ which is free and has the (URP), then $(X,  \Gamma)$ has the (URP).
\end{lem}
\begin{proof}
Denote the quotient map from $(X, \Gamma)$ to $(Y, \Gamma)$  by $\pi$. Note that for any set $E \subseteq Y$, one has $\mathrm{ocap}(\pi^{-1}(E)) \leq \mathrm{ocap}(E)$. Since $(Y, \Gamma)$ has the (URP), for any $\eps>0$ and any finite set $K\subseteq \Gamma$, there exist closed sets $B_1, B_2, ..., B_S \subseteq X$ and $(K, \eps)$-invariant sets $\Gamma_1, \Gamma_2, ..., \Gamma_S \subseteq \Gamma$ such that
$$B_s\gamma,\quad \gamma\in \Gamma_s,\ s=1, ..., S $$
are mutually disjoint and
$$\mathrm{ocap}(X\setminus\bigsqcup_{s=1}^S\bigsqcup_{\gamma\in \Gamma_s}B_s\gamma) < \eps.$$ Then the closed sets $$\pi^{-1}(B_s),\quad s=1, 2, ..., S$$ together with $\Gamma_s$, $s=1, 2, ..., S$, form disjoint Rokhlin towers of $(X, \Gamma)$ with complement of orbit capacity at most $\eps$.
\end{proof}


Since any free topological dynamical system with the small boundary property has the (URP) (Theorem 5.5 of \cite{KS-comparison}; see Theorem \ref{KS-tower} above), one has the following corollary.
\begin{cor}\label{URP-ext-SBP}
If a topological dynamical system $(X, \Gamma)$ is an extension of a free dynamical system with small boundary property, then $(X,  \Gamma)$ has the (URP). In particular, if $(X, \Gamma)$ is an extension of a free minimal Cantor system, $(X, \Gamma)$ has the (URP).
\end{cor}

%
%
%

\subsection{A tracial approximation structure for $\mathrm{C}(X)\rtimes\Gamma$}
Considering a topological dynamical system $(X, \Gamma)$ with the (URP), let us show that the C*-algebra $\mathrm{C}(X)\rtimes\Gamma$ can be (weakly) tracially approximated by (not necessarily unital) homogeneous C*-algebras with dimension ratio almost dominated by the mean dimension of $(X, \Gamma)$.
\begin{thm}\label{T-SDG}
Let $(X, \Gamma)$ be a dynamical system with the (URP).
The C*-algebra $A:=\mathrm{C}(X) \rtimes \Gamma$ has the following property: For any finite set $\{f_1, f_2, ..., f_n\}\subseteq \mathrm{M}_{m}(A)$ (where $m\in\mathbb N$), $h\in\mathrm{C}(X)^+$ with $h(x)\geq \frac{3}{4}$, $x\in F$ for a closed set $F\subseteq X$, 
and any $\delta>0$, there exist a positive element $p\in \mathrm{C}(X)\subseteq A$ with $\norm{p}\leq 1$, a sub-C*-algebra $C\subseteq A$ with $C\cong \bigoplus_{s=1}^S\mathrm{M}_{K_s}(\mathrm{C}_0(Z_s))$ for some $K_1, ..., K_S\in\mathbb N$ and some locally compact Hausdorff spaces $Z_1, ..., Z_S$ together with compact subsets $[Z_s]\subseteq Z_s$, $s=1, ..., S$, $\{f'_1, f_2', ..., f_n'\}\subseteq \mathrm{M}_m(A)$, and $h'\in \mathrm{C}(X)^+$ such that if $p_m:=p\otimes 1_m \in \mathrm{M}_m(A)\cong A\otimes \mathrm{M}_m(\Comp)$, 
then
\begin{enumerate}
\item\label{R-cut-1} $\norm{h-h'}<\delta$, $\norm{f_i-f'_i} < \delta$, $i=1, 2, ..., n$;
\item\label{R-cut-2} $\norm{p_mf'_i-f_i'p_m} < \delta$, $i=1, 2, ..., n$;
\item\label{R-cut-3} $p\in C$, $ph'p\in C$, and $p_mf'_ip_m\in \mathrm{M}_{m}(C)$, $i=1, 2, ..., n$;
\item\label{R-cut-4} $\frac{\mathrm{dim}([Z_s])}{K_s} < \mathrm{mdim}(X, \Gamma) + \delta$, $s=1, ..., S$;
\item\label{R-cut-5} $\mathrm{\mu}(X\setminus p^{-1}(1)) <\delta $, $\mu\in \mathcal M_1(X, \Gamma)$;
\item\label{R-cut-6} under the isomorphism $C\cong \bigoplus_{s=1}^S\mathrm{M}_{K_s}(\mathrm{C}_0(Z_s))$, one has $$\mathrm{rank}((ph'p-\frac{1}{4})_+(z)) \geq K_s (\min_\mu\mu(F)-\delta),\quad z\in [Z_s],$$ where $\mu$ runs through $\mathcal M_1(X,  \Gamma)$;
\item\label{R-cut-7} under the isomorphism $C\cong \bigoplus_{s=1}^S\mathrm{M}_{K_s}(\mathrm{C}_0(Z_s))$, the element $p$ has the form $$p = \bigoplus_{s=1}^S\mathrm{diag}\{p_{s, 1}, ..., p_{s, K_s}\},$$ where $p_{s, i}: Z_s \to [0, 1]$, and
$$\frac{1}{K_s}\abs{\{1\leq i\leq K_s: p_{s, i}(z) = 1 \}} >1-\delta, \quad z\in [Z_s],\ s=1, ..., S.$$   In particular, one has $$\mathrm{rank}(p(z)) > K_s(1-\delta),\quad z\in [Z_s],\ s=1, ..., S,$$ and
$$\mathrm{Tr}(p(z)) > K_s(1-\delta),\quad z\in [Z_s],\ s=1, ..., S;$$
\item\label{R-cut-8} still under the isomorphism $C\cong \bigoplus_{s=1}^S\mathrm{M}_{K_s}(\mathrm{C}_0(Z_s))$, any diagonal element of $C$ actually is in $\mathrm{C}(X)$, and if $f\in C^+$ is an diagonal element satisfying $f|_{[Z_s]} = 1_{K_s}$, $s=1, ..., S$, then, as an element of $\mathrm{C}(X)$,
$$\mu(X\setminus f^{-1}(1)) < \delta,\quad \mu\in\mathcal M_1(X, \Gamma).$$
\end{enumerate} 
\end{thm}

Before proving Theorem \ref{T-SDG}, we have the following two lemmas on partition of unity, which are elementary and might be well known.
\begin{lem}\label{semi-unity}
Let $V_1, ..., V_n$ be open subsets of $X$, where $X$ is a separable locally compact Hausdorff space. Let $D$ be a compact subset of $\bigcup_{i=1}^n V_i$. Then there are continuous functions $\phi_i: X \to [0, 1]$, $i=1, ..., n$, such that 
\begin{enumerate}
\item $\mathrm{supp}(\phi_i) \subseteq V_i$, $i=1, ..., n$,
\item $\sum_{i=1}^n \phi_i(x) = 1$, $x\in D$.
\end{enumerate}
\end{lem}
\begin{proof}
Pick open sets $V_i'\subseteq V_i$, $i=1, ..., n$, such that $\overline{V'_i} \subseteq V_i$ and $D\subseteq \bigcup_{i=1}^n V_i'$. Pick continuous functions $g_i: X \to [0, 1]$ such that 
\begin{equation}\label{inter-function}
g_i|_{V_i'} = 1\quad \textrm{and}\quad g_i|_{V_i^c}=0,\quad i=1, ..., n.
\end{equation} 
Define
\begin{displaymath}
\begin{array}{l}
\phi_1=g_1,\\
\phi_2=(1-g_1)g_2, \\
\cdots \\
\phi_n = (1-g_1)(1-g_2)\cdots(1-g_{n-1})g_n.
\end{array}
\end{displaymath}
It is clear that $\mathrm{supp}(\phi_i) \subseteq V_i$, $i=1, ..., n$.
Note that
$$\phi_1+\cdots+\phi_n = 1-(1-g_1)(1-g_2)\cdots(1-g_n),$$ and hence by \eqref{inter-function}, 
$$\phi_1(x)+\cdots+\phi_n(x) = 1,\quad x\in \bigcup_{i=1}^n V_i' \supseteq D,$$
as desired.
\end{proof}

\begin{lem}\label{partition-ext}
Let $X$ be a separable compact Hausdorff space, and let $\mathcal V$ and $\mathcal W$ be two finite collections of open subsets such that $\mathcal V \cup \mathcal W$ forms a cover of $X$. Assume there are continuous functions $\phi_V: X \to [0, 1]$, $V\in \mathcal V$ such that
\begin{enumerate}
\item  $\mathrm{supp}(\phi_V) \subseteq V$, $V\in\mathcal V$,
\item $\sum_{V\in\mathcal V} \phi_V(x) \leq 1$, $x\in X$, and
\item $\mathrm{int}((\sum_{V\in\mathcal V} \phi_V)^{-1}(1))\cup \bigcup_{W\in \mathcal W} W =X$.
\end{enumerate}
Then, there are continuous functions $\phi_W$, $W\in\mathcal W$, such that 
$$\{ \phi_V, \phi_W,\quad V\in \mathcal V,\ W\in\mathcal W \}$$
form a partition of unity subordinate to $\mathcal V \cup \mathcal W$.
\end{lem}
\begin{proof}
Consider the function $$g:=\sum_{V\in\mathcal V} \phi_V,$$ and list $\mathcal W=\{W_1, W_2, ..., W_n\}$. Since $\mathrm{int}(g^{-1}(1))\cup \bigcup_{W\in \mathcal W} W =X$, there are open sets $W'_1$, $W'_2$, ..., $W'_n$ such that $\overline{W'_1}\subseteq W_1$, $\overline{W'_2}\subseteq W_2$, ..., $\overline{W'_2}\subseteq W_2$ and 
\begin{equation}\label{cover-X}
\mathrm{int}(g^{-1}(1)) \cup (W'_1\cup W'_2 \cup\cdots\cup W'_n)=X.
\end{equation} Pick continuous functions $g_1$, $g_2$, ..., $g_n$ such that $g_i|_{W'_i} = 1$, $g_i|_{W_i^c} = 0$, $i=1, 2, ..., n$.

Define
\begin{displaymath}
\begin{array}{l}
h_1=(1-g)g_1,\\
h_2=(1-g)(1-g_1)g_2, \\
\cdots \\
h_n = (1-g)(1-g_1)(1-g_2)\cdots(1-g_{n-1})g_n.
\end{array}
\end{displaymath}
It is clear that $\mathrm{supp}(h_i) \subseteq W_i$. Moreover, 
$$h_1+h_2+\cdots+h_n = (1-g)(1-(1-g_1)(1-g_2)\cdots(1-g_n)).$$
Then 
$$h_1(x)+h_2(x)+\cdots+h_n(x) =  1-g(x),\quad x\in \overline{W'_1} \cup \cdots \cup \overline{W'_n},$$ and hence
$$g(x) + h_1(x)+h_2(x)+\cdots+h_n(x) =  1,\quad x\in \overline{W'_1} \cup \cdots \cup \overline{W'_n},$$
Note that
$$h_1(x)+h_2(x)+\cdots+h_n(x) = (1-g)(1-(1-g_1)(1-g_2)\cdots(1-g_n)) = 0,\quad x \in \mathrm{int}(g^{-1}(1)),$$
and then by \eqref{cover-X},
$$g(x) + h_1(x)+h_2(x)+\cdots+h_n(x) =  1,\quad x\in X.$$
Then $$\{\phi_V: V \in \mathcal V\} \cup \{h_i: i=1, 2, ..., n\}$$ form a partition of unity subordinate to $\mathcal V \cup \mathcal W$, as desired.
\end{proof}

The following lemma particularly asserts that each Rokhlin tower corresponds to a matrix algebra over the C*-algebra of the base set.
\begin{lem}\label{tower-alg}
Let $A$ be a C*-algebra. Let $u_1, u_2, ..., u_n\in A$ be unitaries with $u_1=1$, and let $\mathcal G\subseteq A$ be a finite set. Assume that
\begin{equation}\label{level-rel}
(au_i)(u^*_jb) = 0,\quad i \neq j,\ a, b \in\mathcal G.
\end{equation}
Then there is an isomorphism $$\mathrm{C^*}\{u^*_ia: i=1, 2, ..., n,\ a\in\mathcal G\} \cong \mathrm{M}_{n}(\mathrm{C^*}(\mathcal G))$$ under which
$$u^*_1a_1u_1 + \cdots + u^*_n a_n u_n \mapsto  \mathrm{diag}\{a_1, ..., a_n\},\quad a_1, ..., a_n\in \mathcal G.$$
\end{lem}

\begin{proof}
Denote by $C=\mathrm{C^*}\{u^*_ia: i=1, 2, ..., n,\ a\in\mathcal G\}\subseteq A$,
and consider 
$$D:=\{u^*_1au_1 + u^*_2au_2+ \cdots + u^*_nau_n: a\in  \mathrm{C^*}(\mathcal G)\}.$$ Note that $D$ is isomorphic to $\mathrm{C^*}(\mathcal G)$ by \eqref{level-rel}.

Embed $A$ into the enveloping von Neumann algebra $A''$. Let $a$ be a strictly positive element of $\textrm{C*}(\mathcal G)$ with norm $1$, and let $p$ denote the w*-limit of $a^{\frac{1}{n}}$, $n=1, 2, ...$, in $A''$. Then $p$ is an (open) projection with $pa=a$, $a\in A$.

Consider $$v_i:=pu_i \in A'', \quad i=1, 2, ..., n.$$ Then, it follows from \eqref{level-rel} that $v_iv^*_j = \delta_{i, j}p$; that is, the elements $\{v_1, v_2, ..., v_n\}$ form a system matrix unit, and thus the C*-algebra generated by $\{v_1, v_2, ..., v_n\}$ is isomorphic to $\mathrm{M}_n(\Comp)$. 

For any $a\in \textrm{C*}(\mathcal G)$, by \eqref{level-rel}, one has
$$u^*_ip(u^*_1au_1 + \cdots + u^*_nau_n) = u^*_ipau_1^*=u^*_ia = u^*_iap=(u^*_1au_1 + \cdots + u^*_nau_n) u^*_ip,$$
$$(u^*_1pu_1 + \cdots + u^*_npu_n)(u^*_1au_1 + u^*_2au_2+ \cdots + u^*_nau_n) = u^*_1au_1+ \cdots u^*_nau_n,$$
and
$$(u^*_1au_1 + u^*_2au_2+ \cdots + u^*_nau_n) (u^*_1pu_1 + \cdots + u^*_npu_n)= u^*_1au_1+ \cdots u^*_nau_n.$$
That is, $D$ and $\{v_1, v_2, ..., v_n\}$ commute, and $u^*_1pu_1 + \cdots + u^*_npu_n$, which equals $v^*_1v_1+\cdots+v^*_nv_n$, acts on $D$ as the unit. Therefore the C*-algebra generated by $D$ and $\{v_1, v_2, ..., v_n\}$ is unital and is isomorphic to the tensor product $\tilde{D}\otimes\mathrm{M}_n(\Comp) \cong\mathrm{M}_n(\tilde{D})$, where $\tilde{D}$ is the unitization of $D$.

Since 
$$(pu_i^*) (u^*_1au_1 + u^*_2au_2+ \cdots + u^*_nau_n) = au_i,\quad a\in\mathcal G,\ i=1, 2, ..., n,$$ the C*-algebra  $C$ is a sub-C*-algebra of $\tilde{D}\otimes\mathrm{M}_n(\Comp)$ generated by products of $1_{\tilde{D}} \otimes \mathrm{M}_n(\Comp)$ and $D\otimes 1_{\mathrm{M}_n(\Comp)}$, which is exactly the sub-C*-algebra $D \otimes \mathrm{M}_n(\Comp)\cong\mathrm{M}_n(D)$.
\end{proof}

\begin{proof}[Proof of Theorem \ref{T-SDG}]
It is enough to show the theorem for $m=1$, i.e., $f_1, ..., f_n\in A$. For the case $m\neq 1$, write $f_i = (f^{(i)}_{j, k})$ where $f^{(i)}_{j, k} \in A$, and then apply the theorem with $\{f^{(i)}_{j, k}: 1\leq i\leq n, 1\leq j, k \leq m\}$ and $\frac{\delta}{m^2}$ in place of $\{f_1, f_2, ..., f_n\}$ and $\delta$ respectively, together with the given function $h$ and the given closed set $F$. 

Without loss of generality, one may assume 
$$f_i=\sum_{\gamma \in\mathcal N} f_{i, \gamma}u_\gamma$$
for some finite set $\mathcal N \subseteq \Gamma$ with $e\in \mathcal N = \mathcal N^{-1}$, and some $f_{i, \gamma} \in \mathrm{C}(X)$. Denote by 
$$M=\max\{1, \norm{f_{i, \gamma}}: i=1, ..., n, \gamma\in\mathcal N\}.$$

Since $X$ is compact, there is an open cover $\mathcal O$ such that
\begin{equation}\label{large-O}
\abs{f_{i, \gamma}(x) - f_{i, \gamma}(y)} < \frac{\delta}{2{\abs{\mathcal N}}},\quad x, y \in U\in \mathcal O,\ i=1, ..., n,\ \gamma \in \mathcal N,
\end{equation} 
and
$$\abs{h(x) - h(y)} < \frac{\delta}{2{\abs{\mathcal N}}},\quad x, y \in U\in \mathcal O. $$

Pick a natural number 
\begin{equation}\label{large-L}
L > \frac{8M{\abs{\mathcal N}}}{\delta},
\end{equation} 
and pick a finite set $K\subseteq\Gamma$ and $\eps>0$ so that if a finite set $\Gamma_0\subseteq\Gamma$ is $(K, \eps)$-invariant, then 
 $$\frac{1}{\abs{\Gamma_0}} D( \bigvee_{\gamma\in\Gamma_0}\mathcal O \gamma) < \mathrm{mdim}(X, \Gamma) + \delta.$$ 
 
For any $x\in X$, denote by  $\delta_x$ the Dirac measure concentrated at $x$; and for any finite set $\Gamma_0\subseteq\Gamma$, denote by $\delta_{x, \Gamma_0} = \frac{1}{\abs{\Gamma_0}}\sum_{\gamma\in\Gamma_0}\delta_{x\gamma}$.

 One may then again assume $K$ is sufficiently large and $\eps$ is sufficiently small so that if a finite set $\Gamma_0\subseteq\Gamma$ is $(K, \eps)$-invariant, then
 \begin{equation}\label{average-lbd}
\delta_{x, \Gamma_0}(F) \geq  \min\{\mu(F): \mu\in \mathcal M_1(X, \Gamma)\}-\frac{\delta}{2},\quad x\in X,
\end{equation}
and
\begin{equation}\label{very-small-bd}
\frac{\abs{\Gamma_0 \setminus \mathrm{int}_{\mathcal N^{L+1}} (\Gamma_0)} }{\abs{\Gamma_0}} < \frac{\delta}{2}.
\end{equation}


%

Since $(X, \Gamma)$ has the (URP), there exist closed sets $B_1, B_2, ..., B_S \subset X$ and $(K, \eps)$-invariant sets $\Gamma_1, \Gamma_2, ..., \Gamma_S \subseteq \Gamma$ such that
$$B_s\gamma,\quad \gamma\in \Gamma_s,\ s=1, ..., S, $$
are mutually disjoint and
\begin{equation}\label{small-lo}
\mathrm{ocap}(X\setminus\bigsqcup_{s=1}^S\bigsqcup_{\gamma\in \Gamma_s} B_s\gamma) < \frac{\delta}{2}.
\end{equation}

For each $\Gamma_s$, $s=1, 2, ..., S$,  pick $\mathcal U_s\preceq  \bigvee_{\gamma\in\Gamma_s}\mathcal O \gamma$ such that
$$\mathrm{ord}(\mathcal U_s) =  D( \bigvee_{\gamma\in\Gamma_s}\mathcal O \gamma).$$


For each $B_s$, $s=1, 2, ..., S$, choose open sets $U^s_3 \supseteq U^s_2 \supseteq U^s_1 \supseteq B_s$ such that $U^s_3 \supseteq \overline{U^s_2}$, $U^s_2 \supseteq \overline{U^s_1}$, and
$$U^s_3 \gamma,\quad \gamma \in \Gamma_s, \ s=1, 2, ..., S,$$
are disjoint.

Consider $$\mathcal U_s \cap U^s_2:=\{U\cap U^s_2: U\in\mathcal U_s\}.$$ 
It is an open cover of $B_s$ with order at most $\mathrm{ord}(\mathcal U_s)$. Also consider $$\mathcal U_s \cap (U^s_3\setminus \overline{U^s_1}):=  \{U \cap (U^s_3\setminus \overline{U^s_1}): U\in\mathcal U_s\}.$$ Then their union
$$(\mathcal U \cap U^s_2) \cup (\mathcal U \cap (U^s_3\setminus \overline{U^s_1}) )= \{U\cap U^s_2, U \cap (U^s_3\setminus \overline{U^s_1}): U\in\mathcal U_s\}$$ forms an open cover of $U^s_3$, and denote this open cover by $\mathcal V_s$.

Pick an open cover $\mathcal W$ of the closed set $X\setminus\bigsqcup_{s=1}^S\bigsqcup_{\gamma\in \Gamma_s}U^s_3\gamma$ such that
\begin{equation}\label{large-W}
\abs{f_{i, \gamma}(x)-f_{i, \gamma}(y)} < \frac{\delta}{2{\abs{\mathcal N}}}, \quad x, y \in W\in \mathcal W,\ i=1, ..., n, \gamma\in\mathcal N,
\end{equation}
$$\abs{h(x)-h(y)} < \frac{\delta}{2{\abs{\mathcal N}}}, \quad x, y \in W\in \mathcal W,$$
and
$$W \cap U^s_2\gamma = \varnothing,\quad W\in \mathcal W,\ \gamma\in\Gamma_s,\ s=1, 2, ..., S.$$ Then
$$ \mathcal X = \mathcal W \cup \{V\gamma: V\in\mathcal V_s,\ \gamma\in\Gamma_s,\ s=1, ..., S,\} $$
forms an open cover of $X$.

Pick an open set ${U^{s}_3}'\subseteq U^s_3$ such that $\overline{{U^s_3}'} \subseteq U_3$ and $$(\bigcup_{s=1}^S\bigcup_{\gamma\in\Gamma_s}{U^s_3}'\gamma)\cup(\bigcup_{W\in\mathcal W} W) = X.$$
By Lemma \ref{semi-unity}, there are functions $\phi_V: X\to [0, 1]$, $V\in\mathcal V_s$, satisfying
\begin{enumerate}
\item $\mathrm{supp}(\phi_V)\subseteq V$, $V \in \mathcal V_s$,
\item $\sum_{V\in\mathcal V_s} \phi_V(x) = 1$, $x\in {U^s_3}'$.
\end{enumerate}
Translate these functions to $\mathcal V_s\gamma$, $\gamma\in\Gamma_s$, 
and by Lemma \ref{partition-ext}, there is a partition of unity subordinate to $\mathcal X=\mathcal W \cup \{V\gamma: V\in\mathcal V_s,\ \gamma\in\Gamma_s,\ s=1, ..., S,\}$
in the form of
$$\{\psi_W: W \in\mathcal W\} \cup \{\gamma(\phi_V): V\in \mathcal V_s, \gamma\in \Gamma_s,\ s=1, ..., S \}.$$

Consider the sub-C*-algebra
\begin{equation}\label{defn-C}
C:=\mathrm{C}^*\{ u^*_\gamma \phi_V: V\in\mathcal V_s, \gamma\in \Gamma_s, s=1, 2, ..., S\}\subseteq \mathrm{C}(X) \rtimes \Gamma.
\end{equation} 
A straightforward calculation shows that if $\gamma_1\neq\gamma_2$, then, for any $V_1, V_2\in\mathcal V_{s}$, 
$$ (\phi_{V_1}u_{\gamma_1})(u_{\gamma_2}^*\phi_{V_2})  = u_{\gamma_1}((\phi_{V_1}\circ \gamma_1^{-1})(\phi_{V_2}\circ \gamma_2^{-1})) u_{\gamma_2}^*  = 0.$$
Hence, by Lemma \ref{tower-alg}, $C$ is isomorphic to $$\bigoplus_{s=1}^S\mathrm{M}_{\abs{\Gamma_s}}(\textrm{C*}\{\phi_V: V \in\mathcal V_s\})$$
and
the set of diagonal elements of $C$ under this isomorphism consists of
\begin{equation}\label{diag-form}
\sum_{s=1}^S\sum_{\gamma\in\Gamma_s, V\in\mathcal V_s} u^*_\gamma\phi_Vu_\gamma = \sum_{s=1}^S\sum_{\gamma\in\Gamma_s, V\in\mathcal V_s} \phi_V\circ \gamma^{-1} \in\mathrm{C}(X).
\end{equation}


For each $s=1, 2, ..., S$, consider the set
$$Z_s:=(\prod_{V\in \mathcal V_s} \phi_V)(X) \subseteq \Real^{\abs{\mathcal V_s}} $$ and
\begin{equation}\label{df-Z}
[Z_s]:= (\prod_{V\in \mathcal V_s} \phi_V)(B_s )\subseteq Z_s.
\end{equation}
Note that $$\mathrm{dim}([Z_s]) \leq \mathrm{ord}(\mathcal U_s) = D( \bigvee_{\gamma\in\Gamma_s}\mathcal O \gamma),$$
and by Lemma 4.3 of \cite{EN-MD0}, 
$$\mathrm{C}^*\{ \phi_V: V\in\mathcal V_s\} \cong \mathrm{C}_0(Z_s).$$
In particular,
$$\frac{\mathrm{dim}([Z_s])}{\abs{\Gamma_s}} \leq \frac{1}{\abs{\Gamma_s}} D( \bigvee_{\gamma\in\Gamma_s}\mathcal O \gamma)\leq \mathrm{mdim}(X, \Gamma) + \delta,$$
and this verifies Property \eqref{R-cut-4}.

Now, let $f\in C^+$ be a diagonal element satisfying $f|_{[Z_s]} = 1_{K_s}$, $s=1, ..., S$, with respect to the isomorphism above. Then, by \eqref{diag-form} and \eqref{df-Z}, one has that, as an element of $\mathrm{C}(X)$, $$f|_{B_s\gamma} =1,\quad s=1, ..., S,\ \gamma\in\Gamma_s,$$
and by \eqref{small-lo}, 
$$\mu(X\setminus f^{-1}(1)) \leq \mu(X\setminus \bigsqcup_{s=1}^S\bigsqcup_{\gamma\in\Gamma_s} B_s\gamma)  < \delta,\quad \mu\in\mathcal M_1(X, \Gamma).$$
This, together with \eqref{diag-form}, verifies Property \eqref{R-cut-8}.

For each $s=1, 2, ..., S$, define 
$$\chi_s = \sum_{V\in\mathcal U_s\cap U^s_2} \phi_V\in\mathrm{C}(X) \cap C.$$ It is clear that $0\leq \chi_s(x) \leq 1$, $x\in X$,
\begin{equation}\label{chi-cond}
\left\{
\begin{array}{ll}
\chi_s(x) = 1, & x\in B_s; \\
\chi_s(x) = 0, & x\notin U^s_2.
\end{array}
\right.
\end{equation}

For each $\Gamma_s$, $s=1, 2, ..., S$, define the subsets (see Definition \ref{defn-amenable} for the notation $\mathrm{int}_K(F)$)
$$
\left\{
\begin{array}{lll}
\Gamma_{s, L+1} & = & \mathrm{int}_{\mathcal N^{L+1}} (\Gamma_s), \\
\Gamma_{s, L} & = & \mathrm{int}_{\mathcal N^{L}} (\Gamma_s) \setminus \mathrm{int}_{\mathcal N^{L+1}} (\Gamma_s), \\
\Gamma_{s, L-1} & = & \mathrm{int}_{\mathcal N^{L-1}} (\Gamma_s) \setminus \mathrm{int}_{\mathcal N^{L}} (\Gamma_s), \\
\vdots & \vdots & \vdots \\
\Gamma_{s, 0} & = & \Gamma_s \setminus \mathrm{int}_{\mathcal N} (\Gamma_s).
\end{array}
\right.
$$
%
Then, for any $\gamma \in\mathcal N$, one has 
\begin{equation}\label{action-nbhd}
\Gamma_{s, l} \gamma \subseteq  \Gamma_{s, l-1} \cup \Gamma_{s, l} \cup \Gamma_{s, l+1},\quad 1\leq l \leq L.  \end{equation}
Indeed,  pick an arbitrary $\gamma'\in \Gamma_{s, l}$. By the construction, one has 
\begin{equation}\label{eq-contain} 
\gamma' \mathcal N^{l} \subseteq \Gamma_s\quad\mathrm{but}\quad \gamma' \mathcal N^{l+1} \nsubseteq \Gamma_s.
\end{equation} 
Therefore $$ \gamma' \gamma \mathcal N^{l-1}  \subseteq \gamma' \mathcal N^l \subseteq \Gamma_s$$ and hence $\gamma'\gamma \in \mathrm{int}_{\mathcal N^{l-1}}\Gamma_s$ (since $e\in\mathcal N^{l-1}$). 

Thus, to show \eqref{action-nbhd}, one only has to show that $\gamma'\gamma \notin \mathrm{int}_{\mathcal N^{l+2}}\Gamma_s$. Suppose $\gamma'\gamma \mathcal N^{l+2} \subseteq \Gamma_s$. Since $\mathcal N$ is symmetric, one has $\gamma^{-1}\in \mathcal N$; hence $\mathcal N^{l+1} \subseteq \gamma \mathcal N^{l+2}$ and 
$$\gamma' \mathcal N^{l+1} \subseteq \gamma'\gamma \mathcal N^{l+2} \subseteq \Gamma_s,$$ which contradicts \eqref{eq-contain}.

Also note that for each $\gamma\in\mathcal N$,
\begin{equation}\label{action-nbhd-L-end}
\Gamma_{s, L+1}\gamma \subseteq  \Gamma_{s, L+1} \cup \Gamma_{s, L}
\end{equation}

For each $\gamma \in \Gamma_s$, define 
$$\ell(\gamma) = l,\quad\textrm{if $\gamma \in \Gamma_{s, l}$}.$$ By \eqref{action-nbhd} and \eqref{action-nbhd-L-end}, the function $\ell$ satisfies 
\begin{equation}\label{action-nbhd-fn}
\abs{\ell(\gamma\gamma') - \ell(\gamma)} \leq 1,\quad \gamma'\in \mathcal N,\ \gamma\in \Gamma_{s, 1}\cup\cdots\cup\Gamma_{s, L+1}.
\end{equation}

Define
\begin{eqnarray*}
p & = & \sum_{s=1}^S\sum_{\gamma\in \Gamma_{s}\setminus\Gamma_{s, 0}} \frac{\ell(\gamma)-1}{L}\chi_s \circ \gamma^{-1} \\
& =  &\sum_{s=1}^S\sum_{l=1}^{L+1} \sum_{\gamma\in \Gamma_{s, l}} \frac{l-1}{L} \chi_s \circ \gamma^{-1}  \\
& = & \sum_{s=1}^S \sum_{l=1}^{L+1} \sum_{\gamma\in \Gamma_{s, l}} \frac{l-1}{L} u^*_\gamma\chi_s u_\gamma \in\mathrm{C}(X)\cap C.
\end{eqnarray*}
Then, by \eqref{chi-cond} (and \eqref{very-small-bd}), 
\begin{eqnarray*}
&& \mathrm{ocap}(X \setminus p^{-1}(1)) \\
& \leq & \max\{
\frac{\abs{\Gamma_s \setminus \mathrm{int}_{\mathcal N^{L+1}} (\Gamma_s)} }{\abs{\Gamma_s}} : s=1, ..., S\}   + \mathrm{ocap}(X\setminus\bigsqcup_{s=1}^S\bigsqcup_{\gamma\in \Gamma_s} B_s\gamma)\\
&  \leq & \frac{\delta}{2} + \frac{\delta}{2} < \delta,
\end{eqnarray*}
and this proves Property \eqref{R-cut-5}.

It follows from Lemma \ref{tower-alg} that under the isomorphism $C\cong \bigoplus_{s=1}^S\mathrm{M}_{\abs{\Gamma_s}}(\textrm{C*}\{\phi_V: V \in\mathcal V_s\}),$ one has $$p=\bigoplus_{s=1}^S\mathrm{diag}\{0, \frac{\ell(\gamma)-1}{L}\chi_s\circ \gamma^{-1}: \gamma\in \Gamma_{s}\setminus\Gamma_{s, 0} \}.$$ Denote the diagonal functions by $p_{s, i}$, $s=1, ..., S$, $i=1, ..., K_s$.

Note that $p|_{B_s\gamma} = 1$, $\gamma \in \Gamma_{s, L+1}$. Therefore, for any $z = \prod_{V\in\mathcal V_s}\phi_V(x)\in [Z_s]$, $s=1, ..., S$, where $x\in B_s$, one has
\begin{eqnarray*}
\frac{1}{K_s}\abs{\{1\leq i\leq K_s: p_{s, i}(z) = 1 \}} & = & \frac{1}{\abs{\Gamma_s}}\abs{\{\gamma\in\Gamma_s: p(x\gamma) = 1\}} \\
& \geq & 1- \max\{\frac{\abs{\Gamma_s \setminus \mathrm{int}_{\mathcal N^{L+1}} (\Gamma_s)} }{\abs{\Gamma_s}} : s=1, ..., S\} \\
& > & 1-\delta.
\end{eqnarray*}
This verifies Property \eqref{R-cut-7}.

%

Note that, by the construction of $C$ (see \eqref{defn-C}), $$\chi_s^{\frac{1}{2}} u_\gamma = (\sum_{V\in \mathcal U_s\cap U^s_2}\phi_V)^{\frac{1}{2}} u_\gamma \in C,\quad \gamma \in\Gamma_s.$$
Also note that, for each $\gamma'\in \mathcal N$ and each $\gamma\in \Gamma_{s, l}$, $l=1, 2, ..., L+1$, one has that $\gamma\gamma' \in \Gamma_s$. Hence
$$ pu_{\gamma'} = \sum_{s=1}^S \sum_{l=1}^{L+1} \sum_{\gamma\in \Gamma_{s, l}} \frac{l-1}{L} u^*_\gamma\chi_s u_{\gamma\gamma'}  = \sum_{s=1}^S \sum_{l=1}^{L+1} \sum_{\gamma\in \Gamma_{s, l}} \frac{l-1}{L} (u^*_\gamma\chi^{\frac{1}{2}}_s) (\chi^{\frac{1}{2}}_s u_{\gamma\gamma'})  \in C,$$
and therefore, 
\begin{equation}\label{pupinC}
pu_\gamma p\in C,\quad \gamma \in \mathcal N.
\end{equation}
Also note that, for each $\gamma' \in\mathcal N$, by \eqref{action-nbhd-fn} and \eqref{large-L},
\begin{eqnarray*}
\norm{u^*_{\gamma'}pu_{\gamma'} - p}  & = & \norm{ \sum_{s=1}^S\sum_{l=1}^{L+1} \sum_{\gamma\in \Gamma_{s, l}} \frac{l-1}{L}\chi_s\circ (\gamma\gamma')^{-1}  -  \sum_{s=1}^S\sum_{l=1}^{L+1} \sum_{\gamma\in \Gamma_{s, l}} \frac{l-1}{L} \chi_s\circ \gamma^{-1} } \\
& = & \max\{ \abs{ \frac{\ell(\gamma\gamma')-1}{L} - \frac{\ell(\gamma)-1}{L}}: \gamma\in\Gamma_s\setminus\Gamma_{s, 0},\ s=1, 2, ..., S \}\\
& < & \frac{1}{L}  <  \frac{\delta}{8M\abs{\mathcal N}},
\end{eqnarray*}
and hence
\begin{equation}\label{pre-comm-p}
\norm{pf_i - f_i p} < \frac{\delta}{8},\quad i=1, 2, ..., n.
\end{equation}

For each $U\in \mathcal X$, pick a point $x_U \in U$. For each $f_i$, define
$$f'_i= \sum_{\gamma\in\mathcal N} (\sum_{U\in\mathcal X} f_{i, \gamma}(x_U)\phi_U) u_\gamma, $$
and
$$h'= \sum_{U\in\mathcal X} h(x_U)\phi_U. $$
Then, by \eqref{large-O} and \eqref{large-W},
\begin{eqnarray}
\norm{f'_i - f_i} & = & \norm{ \sum_{\gamma \in \mathcal N} ( f_{i, \gamma} - (\sum_{U\in\mathcal X} f_{i, \gamma}(x_U)\phi_U) ) u_\gamma} \label{pert-p} \\
& < & \abs{\mathcal N}(\frac{\delta}{2\abs{\mathcal N}}) = \frac{\delta}{2}<\delta, \nonumber
\end{eqnarray}
and the same argument shows that $$\norm{h - h'} < \delta.$$ This verified Properties \eqref{R-cut-1}.

Moreover, note that for any $U\in\mathcal X$, either (in the case that $U\in\mathcal W$) $$U \cap (\bigcup_{s=1}^S\bigcup_{\gamma\in\Gamma_s}{U_2^s}\gamma) = \varnothing$$
or
$U$ is contained inside some $U_3^s\gamma$. Since $$\phi_W(x) = 0,\quad \textrm{if $x \in \bigcup_{s=1}^S\bigcup_{\gamma\in\Gamma_s}{U_3^s}'\gamma \supseteq \bigcup_{s=1}^S\bigcup_{\gamma\in\Gamma_s}{U_2^s}\gamma$} $$ for any $W\in \mathcal W$, and since 
$$p(x) = 0,\quad x\notin \bigcup_{s=1}^S\bigcup_{\gamma\in\Gamma_s}U_2^s\gamma,$$
one has
\begin{eqnarray*}
pf'_ip & =  & p  \sum_{\gamma'\in\mathcal N} (\sum_{U\in\mathcal X} f_{i, \gamma'}(x_U)\phi_U) u_{\gamma'}  p  =   \sum_{\gamma'\in\mathcal N} (\sum_{U\in\mathcal X} f_{i, \gamma'}(x_U)\phi_U) p u_{\gamma'}  p \\
& = &  \sum_{\gamma'\in\mathcal N} \sum_{s=1}^S\sum_{\gamma\in\Gamma_s}\sum_{U\subseteq U^s_3\gamma, U \in \mathcal X } f_{i, \gamma'}(x_U) \phi_U p u_{\gamma'}  p \\
& = &  \sum_{\gamma'\in\mathcal N} \sum_{s=1}^S\sum_{\gamma\in\Gamma_s}\sum_{U \in \mathcal V_s \gamma } f_{i, \gamma'}(x_U) \phi_U p u_{\gamma'}  p \\
& = &  \sum_{\gamma'\in\mathcal N} \sum_{s=1}^S\sum_{\gamma\in\Gamma_s}\sum_{U \in \mathcal V_s } f_{i, \gamma'}(x_{U\gamma}) (\phi_U\circ \gamma^{-1}) p u_{\gamma'}  p \in C,
\end{eqnarray*}
since $\phi_U\circ \gamma^{-1} = u^*_\gamma \phi_U u_\gamma \in C$, $\gamma\in\Gamma_s$, $U\in\mathcal V_s$ and $p u_{\gamma'} p\in C$, $\gamma'\in\mathcal N$ (see \eqref{pupinC}). The same argument shows that $ph'p\in C$.
This verifies Properties \eqref{R-cut-3}. By \eqref{pre-comm-p} and \eqref{pert-p}, Property \eqref{R-cut-2} follows.

For Property \eqref{R-cut-6}, pick an arbitrary $x\in B_s$. By \eqref{average-lbd}, 
\begin{eqnarray}\label{average-lbd-2}
 \frac{1}{\abs{\Gamma_s}}\abs{\{\gamma \in \Gamma_s: h(x\gamma) \geq \frac{3}{4} \}} &\geq & \frac{1}{\abs{\Gamma_s}}\abs{\{\gamma \in \Gamma_s: x\gamma \in F \}}  \nonumber \\
 & \geq &  \min\{\mu(F): \mu\in \mathcal M_1(X, \Gamma)\}-\frac{\delta}{2}.
 \end{eqnarray}
Since $\norm{h-h'} < \delta<\frac{1}{2}$, and $p|_{B_s\gamma} = 1$, $\gamma \in \Gamma_{s, L+1}$,
one has that 
\begin{equation*}
(ph'p)(x\gamma) > \frac{1}{4},\quad \textrm{if $h(x\gamma) \geq \frac{3}{4}$, $\gamma\in \Gamma_{s, L+1}$};
\end{equation*}
and hence 
\begin{equation*}
(ph'p-\frac{1}{4})_+(x\gamma) > 0,\quad \textrm{if $h(x\gamma) \geq \frac{3}{4}$, $\gamma \in \Gamma_{s, L+1}$}.
\end{equation*}

Thus, regarding $(ph'p - \frac{1}{4})_+$ as an element of $C\cong \bigoplus_{s=1}^S\mathrm{M}_{\abs{\Gamma_s}}(\mathrm{C}_0(Z_s))$,  one has (by \eqref{average-lbd-2}) that, for any $z \in [Z_s]$ (with $z=\prod_{V\in\mathcal V_s} \phi_V(x)$ for some $x\in B_s$), 
\begin{eqnarray*}
\frac{1}{\abs{\Gamma_s}}\mathrm{rank}(ph'p - \frac{1}{4})_+(z) & = & \frac{1}{\abs{\Gamma_s}} \abs{\{\gamma\in \Gamma_s: (ph'p - \frac{1}{4})_+(x\gamma) > 0 \}} \\
&\geq & \frac{1}{\abs{\Gamma_s}} \abs{\{\gamma\in \Gamma_{s, L+1}: h(x\gamma) \geq \frac{3}{4} \}}\\
&\geq & \frac{1}{\abs{\Gamma_s}} \abs{\{\gamma\in \Gamma_{s}: h(x\gamma) \geq \frac{3}{4} \}} - \frac{\abs{\Gamma_s\setminus\Gamma_{s, L+1}}}{\abs{\Gamma_s}}\\
& \geq & \min\{\mu(F): \mu\in \mathcal M_1(X, \Gamma)\}-\delta, 
\end{eqnarray*}
as desired.
\end{proof}

\section{Comparison of open sets, comparison radius, and mean topological dimension}

\begin{defn}\label{definition-comp}
Consider a topological dynamical system $(X, \Gamma)$, where $X$ is a compact Hausdorff space and $\Gamma$ is a discrete amenable group. It is said to have $(\lambda, m)$-Cuntz-comparison on open sets, where $\lambda\in (0, 1]$ and $m\in \mathbb N$, if for any open sets $E, F\subseteq X$ with $$ \mu(E) < \lambda \mu(F),\quad \mu \in\mathcal M_1(X,  \Gamma),$$ 
then
$$\varphi_E \precsim \underbrace{\varphi_F\oplus\cdots \oplus \varphi_F}_m\quad\mathrm{in}\ \mathrm{C}(X)\rtimes\Gamma.$$

The dynamical system $(X, \Gamma)$ is said to have Cuntz comparison on Open Sets (COS) if it has $(\lambda, m)$-Cuntz-comparison on open sets for some $\lambda$ and $m$.
\end{defn}

Let $(X, \Gamma)$ be a minimal free dynamical system. Assume that $(X, \Gamma)$ has the (URP) and (COS). 
The main result of this section is that the comparison radius of $\mathrm{C}(X)\rtimes\Gamma$ is at most $\frac{1}{2}\mathrm{mdim}(X, \Gamma)$ (see Theorem \ref{main-thm}). 

First, one needs some preparations. Let $a\in A^+$, where $A$ is a C*-algebra, and let $\eps>0$. Recall that (Definition \ref{defn-p-cut})
$$(a - \eps)_+ = f(a) \in A,$$ where $f(t) = \max\{t-\eps, 0\}$. Note 
$$((a-\eps_1)_+-\eps_2)_+ = (a-\eps_1-\eps_2)_+,$$
and also note that, by Proposition 2.2 of \cite{RorUHF2}, if $\|a-b\| < \eps$, then $(a-\eps)_+ \precsim b.$ These facts are used throughout this section.

Let $(X, \Gamma)$ be a free dynamical system, and assume that $\mathbb E: \mathrm{C}(X)\rtimes \Gamma \to\mathrm{C}(X)$ is a conditional expectation, where $\mathrm{C}(X)\rtimes \Gamma$ is a crossed-product C*-algebra. Then 
\begin{equation}\label{f-cond}
\mathbb E(u_\gamma) = 0,\quad \gamma\in\Gamma\setminus\{e\}.
\end{equation} 
Indeed, let $\gamma\in\Gamma\setminus\{e\}$ and consider $\mathbb E(u_\gamma)\in\mathrm{C}(X)$. Note that for all $g\in\mathrm{C}(X)$, since $u_\gamma^*gu_\gamma \in\mathrm{C}(X)$, one has $$g\mathbb E(u_\gamma) = \mathbb{E}(gu_\gamma)  = \mathbb{E}(u_\gamma u_\gamma^*gu_\gamma) = \mathbb{E}(u_\gamma)(u_\gamma^*gu_\gamma).$$ Assume $\mathbb E(u_\gamma) \neq 0$. Then there is $x_0\in X$ such that $\mathbb E(u_\gamma)(x_0) \neq 0$. Since $(X, \Gamma)$ is free, one has $x_0 \neq x_0\gamma^{-1}$. Pick $g\in \mathrm{C}(X)$ such that $g(x_0) = 1$ and $g(x_0\gamma^{-1}) = 0$. Then
$$g(x_0)\mathbb E(u_\gamma)(x_0) = \mathbb E(u_\gamma)(x_0) \neq 0 = \mathbb E(u_\gamma)(x_0) g(x_0\gamma^{-1}) = \mathbb E(u_\gamma)(x_0) (u_\gamma^* gu_\gamma)(x_0),$$ which is a contradiction. 

Then, one has the following lemma, which is similar to Lemma 3.1 of \cite{EN-ERA-AF}.
\begin{lem}\label{small-domain}
Let $(X, \Gamma)$  be a free topological dynamical system, where $\Gamma$ is a discrete group. Let $ \mathrm{C}(X)\rtimes \Gamma$ be a crossed product C*-algebra with a faithful conditional expectation $\mathbb{E}:  \mathrm{C}(X)\rtimes \Gamma \to \mathrm{C}(X)$ 
(the faithful conditional expectation exists if $\Gamma$ is amenable, see, for instance, Proposition 4.1.9 of \cite{BO-book}). Let $a\in \mathrm{C}(X)\rtimes \Gamma$ be a nonzero positive element. Then, there is a positive nonzero element $h\in\mathrm{C}(X)$ such that $h \precsim a$.
\end{lem}

\begin{proof}
Since $\mathbb E$ is faithful, one may assume that $\norm{\mathbb{E}(a)} = 1$.
One asserts that for arbitrary $\eps>0$, there is $p\in\mathrm{C}(X)$ such that $\norm{p} = 1$ and 
$$pap \approx_\eps p\mathbb{E}(a)p \approx_\eps p^2.$$
Then $h := (p^2-2\eps)_+ \precsim pap \precsim a$ satisfies the statement of the lemma.

Let us prove the assertion. Without loss of generality, one may assume
$$a=\sum_{i=1}^n f_{g_i}u_{g_i},\quad f_{g_i}\in\mathrm{C}(X),\ g_i\in\Gamma.$$ Since $(X, \Gamma)$ is free, by \eqref{f-cond}, one has that $f_e=\mathbb{E}(a)$. Since $\norm{f_e} = 1$, there is a point $x\in X$ such that $f_e(x) = 1$; since the action is free, there is  $p\in\mathrm{C}(X)^+$ such that $\norm{p} = 1$, $pf_e\approx_\eps p$, and 
$$pu_{g_i}p=(p(p\circ\sigma_{g_i}))u_{g_i} = 0,\quad i=1, 2,..., n.$$
Then $p$ is the desired function.
\end{proof}

\begin{lem}\label{div-element}
Let $(X, \Gamma)$ be a dynamical system with the (URP), and assume that $\abs{\Gamma} =\infty$. 

Then, for any $r\in (0, +\infty)$ and any $\eps>0$, there is a positive element $h\in \mathrm{M}_\infty(\mathrm{C}(X))$ such that 
$$|\mathrm{d}_\tau(h) - r | < \eps,\quad \tau\in \mathrm{T}(\mathrm{C}(X)\rtimes\Gamma).$$ Moreover, $h$ can be chosen to have the form $ h= \mathrm{diag}\{h_0, \underbrace{1, ..., 1}_{\lfloor r \rfloor}\}$ for some $h_0\in\mathrm{C}(X)$ together with a closed set $F \subseteq X$ such that $h_0(x) = 1$, $x\in F$, and
$$\min_{\mu} \mu(F)  \geq \{r\}-\eps,$$ where $\mu$ runs through $\mathcal M_1(X, \Gamma)$, $\lfloor r \rfloor$ is the integer part of $r$, and $\{r\}$ is the decimal part of $r$.
\end{lem}

\begin{proof}
First note that since $\abs{\Gamma} = \infty$, it has the property that $\abs{\Gamma_n} \to \infty$ for any F{\o}lner sequence $\Gamma_n$, $n=1, 2, ...$ . Otherwise, if there is a F{\o}lner sequence $(\Gamma_n)$ and $M>0$ such that $\abs{\Gamma_n} < M$, $n=1, 2, ...$, then, for any finite $K\subseteq \Gamma$ with $\abs{K} > M$ and $e\in K$, since $\Gamma_n$ is $(K, \frac{1}{2M})$-invariant for sufficiently large $n$ and $\abs{\Gamma_n} < M$, one has that $\Gamma_n K = \Gamma_n$ for sufficiently large $n$. But $\abs{K} \leq \abs{\Gamma_n K} =\abs{\Gamma_n}<M$, which contradicts to the choice of $K$.

Also note that one only has to prove the lemma for a rational number $r=\frac{m}{n}$ with $m, n\in\mathbb N$ and $\frac{1}{n} < \eps$ (one does not require that $(m, n) = 1$). 

Assume $r\leq 1$. Since $(X, \Gamma)$ has the (URP), there exist closed sets $B_1, B_2, ..., B_S \subseteq X$ and $(K, \eps)$-invariant sets $\Gamma_1, \Gamma_2, ..., \Gamma_S \subseteq\Gamma$ such that
\begin{enumerate}
\item $B_s\gamma$, $\gamma\in\Gamma_s$, $s=1, 2, ..., S$ are disjoint, and
\item $\mathrm{ocap}(X\setminus \bigsqcup_{s=1}^S\bigsqcup_{\gamma\in\Gamma_s}B_s\gamma) <\eps.$
\end{enumerate}
Note that $\Gamma_1$, ...,  $\Gamma_S$ can be chosen so that 
\begin{equation}\label{large-card}
\abs{\Gamma_s} > n^2,\quad s=1, 2, ..., S.
\end{equation}
Pick open sets $U_s \supseteq B_s$, $s=1, 2, ..., S$, so that $$U_s\gamma,\quad \gamma\in\Gamma_s,\  s=1, 2, ..., S$$ are disjoint. By \eqref{large-card},  there are $\Gamma'_s \subseteq \Gamma_s$, $s=1, 2, ..., S$, such that 
\begin{equation}\label{div-set}
\frac{m}{n} -\eps \leq \frac{\abs{\Gamma_s'}}{\abs{\Gamma_s}} \leq \frac{m}{n}.
\end{equation}
Put $$U =\bigsqcup_{s=1}^S \bigsqcup_{\gamma\in \Gamma_s'} U_s\gamma.$$ Then, by \eqref{div-set}, for any $\mu \in \mathcal M_1(X, \Gamma)$,
$$\mu(U) = \sum_{s=1}^S \abs{\Gamma_s'}\mu(U_s) \leq \frac{m}{n} \sum_{s=1}^S \abs{\Gamma_s}\mu(U_s) \leq  \frac{m}{n}$$
and
$$ \mu(U) = \sum_{s=1}^S \abs{\Gamma_s'}\mu(U_s) \geq  (\frac{m}{n}-\eps) \sum_{s=1}^S \abs{\Gamma_s}\mu(U_s) \geq  (\frac{m}{n}-\eps)(1-\eps)  > \frac{m}{n} - 2\eps.$$ Then the function $ h:=\varphi_U$ satisfies the lemma (with $d=1$ and $F_1 = \bigsqcup_{s=1}^S \bigsqcup_{\gamma\in \Gamma_s'} B_s\gamma$).

For general $r$, pick $h_0$ to satisfy the lemma for $\{r\}$; and the element
$$\mathrm{diag}\{h_0, \underbrace{1, ..., 1}_{\lfloor r\rfloor}\}$$
satisfies the lemma.
\end{proof}

\begin{defn}
A positive element $a$ of a C*-algebra $A$ is said to be compact if $\mathrm{sp}(a)\cap(0, \delta) = \varnothing$ for some $\delta >0$.
\end{defn}

Note that if $a$ is a compact element, then $a$ is Cuntz equivalent to the spectral projection $\chi_{(\frac{\delta}{2}, +\infty)}(a)$ where $\delta$ satisfies $\mathrm{sp}(a)\cap(0, \delta) = \varnothing$.

\begin{lem}\label{small-decp}
Let $A$ be a finite C*-algebra, and let $a, b$ be nonzero positive elements of $A$ with  $a\precsim b$. Assume either
\begin{enumerate}
\item $b$ is not compact, or
\item $a$ is not compact, or,
\item $a$ is not Cuntz equivalent to $b$.
 \end{enumerate}
  Then, for any $\eps>0$, there are nonzero positive elements $b_1, b_2$ such that
\begin{enumerate}
\item $b_1 \perp b_2$,
\item $b_1, b_2\in  \overline{bAb}$, and
\item $(a-\eps)_+ \precsim b_2$.
\end{enumerate}
\end{lem}
\begin{proof}

Assume that $\mathrm{sp}(b)\cap(0, \delta) \neq \varnothing$ for all $\delta >0$ (that is, $b$ is not a compact element). Then, for the arbitrarily given $\eps$, there is $\delta$ such that
$(a-\eps)_+ \precsim (b-\delta)_+$. Then, $b_1:=f_\delta(b)$ and $b_2:=(b-\delta)_+$, where $f_\delta(t)=\max\{0, t(\delta-t)\}$, are desired elements.

Now, assume that $a$ is not compact, and one may also assume that $b$ is a compact element, and without loss of generality, one may assume that $b$ is a nonzero projection. For the arbitrarily given $\eps>0$, there is $a'\in bAb$ such that $(a-\frac{\eps}{2})_+$ is Cuntz equivalent to $a'$ (in particular, $(a-\frac{\eps}{2})_+ \precsim a'$). 

If $(a-\frac{\eps}{2})_+$ is not a compact element, then $a'$ is not a compact element neither. By the argument for the case that $b$ is not a compact element, there are $b_1, b_2$ such that 
\begin{enumerate}
\item $b_1 \perp b_2$,
\item $b_1, b_2\in \overline{a'Aa'}\subseteq \overline{bAb}$, and
\item $((a-\frac{\eps}{2})_+-\frac{\eps}{2})_+ \precsim b_2$.
\end{enumerate}
Since $(a-\eps)_+ \precsim ((a-\frac{\eps}{2})_+-\frac{\eps}{2})_+$, the elements $b_1, b_2$ satisfy the lemma.

If $(a-\frac{\eps}{2})_+$ is a compact element (so $(a-\frac{\eps}{2})_+$ and $a'$ can be assumed to be projections), since $a$ is assumed not to be a compact element, one has that $(a-\frac{\eps}{2})_+ \precsim b$ but not equivalent to $b$, and hence $a'\neq b$ (since $A$ is finite). Hence the elements $b_1:=b-a'$ ($b$ is assumed to be a projection) and $b_2:=a'$ satisfy the lemma. 

For the remaining case that $a$ and $b$ are non-equivalent compact elements, 
one may assume that $a, b$ are projections such that $a< b$ but $a\neq b$. Then $b_1:=b-a$ and $b_2:=a$ satisfy the lemma.
\end{proof}

\begin{lem}\label{AC-cut}
For any $\eps\in(0, 1)$ and any nonzero $x = (x_{i, j}) \in \mathrm{M}_L(A)$ (for some $L\in {\mathbb N}$), where $A$ is a C*-algebra, there is $\delta(\eps, \norm{x}, L)>0$, such that for any $a, b, p\in A^+$ with $\norm{a}, \norm{b}, \norm{p} \leq 1$ satisfying
\begin{enumerate}
\item\label{Cuntz-cond-1} $\norm{a\otimes 1_M-x^*(b\otimes 1_N)x} < \eps$, for some $M, N\in\mathbb N$, 
\item $\norm{px_{i, j}-x_{i, j}p} < \delta$, and
\item $p$, $pap$, $pbp$, $px_{i, j}p$ are in a sub-C*-algebra $C\subseteq A$,
\end{enumerate}
there is $c\in C$ such that
\begin{enumerate} 
\item  $ (pap-102\eps)_+\otimes 1_M  \precsim c\otimes 1_N$ in $C$, and 
\item $c \precsim p(b-\frac{\eps^3}{8\max\{1, \norm{x}^6\}})_+p$ in $A$.
\end{enumerate}
\end{lem}
\begin{proof}

Set $$\eps_0=\frac{\eps^3}{2\max\{1, \norm{x}^6\}}\quad\mathrm{and}\quad\eps_1=\frac{\eps}{4\norm{x}^2}.$$

For each $t_0\in (0, +\infty)$, define the positive function 
$$
h_{t_0}(t) = \left\{
\begin{array}{ll}
1, & t\geq t_0;\\
0, & t\leq \frac{t_0}{2}; \\
\mathrm{linear}, & \frac{t_0}{2} < t < t_0.
\end{array}
\right.
$$

Choose $\delta\in(0, \frac{\eps}{2L^2})$ sufficiently small such that
\begin{equation}\label{comm-cut}
\norm{(h^{\frac{1}{2}}_{\eps_1}(y)\otimes 1_L )x - x (h^{\frac{1}{2}}_{\eps_1}(y)\otimes 1_L)} < \eps
\end{equation} 
for any $y\in A^+$ with $\norm{x_{i, j}y-yx_{i, j}}<\delta$ and $\norm{y} \leq 1$. 
Note that the choice of $\delta$ only depends on $\eps$, $\norm{x}$, and $L$. Indeed, pick a polynomial
$P(t) = c_1t + \cdots + c_n t^n$ such that $$\norm{P(y) - h_{\eps_1}^{\frac{1}{2}}(y)} <\frac{\eps}{4},\quad \norm{y}\leq 1.$$
Set $$\delta = \frac{\eps}{2nL^2\max\{\abs{c_i}: i=1, ..., n\}\norm{x}}.$$ Then
$$\norm{(P(y)\otimes 1_L) x - x (P(y)\otimes 1_L)} <\frac{\eps}{2}$$
if $\norm{x_{i, j}y - y x_{i, j}} < \delta$, 
and hence
$$\norm{(h^{\frac{1}{2}}_{\frac{\eps}{4}}(y) \otimes 1_L ) x - x (h^{\frac{1}{2}}_{\frac{\eps}{4}}(y)\otimes 1_L)} \leq \norm{(P(y)\otimes 1_L) x - x (P(y)\otimes 1_L)} +\frac{\eps}{2}  <  \eps.$$
One asserts that this $\delta$ satisfies the conclusion of the lemma.

Let $a, b, p\in A^+$ satisfy the conditions of the lemma. Then, by \eqref{comm-cut}, with $$\bar{h}_{\eps_1}(t) := h^{\frac{1}{2}}_{\eps_1}(t)/t \quad \mathrm{and}\quad z:= ((p-\eps_1)_+ \bar{h}_{\eps_1}(p)\otimes 1_L)x(h^{\frac{1}{2}}_{\eps_1}(p)\otimes 1_L) $$ ($\bar{h}_{\eps_1}$ is a well-defined continuous function on $\Real$), one has 
\begin{eqnarray*}
& & (p\otimes 1_L)(x^*(b\otimes 1_N)x)(p\otimes 1_L) \\
 & \approx_{2L^2\delta} & x^*(p\otimes 1_L)(b\otimes 1_N)(p\otimes 1_L) x \\
 & \approx_{\eps} &  x^*((p-\eps_1)_+\otimes 1_L)((b-\eps_0)_+\otimes 1_N)((p-\eps_1)_+\otimes 1_L)x \\
&=& x^*(h_{\eps_1}(p)(p-\eps_1)_+\otimes 1_L)((b-\eps_0)_+\otimes 1_N)((p-\eps_1)_+h_{\eps_1}(p)\otimes 1_L)x \\
& \approx_{2\eps} & (h^{\frac{1}{2}}_{\eps_1}(p)\otimes 1_L)x^*(h^{\frac{1}{2}}_{\eps_1}(p) (p-\eps_1)_+\otimes 1_L)((b-\eps_0)_+\otimes 1_N)\\
&& ((p-\eps_1)_+h^{\frac{1}{2}}_{\eps_1}(p)\otimes 1_L)x(h^{\frac{1}{2}}_{\eps_1}(p)\otimes 1_L) \\
& = & (h^{\frac{1}{2}}_{\eps_1}(p)\otimes 1_L)x^*(\bar{h}_{\eps_1}(p) (p-\eps_1)_+\otimes 1_L) ((p(b-\eps_0)_+p)\otimes 1_N) \\
&& ((p-\eps_1)_+ \bar{h}_{\eps_1}(p)\otimes 1_L)x(h^{\frac{1}{2}}_{\eps_1}(p)\otimes 1_L) \\
& = & z^*(p(b-\eps_0)_+p\otimes 1_N) z.
\end{eqnarray*}
Hence, together with the assumption of the lemma,
\begin{equation}\label{chain-eq-C}
(pap\otimes 1_M) \approx_{\eps} (p\otimes 1_L)(x^*(b\otimes 1_N)x)(p\otimes 1_L) \approx_{5\eps}z^*(p(b-\eps_0)_+p\otimes 1_N)z.
\end{equation} 

Since $p\in C$ and $px_{i, j}p\in C$, $i, j=1, ..., L$, one has that $$ \bar{h}_{\eps_1}(p) x_{i, j} {h}^{\frac{1}{2}}_{\eps_1}(p) \in C$$ (the functions $ \bar{h}_{\eps_1}$ and ${h}^{\frac{1}{2}}_{\eps_1}$ take value $0$ at $0$), and hence $z\in \mathrm{M}_L(C)$, and $$z^*(pbp\otimes 1_N)z \in \mathrm{M}_L(C).$$


Since $$\norm{(p-\eps_1)_+\bar{h}_{\eps_1}(p) }\leq \frac{1}{\eps_1},$$ one has 
$$\norm{z} \leq \norm{(p-\eps_1)_+\bar{h}_{\eps_1}(p)} \norm{x} \norm{h_{\eps_1}^{\frac{1}{2}}(p)} \leq \frac{\norm{x}}{\eps_1},$$ 
and hence $$\norm{ z^*(p(b-\eps_0)_+p\otimes 1_N)z - z^*(pbp\otimes 1_N)z } <\eps_0\norm{z}^2 < \eps_0\frac{\norm{x}^2}{\eps^2_1} \leq 32 \eps.$$ 
By \eqref{chain-eq-C}, one has
$$\norm{pap\otimes 1_M - z^*(pbp\otimes 1_N)z } < 32\eps + 6\eps = 38\eps,$$ and hence
\begin{eqnarray*}
& & \norm{(pap - 16\eps)_+\otimes 1_M - z^*((pbp) - \eps_0)_+\otimes 1_N)z } \\ 
&\leq & \norm{pap\otimes 1_M - z^*(pbp\otimes 1_N)z } +16\eps+\eps_0\norm{z}^2 \\
&<& 38\eps +16\eps+32\eps = 86\eps.
\end{eqnarray*}
Therefore, 
\begin{eqnarray*}
(pap - 102 \eps)_+\otimes 1_M & \precsim & ((pap - 16\eps)_+ - 86\eps )_+ \otimes 1_M \\
& \precsim & z^*((pbp-\eps_0)_+\otimes 1_N)z \\
& \precsim & (pbp-\eps_0)_+\otimes 1_N,
\end{eqnarray*}
in the algebra $\mathrm{M}_\infty(C)$.
Since 
$$\| (pbp-\frac{\eps_0}{4})_+ - p(b-\frac{\eps_0}{4})_+p \| < \frac{\eps_0}{2},$$
one has
$$(pbp-\eps_0)_+ \precsim ((pbp-\frac{\eps_0}{4})_+-\frac{\eps_0}{2})_+ \precsim p(b-\frac{\eps_0}{4})_+p$$
in the algebra $A$. Then $c:=(pbp - \eps_0)_+\in C$ has the desired property.
\end{proof}

The following lemma is well know. A proof is included for reader's convenience.
\begin{lem}\label{Cu-DIV}
Let $A$ be a simple non-elementary C*-algebra, and let $n\in \mathbb N$ be arbitrary. Then there are nonzero positive elements $e_1, ..., e_n \in A$ which are mutually orthogonal and mutually Cuntz equivalent.
\end{lem}

\begin{proof}
Since $A$ is non-elementary, it contains a positive element with spectrum containing infinitely many points, and hence there are nonzero positive elements $e_1, e_2, ..., e_n\in A$ which are mutually orthogonal. 

Consider $e_1$ and $e_2$. Since $A$ is simple, for any $\eps>0$, there are $x_1, ..., x_n$, $y_1, ..., y_n$ for some $n\in\mathbb N$ such that
$x_1e_1y_1 + \cdots + x_ne_1y_n \approx_\eps e_2.$ In particular, with $\eps$ sufficiently small, there is $x\in A$ such that $z_1:=e_1xe_2 \neq 0$. Note that $$z_1z_1^* \in\overline{e_1Ae_1} \quad\mathrm{and}\quad  z_1^*z_1 \in\overline{e_2Ae_2}.$$ With the polar decomposition of $z_1$ and a functional calculus, one may assume that there is a positive element $e_1'$ with norm one such that $$e_1'(z_1z_1^*) = e_1'.$$

Replacing $e_1, e_2$ by $z_1z_1^*$ and $z_1^*z_1$ respectively, one obtains nonzero positive elements $$e_1, e_2, ..., e_n$$ which are mutually orthogonal and $e_1=z_1z_1^*,$ $e_2=z_1^*z_1$. Moreover, there is a positive element $e'_1$ with norm one such that $$e_1e_1'=e_1'.$$

Using the simplicity again, the same argument as above shows that there is $z_2\in A\setminus\{0\}$ such that
$$z_2z_2^* \in\overline{e'_1Ae'_1} \quad\mathrm{and}\quad  z_2^*z_2 \in\overline{e_3Ae_3},$$ and, moreover, there is a positive element $e_1''$ with norm one such that $$e_1''(z_2z_2^*) = e_1''.$$

Note that $$((z_2z_2^*)^{\frac{1}{2}} z_1)^* ((z_2z_2^*)^{\frac{1}{2}} z_1) \in\overline{e_2Ae_2}$$
and
$$((z_2z_2^*)^{\frac{1}{2}} z_1) ((z_2z_2^*)^{\frac{1}{2}}z_1)^* = (z_2z_2^*)^{\frac{1}{2}} z_1z_1^* (z_2z_2^*)^{\frac{1}{2}} =  (z_2z_2^*)^{\frac{1}{2}} e_1 (z_2z_2^*)^{\frac{1}{2}} = z_2z_2^*.$$

Therefore, replacing $e_1, e_1', e_2, e_3$, and $z_1$ by $z_2z_2^*, e_1'', ((z_2z_2^*)^{\frac{1}{2}} z_1)^* ((z_2z_2^*)^{\frac{1}{2}} z_1), z^*_2z_2$, and $(z_2z_2^*)^{\frac{1}{2}}z_1$ respectively, one obtains positive elements $$e_1, e_2, ..., e_n$$ which are mutually orthogonal and $e_1=z_1z_1^*,$ $e_2=z_1^*z_1$, and $e_1=z_2z_2^*,$ $e_3=z_2^*z_2$. Moreover, there is a positive element $e'_1$ with norm one such that $e_1e_1'=e_1'$. 

Repeating this argument finitely many times, one obtains nonzero positive elements $$e_1, e_2, ..., e_n$$ which are mutually orthogonal and elements $z_1, ..., z_{n-1}$ such that $$e_1=z_iz_i^*\quad\mathrm{and}\quad e_{i+1}=z_{i}^*z_{i},\quad i=1, ..., n-1. $$ Then the elements $e_1, e_2, ..., e_n$ clearly satisfy the lemma.
\end{proof}

\begin{thm}\label{main-thm} 
Let $(X, \Gamma)$ be a free and minimal topological dynamical system satisfying the (URP) and (COS). 

Let $a, b$ be non-zero positive elements of $\mathrm{M}_\infty(\mathrm{C}(X)\rtimes \Gamma)$ such that
$$\mathrm{d}_\tau(a) + r < \mathrm{d}_\tau(b),\quad \tau\in\mathrm{T}(\mathrm{C}(X)\rtimes \Gamma),$$
for some $r>\frac{1}{2}\mathrm{mdim}(X, \Gamma)$. Then $a\precsim b$. In other words, $$\mathrm{rc}(\mathrm{C}(X)\rtimes \Gamma) \leq \frac{1}{2}\mathrm{mdim}(X, \Gamma).$$
\end{thm}

\begin{proof}
One only has to prove the statement for $\abs{\Gamma} = \infty$. In the case that $\Gamma$ is a finite group, since the action is minimal, one has that the system $(X, \Gamma)$ is conjugate to the action of $\Gamma$ on itself by translation. Hence $(X, \Gamma)$ has mean topological dimension zero and $\mathrm{C}(X) \rtimes \Gamma \cong \mathrm{M}_{\abs{\Gamma}}(\Comp)$, which has zero radius of comparison. In particular, the statement of the theorem holds. 

Let us assume that $\abs{\Gamma} = \infty$. Moreover, without loss of generality, one may assume that $\norm{a} = \norm{b} = 1$.

Pick $r' \in (\frac{1}{2}\mathrm{mdim}(X, \Gamma), r)$ and some $\delta'>0$ such that $$\mathrm{d}_\tau(a) + r' + 3\delta' < \mathrm{d}_\tau(b), \quad \tau\in\mathrm{T}(\mathrm{C}(X)\rtimes\Gamma).$$

Since $\abs{\Gamma} = \infty$, by Lemma \ref{div-element} (with $r= \frac{1}{2}\mathrm{mdim}(X, \Gamma) + 2\delta'$ and $\eps = \delta'$), there is a positive element $h \in \mathrm{M}^{+}_\infty(\mathrm{C}(X))$ such that
$$\abs{\mathrm{d}_\tau(h) - (\frac{1}{2}\mathrm{mdim}(X, \Gamma) + 2\delta')} < \delta',\quad \tau\in\mathrm{T}(\mathrm{C}(X)\rtimes\Gamma),$$ 
and
\begin{equation}\label{approx-r}
\mathrm{d}_\tau(h) >\frac{1}{2}\mathrm{mdim}(X, \Gamma) + 2\delta' - \delta'  = \frac{1}{2}\mathrm{mdim}(X, \Gamma)+\delta',\quad \tau\in\mathrm{T}(\mathrm{C}(X)\rtimes \Gamma),
\end{equation}
Note that 
\begin{equation}\label{ele-fill}
\mathrm{d}_\tau(a) + \mathrm{d}_\tau(h) < \mathrm{d}_\tau(a) +  \frac{1}{2}\mathrm{mdim}(X, \Gamma) + 2\delta' +\delta' < \mathrm{d}_\tau(a) + r' + 3\delta'  < \mathrm{d}_\tau(b),
\end{equation}
for any $\tau\in\mathrm{T}(\mathrm{C}(X)\rtimes \Gamma)$.
Moreover, by Lemma \ref{div-element}, $h$ can be chosen such that $$h=\mathrm{diag}\{h_0, \underbrace{1, ..., 1}_d\}$$ for a positive function $h_0\in\mathrm{C}(X)$, some integer $d$, and $h_0(x) = 1$ on a closed set $F$ with 
\begin{equation}\label{const-sets}
\min_{\mu} \mu(F) + d \geq (\frac{1}{2}\mathrm{mdim}(X, \Gamma) + 2\delta') - \delta' = \frac{1}{2}\mathrm{mdim}(X, \Gamma) + \delta',
\end{equation} 
where $\mu$ runs through $\mathcal M_1(X, \Gamma)$.

Pick $m\in\mathbb N$ such that $a, b, h\in\mathrm{M}_m(\mathrm{C}(X)\rtimes\Gamma)$. 


Since $\mathrm{C}(X) \rtimes\Gamma$ is nuclear, all quasitraces are traces. By Theorem 4.3 of \cite{RorUHF2}, any lower semicontinuous dimension function of $\mathrm{C}(X) \rtimes\Gamma$ has the form $\mathrm{d}_\tau$ for some $\tau \in \mathrm{T}(\mathrm{C}(X)\rtimes \Gamma)$.

Since $(X, \Gamma)$ is minimal, the C*-algebra $\mathrm{C}(X) \rtimes\Gamma$ is simple. 
Hence the Cuntz class of $b$ is a strong order unit 
of the Cuntz semigroup of $\mathrm{C}(X) \rtimes\Gamma$ (see Proposition 4.2 of \cite{Cuntz-CuGr}). By \eqref{ele-fill} and the proof of Proposition 3.2 of \cite{RorUHF2}, there is $N \in \mathbb N$ such that
$$ (a\oplus h)\otimes 1_{N+1}  \precsim b\otimes 1_N.$$
Moreover, one may assume that
$(a\oplus h)\otimes 1_{N+1} $ is not Cuntz equivalent to $b\otimes 1_N$.

Let $\eps \in (0, \frac{1}{408})$ be arbitrary, by Lemma \ref{small-decp}, there are nonzero positive elements $b_1$ and $b_2$ such that 
\begin{enumerate}
\item $b_1\perp b_2$,
\item $b_1, b_2\in \overline{bAb}$, and
\item\label{comp-step-1}  $ ((a-\eps)_+\oplus (h-\eps)_+) \otimes 1_{N+1}  \precsim b_2 \otimes 1_N.$
\end{enumerate}
By \eqref{comp-step-1}, there is $\delta>0$ such that
\begin{equation}\label{comp-step-2}
((a-2\eps)_+\oplus (h-2\eps)_+) \otimes 1_{N+1} \precsim (b_2-\delta)_+\otimes 1_N.
\end{equation}

Since $(X, \Gamma)$ has the (COS), the C*-algebra $\mathrm{C}(X) \rtimes \Gamma$ has $(\lambda, m')$-Cuntz-comparison of open sets for some $\lambda \in\Real^+$ and $m'\in\mathbb N$. Fix $\lambda$ and $m'$.

Consider the nonzero positive element $b_1$. Since $\mathrm{C}(X)\rtimes\Gamma$ is simple and non-elementary, the hereditary C*-algebra generated by $b_1$ is also simple and non-elementary. By Lemma \ref{Cu-DIV}, there are mutually orthogonal nonzero positive elements $b_1^{(1)}, ..., b_1^{(m')}$ in the hereditary C*-algebra generated by $b_1$ such that $b_1^{(1)}, ..., b_1^{(m')}$  are mutually Cuntz equivalent.  By Lemma \ref{small-domain}, there exists a nonzero positive element $\tilde{b}_1\in\mathrm{C}(X)$ such that $\tilde{b}_1 \precsim b_1^{(1)}$, and hence there are nonzero positive elements $b_{1, 1}, b_{1, 2}\in \mathrm{C}(X)$  such that $$b_{1, 1} \perp b_{1, 2},$$ 
and
\begin{equation}\label{small-b} 
b_{1, 1} \otimes 1_{m'} + b_{1, 2} \otimes 1_{m'} = (b_{1, 1} + b_{1, 2}) \otimes 1_{m'} \precsim \tilde{b}_1 \otimes 1_{m'} \precsim b_1^{(1)} + \cdots + b_1^{(m')}  \precsim b_1.
\end{equation}

By \eqref{comp-step-2}, there is $x=(x_{i, j}) \in \mathrm{M}_{L}(\mathrm{M}_{m}(A))$ (for some $L\in\mathbb N$) such that
\begin{equation}\label{pre-Cuntz-comp}
\norm{ ((a-2\eps)_+\oplus (h-2\eps)_+) \otimes 1_{N+1} - x^*((b_2-\delta)_+ \otimes 1_N) x} < \frac{\eps}{16}.
\end{equation}

Let $\delta'' = \delta(\eps, \norm{x}, L)$ be the constant of Lemma \ref{AC-cut} with respect to $x$ and $\eps$. Moreover, $\delta''$ can be chosen so that
$$\delta'' < \min\{\frac{\eps}{16}, \frac{\eps^3}{8\max\{1, \norm{x}^6\}}\}$$
and if $\norm{pc-cp}<\delta''$ for some positive elements $p, c$ in some C*-algebra with $\norm{c}, \norm{p}\leq 1$, then 
\begin{equation}\label{pre-am-comm}
\norm{pcp - (1-p^2)^{\frac{1}{2}}c(1-p^2)^{\frac{1}{2}}} < \eps.
\end{equation}

Applying Theorem \ref{T-SDG} with $\{(a-2\eps)_+,\ (b_2 - \delta)_+,\ x_{i, j}: 1\leq i, j \leq L\}$ in place of $\{f_1, f_2, ..., f_n\}$, $(h_0-2\eps)_+$ in place of $h$ (note that the value of $(h_0-2\eps)_+$ at any point of $F$ is at least $\frac{3}{4}$), and $\min\{\frac{\delta'}{4(d+1)}, \delta'', \frac{1}{m}\lambda\{\inf_\nu\{\nu(b_{1, s}^{-1}(0, +\infty)): \nu\in\mathcal{M}_1(X, \Gamma)\}: s=1, 2\}$ in place of $\delta$, there are $$a',\ h' = \mathrm{diag}\{h_0', \underbrace{(1-2\eps)1_A, ..., (1-2\eps)1_A}_d\},\  b',\  x'=(x'_{i, j}),\ y',\  \textrm{and}\ p= p '\otimes 1_m$$ for some $p', h_0'\in\mathrm{C}(X)^+$ with $\norm{p'}\leq 1$, and a sub-C*-algebra $C\subseteq A$ and $A\cong\bigoplus_{s=1}^S\mathrm{M}_{K_s}(\mathrm{C}_0(Z_s))$ for locally compact metrizable spaces $Z_s$ together with a compact set $[Z_s] \subseteq Z_s$ such that
\begin{equation}\label{pert-N-1}
\norm{(a-2\eps)_+ - a'} <\frac{\eps}{16},\quad \norm{(h-2\eps)_+ - h'} <\frac{\eps}{16},\quad \norm{x-x'} < \delta'' < \frac{\eps}{4} 
\end{equation}
\begin{equation}\label{pert-N-2}
\norm{(b_2-\delta)_+ - b'} <\frac{\eps^3}{8\max\{1, \norm{x}^6\}} < \frac{\eps}{16},
\end{equation}
\begin{equation}\label{pert-am-comm}
\norm{px'_{i, j}-x_{i, j}'p}<\delta'',\quad 1 \leq i, j \leq L,
\end{equation}
\begin{equation}\label{pert-am-comm-1}
\norm{pa' - a'p} < \delta''
\end{equation}
\begin{equation}\label{cut-in-C}
p'\in C,\ p'h_0'p' \in C,\ px'_{i, j}p, pa'p, ph'p, pb'p \in \mathrm{M}_m(C),
\end{equation}
\begin{equation}\label{small-trace}
\mathrm{\mu}(X\setminus (p')^{-1}(1)) <\frac{1}{m}\lambda\{\inf_\nu\{\nu(b_{1, 2}^{-1}(0, +\infty))\},\quad  \mu\in \mathcal M_1(X,  \Gamma),
\end{equation}
and
\begin{equation}\label{approx-mdim}
 \frac{\mathrm{dim}([Z_s])}{K_s} < \mathrm{mdim}(X, \Gamma) + \delta',\quad s=1, 2, ..., S,
 \end{equation} 
 and, under the isomorphism $C\cong\bigoplus_{s=1}^S\mathrm{M}_{K_s}(\mathrm{C}_0(Z_s))$, 
\begin{equation}\label{const-sets-1}
\mathrm{rank}((p'h'_0p'-\frac{1}{4})_+(z))  \geq  K_s (\min_\mu \mu(F) -\frac{\delta'}{4}),\quad z\in[Z_s],
\end{equation}
and
\begin{equation}\label{lp-rank}
\mathrm{rank}(p(z)) \geq K_s(1-\frac{\delta'}{4(d+1)}), \quad z\in [Z_s].
\end{equation}
Moreover, still under the isomorphism $C\cong\bigoplus_{s=1}^S\mathrm{M}_{K_s}(\mathrm{C}_0(Z_s))$, if $f\in C^+$ is an diagonal element satisfying $f|_{Z_s} = 1_{K_s}$, $s=1, ..., S$, then $f\in \mathrm{C}(X)$, and 
\begin{equation}\label{large-supp}
\mu(X\setminus f^{-1}(1)) < \frac{1}{m}\lambda\{\inf_\nu\{\nu(b_{1, 1}^{-1}(0, +\infty))\},\quad \mu\in\mathcal M_1(X, \Gamma).
\end{equation}

Note that, by \eqref{const-sets-1}, \eqref{lp-rank} and \eqref{const-sets}, for any $z\in[Z_s]$,
\begin{eqnarray*}
\mathrm{rank}((ph'p-\frac{1}{4})_+(z)) & = & \mathrm{rank}((p'h'_0p'-\frac{1}{4})_+(z)) + d \cdot \mathrm{rank}(p(z)) \\
& \geq & K_s (\min_\mu \mu(F) -\frac{\delta'}{4}) + d K_s (1-\frac{\delta'}{4(d+1)}) \\
& \geq & K_s(\min_\mu \mu(F) + d -\frac{\delta'}{4}  -\frac{\delta'}{4}) \\
& \geq & K_s(\frac{1}{2}\mathrm{mdim}(X, \Gamma) + \delta' - \frac{\delta'}{2}).
\end{eqnarray*}
Then, since $(ph'p-102\eps)_+\succsim (ph'p-\frac{1}{4})_+$ ($\eps$ is assumed to be at most $\frac{1}{408}$), by \eqref{approx-mdim}, one has that for any $z\in[Z_s]$,
\begin{eqnarray}\label{rank-pert-lbd}
\mathrm{rank}(ph'p-102\eps)_+(z)&\geq& \mathrm{rank}((ph'p-\frac{1}{4})_+(z))\\
& \geq & K_s(\frac{1}{2}\mathrm{mdim}(X, \Gamma) + \delta' - \frac{\delta'}{2}) \nonumber \\
& = & \frac{K_s}{2} (\mathrm{mdim}(X, \Gamma) + \delta') \nonumber \\
& > & \frac{1}{2}\mathrm{dim}([Z_s])  > \frac{1}{2}(\mathrm{dim}([Z_s]) -1). \nonumber
\end{eqnarray}
Since $\mathrm{C}(X)\rtimes \Gamma$ has $(\lambda, m')$-comparison on open sets, it follows from \eqref{small-trace} that
\begin{equation}\label{large-p}
1-p^2\precsim b_{1, 2} \otimes 1_{m'}.
\end{equation} 

Note that, without loss of generality, one may assume that $\norm{x} = \norm{x'}$. By \eqref{pre-Cuntz-comp}, \eqref{pert-N-1}, and \eqref{pert-N-2}, 
$$\norm{(a'\oplus h')\otimes 1_{N+1} -(x')^*(b'\otimes 1_N) x'} < \eps;$$
together with \eqref{pert-am-comm} and \eqref{cut-in-C}, it follows from Lemma \ref{AC-cut} that there is $c\in \mathrm{M}_m(C)$
\begin{equation}\label{Multp-Comp-C}
((pa'p-102\eps)_+\oplus (ph'p-102\eps)_+ )\otimes 1_{N+1}  \precsim c\otimes 1_N 
\end{equation}
in $\mathrm{M}_m(C)$
and
\begin{equation}\label{c-pert-dom}
c \precsim p(b'-\frac{\eps^3}{8\max\{1, \norm{x}^6\}})_+p
\end{equation}
in $\mathrm{M}_m(A)$.
By \eqref{Multp-Comp-C},
$$\mathrm{d}_\tau((pa'p-102\eps)_+) +\mathrm{d}_\tau((ph'p-102\eps)_+) < \mathrm{d}_\tau(c),\quad\tau\in\mathrm{T}(C);$$
and hence
$$\mathrm{rank}((pa'p-102\eps)_+(z)) +\mathrm{rank}((ph'p-102\eps)_+(z)) < \mathrm{rank}(c(z)),\quad z\in Z_s, s=1, ..., S.$$
Therefore, by \eqref{rank-pert-lbd}, 
$$\mathrm{rank}((pa'p-102\eps)_+(z)) + \frac{\mathrm{dim}([Z_s])-1}{2}  < \mathrm{rank}(c(z)),\quad z\in [Z_s], s=1, ..., S.$$
and by Theorem 4.6 of \cite{Toms-Comp-DS},  $$(pa'p-102\eps)_+|_{[Z_s]} \precsim c|_{[Z_s]}\quad \textrm{in $\mathrm{M}_{m}(\mathrm{M}_K(\mathrm{C}([Z_s])))$}.$$ There is then a positive central element $g\in \mathrm{M}_{m}(C)$ with $g= g'\otimes 1_m $ for some $g'\in\mathrm{C}(X)^+$ (and $g'\in C$) such that $\norm{g'}=1$, $g'|_{[Z_s]} = 1_{K_s}$, and
\begin{equation}\label{pert-dom-c}
(g((pa'p-102\eps)_+)g - \eps )_+ \precsim c
\end{equation}
in $\mathrm{M}_m(C)$. Since $g'|_{[Z_s]} = 1_{K_s}$, $s=1, ..., S$,  by \eqref{large-supp} and $(\lambda, m')$-comparison, one has 
\begin{equation}\label{large-g}
1-g^2 \precsim b_{1, 1} \otimes 1_{m'}.
\end{equation} 

Note that
\begin{eqnarray*}
(pa'p-102\eps)_+ & = & g((pa'p-102\eps)_+)g +  (1-g^2)^{\frac{1}{2}}(pa'p-102\eps)_+(1-g^2)^{\frac{1}{2}} \\
& \approx_{\eps} & (g((pa'p-102\eps)_+)g -\eps)_+ + \\ &&  (1-g^2)^{\frac{1}{2}}(pa'p-102\eps)_+(1-g^2)^{\frac{1}{2}}.
\end{eqnarray*}
Therefore, together with \eqref{pert-am-comm} and \eqref{pre-am-comm},
\begin{eqnarray*}
a & \approx_{2\eps} & (a- 2\eps)_+ \\
& \approx_{\eps} & a' \\
&\approx_{\eps}& pa'p + (1-p^2)^{\frac{1}{2}}a'(1-p^2)^{\frac{1}{2}} \\
& \approx_{102\eps}& (pa'p - 102\eps)_+ + (1-p^2)^{\frac{1}{2}}a'(1-p^2)^{\frac{1}{2}} \\
& \approx_{\eps}& (g((pa'p-102\eps)_+)g -\eps)_+ +  (1-g^2)^{\frac{1}{2}}(pa'p-102\eps)_+(1-g^2)^{\frac{1}{2}} \\ && + (1-p^2)^{\frac{1}{2}}a'(1-p^2)^{\frac{1}{2}},
\end{eqnarray*}
and hence
\begin{eqnarray*}
(a - 107\eps)_+ & \precsim & (g((pa'p-102\eps)_+)g -\eps)_+ +  (1-g^2)^{\frac{1}{2}}(pa'p-102\eps)_+(1-g^2)^{\frac{1}{2}}\\ && + (1-p^2)^{\frac{1}{2}}a'(1-p^2)^{\frac{1}{2}}  \\
&\precsim & c \oplus (1-g^2) \oplus (1- p^2) \quad\quad\quad\quad\quad(\textrm{by \eqref{pert-dom-c}})\\
&\precsim & p(b'- \frac{\eps^3}{8\max\{1, \norm{x}^6\}} )_+p \oplus (1-g^2) \oplus (1- p^2) \quad\quad(\textrm{by \eqref{c-pert-dom}}) \\
&\precsim & (b_2-\delta)_+ \oplus  \bigoplus _{m'} b_{1, 1} \oplus \bigoplus_{m'} b_{1, 2}  \quad\quad(\textrm{by \eqref{pert-N-2}, \eqref{large-p}, and \eqref{large-g}})  \\
& \precsim & b_2 \oplus b_1 \precsim b \quad\quad\quad\quad\quad(\textrm{by \eqref{small-b}}).
\end{eqnarray*}
Since $\eps$ is arbitrarily small, one has that $a\precsim b$, as desired.
\end{proof}



\section{Recursive subhomogeneous C*-algebras with diagonal maps}\label{RSH-Z}

In the rest of paper, let us investigate Cuntz-comparison of Open Sets (COS) for a given dynamical system $(X, \Gamma)$. Let us start with a special class of concrete C*-algebras which will appear naturally as the C*-algebras of small subgroupoids of a transformation groupoid (see Section \ref{small-groupoid}). Any C*-algebra in this class enjoys a comparison property of diagonal elements (see Theorem \ref{DSH-comp}), and this eventually leads to the (COS) for $(X, \Gamma)$ if small subgroupoids with arbitrary large orbits exist (see Corollary \ref{orb-comp-measure}).

Recall that
\begin{defn}[\cite{Phill-RSA1}]
The class of recursive subhomogeneous C*-algebras (RSH algebras) is the smallest class $\mathcal R$ of C*-algebras which is closed under isomorphism and such that:
\begin{enumerate}
\item if $X$ is a compact Hausdorff space and $n\geq 1$, then $\mathrm{M}_n(\mathrm{C}(X))\in\mathcal R$,
\item $\mathcal R$ is closed under the following pull back construction: if $A\in\mathcal R$, $X$ is a compact Hausdorff space, $X^{(0)}\subseteq X$ is closed, $\varphi: A\to \mathrm{M}_n(\mathrm{C}(X^{(0)}))$ is any unital homomorphism, and $\rho: \mathrm{M}_n(\mathrm{C}(X)) \to \mathrm{M}_n(\mathrm{C}(X^{(0)}))$ is the restriction homomorphism, then the pullback $$A\oplus_{\mathrm{M}_n(\mathrm{C}(X^{(0)}))}  \mathrm{M}_n(\mathrm{C}(X))=\{(a, f)\in A\oplus \mathrm{M}_n(\mathrm{C}(X)):\ \varphi(a)=\rho(f)\}$$ is in $\mathcal{R}$.
\end{enumerate}
\end{defn}

From the definition, it is clear that any recursive subhomogeneous C*-algebra can be written in the form 
$$A\cong\left[\cdots\left[\left[C_0\oplus_{C_1^{(0)}}C_1\right]\oplus_{C_2^{(0)}}C_2\right]\cdots\right]\oplus_{C_l^{(0)}}C_l,$$ 
with $C_k=\mathrm{M}_{n(k)}(\mathrm{C}(X_k))$ for compact Hausdorff spaces $X_k$ and positive integers $n(k)$, with $C_k^{(0)}=\mathrm{M}_{n(k)}(\mathrm{C}(X_k^{(0)}))$ for compact subsets $X_k^{(0)}\subseteq X_k$ (possible empty), and where the maps $C_k\to C_k^{(0)}$ are always the restriction maps. An expression of this type will be referred to as a decomposition of $A$, and the notation used here will be referred to as the standard notation for a decomposition.

The RSH algebras considered in this paper will have the gluing maps $\varphi$ being of diagonal type:
\begin{defn}
Let $$A = \left[\cdots\left[\left[ C_0\oplus_{C_1^{(0)}} C_1 \right]\oplus_{C_2^{(0)}}C_2\right] \cdots\right]\oplus_{C_K^{(0)}} C_K$$ be an RSH algebra. A homomorphism $\phi: A \to \mathrm{M}_n(\mathrm{C}(Y))$, where $Y$ is a compact metrizable space, is said to be of diagonal type if there is a partition $$Y=Y_1\sqcup Y_2\sqcup \cdots\sqcup Y_n$$ such that for each $Y_i$, there are continuous maps $\lambda_1^{(i)}: Y_i \to X_{t_1}$, $ \lambda_2^{(i)}: Y_i \to X_{t_2}$, ..., $ \lambda_{s_i}^{(i)}: Y_i \to X_{t_{s_i}}$ for some $s_i\in\mathbb N$ such that
$$\phi((f_0, f_1, ..., f_K))|_{Y_i} = 
\left(
\begin{array}{cccc}
f_{t_1} \circ \lambda_1^{(i)} & & & \\
& f_{t_2} \circ \lambda_2^{(i)} & & \\
& & \ddots & \\
& & & f_{t_{s_i}} \circ \lambda_{s_i}^{(i)}
\end{array}
\right),$$
where $(f_0, f_1, ..., f_K) \in A.
$
\end{defn}



Consider the crossed product C*-algebra $\mathrm{C}(X) \rtimes \Int$, and consider a closed set $Y \subseteq X$ with nonempty interior. The Putnam sub-C*-algebra $A_Y\subseteq \mathrm{C}(X) \rtimes\Int$ is an RSH with diagonal maps in a natural way. The construction was introduced in \cite{HPS-Cantor} for $X$ being the Cantor set and then was generalized in \cite{Lin-Q-RSA} for a general $X$. (More examples will be constructed in the next section using groupoids.)

\begin{defn}\label{defn-AY}
Let $Y$ be a closed subset of $X$ with non-empty interior. The C*-algebra $A_Y$ is defined by
$$A_Y:=\textrm{C*}\{f, ug;\ f, g\in\mathrm{C}(X),\ g|_Y=0\}\subseteq \mathrm{C}(X)\rtimes\Int.$$
\end{defn}

Let us collect some basic properties of this sub-C*-algebra:

Consider the first return times
$$\{j\in \mathbb N\cup \{0\};\ \textrm{$\sigma^j(x)\in Y$, $\sigma^i(x)\notin Y$, $1\leq i\leq j-1$ for some $x\in Y$}\}.$$
Since $\sigma$ is minimal, $X$ is compact, and $Y$ has a non-empty interior, this set of numbers is finite; let us write it as
$$J_1<J_2<\cdots <J_K$$
for some $K\in\mathbb N$.
Since $X$ is an infinite set and $\sigma$ is minimal, the first return time $J_1$ is arbitrarily large if $Y$ is sufficiently small.

For each $1\leq k\leq K$, consider the (locally compact---see below) subset of $X$
$$Z_k=\{x\in Y;\  \textrm{$\sigma^{J_k}(x)\in Y$ but $\sigma^i(x)\notin Y$ for any $1\leq i\leq J_k-1$}\}.$$
Then the sets 
$$\{Z_1, \sigma(Z_1), ..., \sigma^{J_1-1}(Z_1)\}, ..., \{Z_k, \sigma(Z_k), ..., \sigma^{J_k-1}(Z_k)\}$$
---which are naturally listed as shown---form a partition of $X$. This is often called a Rokhlin partition.

\begin{lem}[\cite{Lin-Q-RSA}; see also \cite{QL-Ph-min-diff} or Lemma 2.15 of \cite{NCP-fmd}]\label{QLin-decp}
In terms of the notation introduced above, one has that, for each $1\leq k\leq K$,\\
\indent $\mathrm{(1)}$ the set $Z_1\cup\cdots\cup Z_k$ is closed (and so $Z_k$ is locally compact),\\
\indent $\mathrm{(2)}$ the set $\overline{Z_k}\cap (Z_1\cup\cdots\cup Z_{k-1})$ is the disjoint union of the subsets
$$W_{t_1,..., t_s}=\partial Z_k\cap Z_{t_1}\cap\sigma^{-J_{t_1}}(Z_{t_2})\cap\cdots\cap\sigma^{-(J_{t_1}+\cdots+J_{t_{s-1}})}(Z_{t_s}),$$ where $J_{t_1}+J_{t_{2}} + \cdots +J_{t_s}=J_k$ for some $1\leq t_1, t_2, ..., t_s \leq k$.
\end{lem}

A quite explicit description of the subalgebra $A_Y$ of the crossed product, a C*-algebra of type I,  was obtained by Q.~Lin (\cite{Lin-Q-RSA}). It is a subhomogeneous algebra of order at most $J_K$. In fact, it is an RSH algebra with all gluing maps being diagonal.

\begin{thm}[\cite{Lin-Q-RSA}; see also \cite{QL-Ph-min-diff} or Theorem 2.22 of \cite{NCP-fmd}]\label{Lin-Sub}
In terms of the notation introduced above, one has that the C*-algebra $A_Y$ is isomorphic to the sub-C*-algebra of $\bigoplus_{k=1}^K \mathrm{M}_{J_k}(\mathrm{C}(\overline{Z_k}))$ consisting of the elements $(F_1, ..., F_K)$ with
\begin{displaymath}
F_k|_{W_{t_1, ..., t_s}}=
\left(
\begin{array}{cccc}
F_{t_1}|_{W_{t_1, ..., t_s}} & & & \\
& F_{t_2}\circ\sigma^{J_{t_1}}|_{W_{t_1, ..., t_s}} & &\\
& & \ddots &\\
& & & F_{t_s}\circ\sigma^{J_{t_s-1}}|_{W_{t_1, ..., t_s}}
\end{array}
\right)
\end{displaymath}
whenever 
$$W_{t_1, ..., t_s}=\partial Z_k\cap Z_{t_1}\cap\sigma^{-J_{t_1}}(Z_{t_2})\cap\cdots\cap\sigma^{-(J_{t_1}+\cdots+J_{t_{s-1}})}(Z_{t_s})\neq\O,$$ where $J_{t_1}+J_{t_{2}}+\cdots +J_{t_s}=J_k$. 

Moreover, for any $f, g\in\mathrm{C}(X)$ with $g|_Y=0$, the images of $f, ug\in A_Y$ in this identification are
\begin{equation}\label{id-01}
f=\bigoplus_{k=1}^K
\left(
\begin{array}{cccc}
f\circ \sigma|_{\overline{Z_k}} & & & \\
& f\circ \sigma^2|_{\overline{Z_k}} & &\\
& & \ddots &\\
& & & f\circ \sigma^{J_k}|_{\overline{Z_k}}
\end{array}
\right)\in \bigoplus_{k=1}^K \mathrm{M}_{J_k}(\mathrm{C}(\overline{Z_k}))
\end{equation}
and
\begin{equation}\label{id-02}
ug=\bigoplus_{k=1}^K
\left(
\begin{array}{cccc}
0 & & & \\
g\circ \sigma|_{\overline{Z_k}}& 0& & \\
&  \ddots& \ddots &\\
& & g\circ \sigma^{J_k-1}|_{\overline{Z_k}} & 0 
\end{array}
\right)\in \bigoplus_{k=1}^K \mathrm{M}_{J_k}(\mathrm{C}(\overline{Z_k})),
\end{equation}
respectively.
\end{thm}

\section{Small subgroupoids and recursive subhomogeneous C*-algebras}\label{small-groupoid}

In this section, let us consider a class of small subgroupoid of the transformation groupoid $X \rtimes\Gamma$. It turns out that the C*-algebra of a such subgroupoid is an RSH algebra with diagonal maps.


Consider a topological dynamical system $(X, \Gamma)$. Recall that (see Example 1.2.a of \cite{Renault-LNM}) the transformation groupoid $X \rtimes \Gamma$ is defined by taking $X \times \Gamma$ with 
\begin{enumerate}
\item $\mathrm{r}(x, \gamma) = (x, e)$, $\mathrm{s}(x, \gamma) = (x\gamma, e)$, $(x, \gamma) \in X \times\Gamma$,
\item $(x, \gamma_1)(y, \gamma_2) = (x, \gamma_1\gamma_2)$ if $x\gamma_1 = y$, and
\item $(x, \gamma)^{-1} = (x\gamma, \gamma^{-1})$, $(x, \gamma) \in X \times\Gamma$.
\end{enumerate}
The topology on $X \rtimes\Gamma$ is the product topology on $X \times \Gamma$. Note that $X$, regarded as  $\{(x, e): x\in X\}$, is  the unit space of $X \rtimes\Gamma$, and $X\rtimes\Gamma$ is principle if $(X, \Gamma)$ is free.

\begin{defn}
A small subgroupoid $\mathcal G \subseteq X\rtimes\Gamma$ is a subgroupoid which is open, relatively compact, and $\{(x, e): x\in X\} \subseteq \mathcal G.$
\end{defn}

It is well known that the C*-algebra $\textrm{C*}(\mathcal G)\subseteq \textrm{C*}(X \rtimes\Gamma)$ is subhomogeneous if $\mathcal G$ is small subgroupoid.  Using the orbit structure of $\mathcal G$, let us show that $\textrm{C*}(\mathcal G)$ actually is an RSH-algebra with diagonal maps.

Consider
$$\mathrm{Supp}_\Gamma(\mathcal G) : = \{\gamma \in \Gamma: \exists x\in X,\ (x, \gamma) \in \mathcal G\}.$$ Since $\mathcal G$ is relatively compact, one has that $$N:=\abs{\mathrm{Supp}_\Gamma(\mathcal G)} < +\infty.$$
Since the unit space of the subgroupoid $\mathcal G$ is $X$, it induces an equivalence relation on $X$ by 
$$x \sim_{\mathcal G} y \Longleftrightarrow \exists (x, \gamma) \in \mathcal G,\ x\gamma= y.$$
For each $x\in X$, define
$$\mathrm{Orbit}_{\mathcal G}(x) = [x] = \{y\in X : y\sim_{\mathcal G} x\}.$$
It is clear that $$\abs{\mathrm{Orbit}_{\mathcal G}(x) } \leq N,\quad x\in X.$$

\begin{defn}
Let $\mathcal G\subseteq X \rtimes\Gamma$ be any subgroupoid, and let $x\in X$. Define the shape of $x$ to be
$$\mathrm{S}(x) = \{\gamma \in \Gamma: (x, \gamma) \in \mathcal G \}\subseteq \Gamma.$$
\end{defn}
\begin{rem}
Note that since $(x, e) \in \mathcal G$, $x\in X$, one has $$e \in \mathrm{S}(x), \quad x \in X$$
\end{rem}

The function $\mathrm{S}$ has the following properties.
\begin{lem}
Let $(X, \Gamma)$ be a free dynamical system, and let $\mathcal G$ be a relatively compact subgroupoid of $X \rtimes \Gamma$. Then
\begin{enumerate}
\item $|\mathrm{Orbit}_{\mathcal G}(x)| = |\mathrm{S}(x)|$, $x\in X$;
\item $\mathrm{Orbit}_{\mathcal G}(x) = x\mathrm{S}(x)$, $x \in X$;
\item if $(x, \gamma) \in \mathcal G$, then $$\mathrm{S}(x\gamma) =  \gamma^{-1} \mathrm{S}(x);$$ i.e., if $y\sim_{\mathcal G} x$, then $\mathrm{S}(y) =  \gamma^{-1} \mathrm{S}(x),$ where $y=x\gamma$.
\end{enumerate}
\end{lem}
\begin{proof}
Note that $y\in \mathrm{Orbit}_{\mathcal G}(x)$, if and only if, there is $\gamma\in\Gamma$ such that $(x, \gamma)\in\mathcal G$ (hence $\gamma\in\mathrm{S}(x)$) and $y= x\gamma$; and the latter condition is equivalent to $y\in x\mathrm{S}(x)$. This proves (2), and (1) follows from (2) by the freeness of the action.

Let $\gamma' \in \mathrm{S}(x\gamma)$, then $(x\gamma, \gamma') \in \mathcal G$. Since $(x, \gamma) \in \mathcal G$, one has that $(x, \gamma\gamma') = (x, \gamma)(x\gamma, \gamma') \in \mathcal G$, and hence $\gamma\gamma' \in\mathrm{S}(x)$ and $\gamma' \in \gamma^{-1}\mathrm{S}(x)$. Therefore $\mathrm{S}(x\gamma) \subseteq  \gamma^{-1} \mathrm{S}(x);$

Now, if $\gamma'\in \gamma^{-1} \mathrm{S}(x)$, then $\gamma\gamma'\in\mathrm{S}(x)$, and hence $(x, \gamma\gamma') \in \mathcal G$. Since $(x, \gamma) \in \mathcal G$, one has
$$(x\gamma, \gamma') = (x, \gamma)^{-1} (x, \gamma\gamma') \in\mathcal G,$$
and hence $\gamma'\in \mathrm{S}(x\gamma)$, as desired.
\end{proof}

\begin{defn}
Let $\mathcal F\subseteq\Gamma$ be a finite set containing $e$. Define
$$ Z_{\mathcal F} = \{x\in X: \mathrm{S}(x) = \mathcal F\}.$$
\end{defn}


\begin{lem}\label{orbit-lsc}
Let $\mathcal G \subseteq X \rtimes \Gamma$ be open and relatively compact subgroupoid. Then, the function $x\mapsto \mathrm{S}(x)$ is lower semicontinuous in the following sense: for any $x$, there is an open set $U \ni x$ such that $$\mathrm{S}(x) \subseteq \mathrm{S}(x'), \quad x' \in U.$$  Hence, for any finite subset $\mathcal F\subseteq \Gamma$, if $x_n \in Z_{\mathcal F}$ and $x_n \to x_\infty \in X$, then $$\mathrm{S}(x_\infty) \subseteq \mathcal F.$$ 
\end{lem}

\begin{proof}
Write $\mathrm{S}(x) = \{r_1, r_2, ..., r_n\}$. Since $\mathcal G$ is open, there is $U \ni x$ such that $$(x', \gamma_i) \in \mathcal G,\quad x'\in U,\ i=1, 2, ..., n.$$ In particular this implies $\gamma_i \in \mathrm{S}(x')$, $i=1, 2, ..., n$, as desired.
\end{proof}

\begin{defn}
Let $x\in X$. Define $$\mathcal G^x = \{g\in\mathcal G: \mathrm{r}(g) = x\}.$$
\end{defn}

\begin{rem}
Note that if $x\in Z_{\mathcal F}$ where $\mathcal F \subseteq \Gamma$ is a finite set, 
then the map $(x, \gamma) \mapsto \gamma $ induces a one-to-one correspondence between $\mathcal G^x$ and $\mathrm{S}(x) = \mathcal F$.
\end{rem}

For a small subgroupoid $\mathcal G$, there is a one-to-one correspondence between the orbits and the equivalence classes of the irreducible representations of $\textrm{C*}(\mathcal G)$.
\begin{lem}[Proposition 3.8 of \cite{Clark-JOT2007}]\label{groupoid-rep}
Let $(X, \Gamma)$ be a free dynamical system, and let $\mathcal G \subseteq X \rtimes \Gamma$ be a relatively compact subgroupoid.
Let $x\in X$. The map 
$$\pi_x: \mathrm{C}_{\mathrm c} (\mathcal G) \ni f \mapsto  \sum_{g, h\in \mathcal G^x} f(gh^{-1}) E_{g, h},  $$
where $\{E_{g, h}\}_{\mathcal G^x}$ are the standard matrix units of $\mathcal G^x$, induces an irreducible representation of $\textrm{C*}(\mathcal G_0)$ on $\Comp^{|\mathcal G^x|}$ (if $\mathcal G^x \neq\varnothing$), still denoted by $\pi_x$. Moreover, any irreducible representation of $\textrm{C*}(\mathcal G_0)$ arises in this way, and $x \sim_{\mathcal G} y$ if and only if $\pi_x$ is unitarily equivalent to $\pi_y$.
\end{lem}

\begin{lem}\label{local-homeo}
Let $\mathcal G \subseteq X \rtimes \Gamma$ be a relatively compact subgroupoid. 
For each finite set $\mathcal F\subseteq \Gamma$, there is a homeomorphism $$\mathcal G^{Z_{\mathcal F}} := \{g\in\mathcal G: \mathrm{r}(g) \in Z_{\mathcal F}\} \cong Z_{\mathcal F} \times \mathcal F.$$
\end{lem}

\begin{proof}
Define the map
$$\mathcal G \supseteq \mathcal G^{Z_{\mathcal F}} \ni (x, \gamma) \mapsto (x, \gamma) \in Z_{\mathcal F} \times \mathcal F,$$ and it is  the desired homeomorphism.
\end{proof}

\begin{lem}\label{matrix-rep}
Let $(X, \Gamma)$ be a free dynamical system, and let $\mathcal G\subseteq X\rtimes\Gamma$ be a relatively compact subgroupoid.  Let $\mathcal F\subseteq\Gamma$ be a finite set. Then there is a homomorphism
$$\pi_{Z_{\mathcal F}}: \mathrm{C}^*(\mathcal G) \to \mathrm{M}_{|\mathcal F|}(\mathrm{C}(\overline{Z_{\mathcal F}}))$$ by
$$ (\pi_{Z_{\mathcal F}}(f))(x) = \sum_{\gamma_1, \gamma_2 \in \mathcal F} f(x\gamma_2, \gamma^{-1}_2\gamma_1) E_{\gamma_1, \gamma_2}, \quad x\in \overline{Z_{\mathcal F}},$$
where $\{E_{\gamma_1, \gamma_2}: \gamma_1, \gamma_2\in\mathcal F\}$ are the standard matrix units over $\mathcal F$ and $f\in \mathrm{C_0}(\mathcal G)$ being regarded as a function on $X \rtimes\Gamma$.
\end{lem}

\begin{proof}
By Lemma \ref{local-homeo}, one has that $\mathcal G^{Z_{\mathcal F}}$ is homeomorphic to $\mathcal Z_{\mathcal F} \times \mathcal F$. Then, for each $g\in \mathcal G^{Z_{\mathcal F}}$, there is a unique $\gamma_g \in\mathcal F$ such that $$g = (x, \gamma_g),\quad x=\mathrm{r}(g) \in Z_{\mathcal F}.$$

By Lemma \ref{groupoid-rep}, for any $x\in X$, the map $$\pi_x: f  \mapsto  \sum_{g, h\in \mathcal G^x} f(gh^{-1}) E_{g, h}$$ induces an irreducible representation of $\textrm{C*}(\mathcal G)$.
Then, if $x \in Z_{\mathcal F}$, one has
\begin{eqnarray*}
 \sum_{g, h\in \mathcal G^x} f(gh^{-1}) E_{g, h} 
& = & \sum_{\gamma_g, \gamma_h\in \mathcal F} f(x\gamma_h, \gamma^{-1}_h\gamma_g) E_{(x, \gamma_g), (x, \gamma_h)} \\
& = & \sum_{\gamma_1, \gamma_2\in \mathcal F} f(x\gamma_2, \gamma^{-1}_2\gamma_1) E_{(x, \gamma_1), (x, \gamma_2)}.
\end{eqnarray*}
Regarding $x \mapsto E_{(x, \gamma_1), (x, \gamma_2)}$ as the constant function $E_{\gamma_1,\gamma_2}$, one has the homomorphism
$$\pi_{Z_{\mathcal F}}: f \mapsto(x\mapsto \sum_{\gamma_1, \gamma_2 \in \mathcal F} f(x\gamma_2, \gamma^{-1}_2\gamma_1) E_{\gamma_1, \gamma_2}) \in \mathrm{M}_{|\mathcal F|}(\mathrm{C}({Z_{\mathcal F}})).$$
Note the function $\pi_{Z_{\mathcal F}}(f)$ can be continuously extended to $\partial Z_{\mathcal F}$, and the evaluation at each $x\in\partial Z_{\mathcal F}$ is still a (but non-irreducible) representation of $\textrm{C*}(\mathcal G)$. 
\end{proof}


\begin{lem}\label{matrix-embedding}
Let $(X, \Gamma)$ be a free dynamical system, and let $\mathcal G\subseteq X\rtimes\Gamma$ be a relatively compact subgroupoid with unit space $X$. 
Since the unit space of $\mathcal G$ is $X$, one has the decomposition
$$X= (\bigsqcup_{\gamma\in \mathcal F_1} Z_{\gamma^{-1}\mathcal F_1}) \sqcup (\bigsqcup_{\gamma\in \mathcal F_2} Z_{\gamma^{-1}\mathcal F_2})\sqcup\cdots\sqcup (\bigsqcup_{\gamma\in \mathcal F_L} Z_{\gamma^{-1}\mathcal F_L}),$$
where $\mathcal F_1, ..., \mathcal F_L$ are finite subsets of $\Gamma$. 
Then the map
$$\Phi= \bigoplus_{i=1}^L \pi_{Z_{\mathcal F_i}}: \mathrm{C^*}(\mathcal G) \to \bigoplus_{i=1}^L \mathrm{M}_{\abs{\mathcal F_i}}(\overline{Z_{\mathcal F_i}}))$$
is an injection.
\end{lem}
\begin{proof}
Note that any irreducible representation of $\mathrm{C^*}(\mathcal G)$ factors though $\Phi$, and hence $\Phi$ must be injective. 
\end{proof}

\begin{lem}\label{orbit-breaking}
Let $(X, \Gamma)$ be a free dynamical system, and let $\mathcal G\subseteq X\rtimes\Gamma$ be a small subgroupoid.
Let $\mathcal F\subseteq\Gamma$ be a finite set with $Z_{\mathcal F} \neq \varnothing$, and let $x\in\overline{Z_{\mathcal F}}$. Then there are finite sets $\mathcal F_1, \mathcal F_2, ..., \mathcal F_s$ and $\gamma_1, \gamma_2, ..., \gamma_s \in\mathcal F$ such that $e\in\mathcal F_i$, $i=1, 2, ..., s$, 
$$\mathcal F = \gamma_1 \mathcal F_1 \sqcup \gamma_2 \mathcal F_2 \sqcup \cdots \sqcup \gamma_s \mathcal F_s,$$
and
$$\mathrm{S}(x\gamma_i) = \mathcal F_i,\quad i=1, 2, ..., s.$$ The decomposition of $\mathcal F$ is unique (for the given $x$).

%
%
\end{lem}

\begin{proof}
Since the unit space of $\mathcal G$ is $X$, there is a partition
$$x\mathcal F = S_1 \sqcup S_2 \sqcup \cdots \sqcup S_s$$ such that
$y \sim_{\mathcal G} z$ if and only if $y, z \in S_i$ for some $S_i$. 

One claims that $S_i = [y]$, $i=1, 2, ..., s$, whenever $y\in S_i$. Then the lemma follows.

For each $S_i$ and $y\in S_i$, let $z\in X$ be $z\sim_{\mathcal G} y$. Since $z \sim_{\mathcal G} y$, there is $\gamma\in\Gamma$ such that $z=(y, \gamma) \in \mathcal G$. Since $\mathcal G$ is open, one has that $(y', \gamma) \in \mathcal G$ if $y'$ is sufficiently close to $y$. Since $y \in x \mathcal F$, there is $\gamma' \in \mathcal F$ such that $y = x\gamma'$. Note that $x\in\overline{Z_{\mathcal F}}$, there is a sequence $(x_n) \subseteq Z_{\mathcal F}$ with $x_n \to x$ as $n \to\infty$, and hence $x_n\gamma' \to x \gamma' = y$, as $n\to\infty$. In particular, $(x_n\gamma', \gamma) \in \mathcal G$ for sufficiently large $n$. Noting that $(x_n, \gamma') \in\mathcal G$ (since $x_n \in Z_{\mathcal F}$ and $\gamma'\in\mathcal F$), one has $$(x_n, \gamma'\gamma) = (x_n, \gamma')(x_n\gamma', \gamma) \in \mathcal G.$$ Hence $\gamma'\gamma\in \mathcal F$, and $ z = y \gamma = x\gamma'\gamma \in x \mathcal F,$ as desired.
%
\end{proof}

\begin{cor}\label{connecting-map}
Let $(X, \Gamma)$ be a free dynamical system, and let $\mathcal G \subseteq X \rtimes \Gamma$ be a small groupoid. With the notation as above, one has that for any $f\in \textrm{C*}(\mathcal G)$, 
$$ \pi_{Z_{\mathcal F}}(f)(x) = \mathrm{diag}\{\pi_{Z_{\mathcal F_1}}(f)(x\gamma_1), ..., \pi_{Z_{\mathcal F_s}}(f)(x\gamma_s) \},\quad x\in\overline{Z_{\mathcal F}}.$$ 
\end{cor}
\begin{proof}
Note that if $\gamma_g \in \gamma_i\mathcal F_i$ and $\gamma_h \in \gamma_j\mathcal F_j$ with $i \neq j$, then $(x\gamma_h, \gamma_h^{-1}\gamma_g) \notin \mathcal G$. Otherwise, $x\gamma_h \sim_{\mathcal G} x\gamma_g$, which contradicts the fact that $\mathrm{S}(x\gamma_i) = \mathcal F_i$ and $\mathrm{S}(x\gamma_j) = \mathcal F_j$. In particular, for any $f\in \mathrm{C}_0(\mathcal G)$, one has $$f((x\gamma_h, \gamma_h^{-1}\gamma_g)) = 0,\quad \textrm{if $\gamma_g \in \gamma_i\mathcal F_i$ and $\gamma_h \in \gamma_j\mathcal F_j$ with $i \neq j$}.$$
It then follows from Lemma \ref{matrix-rep} that for any $f\in \mathrm{C}_0(\mathcal G)$,
\begin{eqnarray*}
(\pi_{Z_{\mathcal F}}(f))(x) & = & \sum_{\gamma_g, \gamma_h \in \mathcal F}f(x\gamma_h, \gamma_h^{-1}\gamma_g) E_{\gamma_g, \gamma_h} \\
& = &  \sum_{i=1}^s \sum_{\gamma_g, \gamma_h \in \gamma_i\mathcal F_i}f(x \gamma_h, \gamma_h^{-1}\gamma_g) E_{\gamma_g, \gamma_h} \\
& = &  \sum_{i=1}^s \sum_{\gamma_g, \gamma_h \in \mathcal F_i}f(x\gamma_i \gamma_h, \gamma_h^{-1}\gamma_g) E_{\gamma_g, \gamma_h} \\
& = & \bigoplus_{i=1}^s (\pi_{Z_{\mathcal F_i}}(f))(x\gamma_i),
\end{eqnarray*}
as desired.
\end{proof}

\begin{thm}\label{diag-SHA}
Let $(X, \Gamma)$ be a free dynamical system, let $\mathcal G\subseteq X\rtimes \Gamma$ be a small subgroupoid, and let $E\subseteq X$ be an open set. Then there is an isomorphism $$\mathrm{C^*}(\mathcal G) \cong \left[\cdots\left[\left[ C_0\oplus_{C_1^{(0)}} C_1 \right]\oplus_{C_2^{(0)}}C_2\right] \cdots\right]\oplus_{C_K^{(0)}} C_K,$$ where $C_i=\mathrm{M}_{n_i}(\mathrm{C}(Z_i))$ for a set $Z_i$ which is a disjoint union of finitely many closed subsets of $X$, and $C_i^{(0)} = \mathrm{M}_{n_i}(\mathrm{C}(Z^{(0)}_i))$ with $Z^{(0)}_i$ a closed subset of $Z_i$,
such that
\begin{enumerate}
\item $n_0 < n_1 < \cdots < n_K$,
\item all the maps $\eta_i: A_i \to C_{i+1}^{(0)}$ are of diagonal type, where $$A_i:= \left[\cdots\left[\left[ C_0\oplus_{C_1^{(0)}} C_1 \right]\oplus_{C_2^{(0)}}C_2\right] \cdots\right]\oplus_{C_{i}^{(0)}} C_{i},$$
\item under this isomorphism and the RSH-decomposition, the element $\varphi_E\in\mathrm{C^*}(\mathcal G)$ is a diagonal matrix on each $Z_i$ and $$\mathrm{rank}((\varphi_E)(x)) = \abs{\mathrm{Orbit}_{\mathcal G}(x) \cap E}, \quad x\in Z_i\setminus Z_i^{(0)},$$
where $\varphi_E$ is defined in \eqref{defn-phi-E}.
\end{enumerate}
\end{thm}
\begin{proof}
Decompose $X$ into
$$ (\bigsqcup_{\gamma\in \mathcal F_1} Z_{\gamma^{-1}\mathcal F_1}) \sqcup (\bigsqcup_{\gamma\in \mathcal F_2} Z_{\gamma^{-1}\mathcal F_2})\sqcup\cdots\sqcup (\bigsqcup_{\gamma\in \mathcal F_L} Z_{\gamma^{-1}\mathcal F_L}),$$
where $\mathcal F_1, ..., \mathcal F_L$ are finite subsets of $\Gamma$. 
By Lemma \ref{matrix-embedding}, there is an embedding 
$$\Phi= \bigoplus_{i=1}^L \pi_{Z_{\mathcal F_i}}: \mathrm{C^*}(\mathcal G) \to \bigoplus_{i=1}^L \mathrm{M}_{\abs{\mathcal F_i}}(\overline{Z_{\mathcal F_i}})).$$
Reindex $\mathcal F_i$ into 
$$ \{\mathcal F_{1, 1}, \mathcal F_{1, 2}, ..., \mathcal F_{1, l_1}\}, \{\mathcal F_{2, 1}, \mathcal F_{2, 2}, ..., \mathcal F_{2, l_2} \}, ..., \{\mathcal F_{K, 1}, \mathcal F_{K, 2}, ..., \mathcal F_{K, l_K}\} $$
so that $$\abs{\mathcal F_{i, j_1}} = \abs{\mathcal F_{i, j_2}},\quad 1\leq i\leq K,\ 1\leq j_1, j_2 \leq l_i,$$
and
$$\abs{\mathcal F_{i, 1}} < \abs{\mathcal F_{i+1, 1}}, \quad 1\leq i\leq K-1.$$

Put $$n_i = \abs{\mathcal F_{i, 1}},\quad Z_i = \bigsqcup_{j=1}^{l_i} \overline{Z_{\mathcal F_{i, j}}}, \quad \textrm{and} \quad Z_i^{(0)} = \bigsqcup_{j=1}^{l_i} \overline{Z_{\mathcal F_{i, j}}} \setminus Z_{\mathcal F_{i, j}}.$$
Note that the sets $ \overline{Z_{\mathcal F_{i, j}}}$, $j=1, ..., l_i$, considered as subsets of $X$, might have non-empty intersections; but one takes the abstract disjoint union in the construction of $Z_i$. Also note that $$n_1 < n_2 < \cdots < n_K,$$ and $Z_i^{(0)}$ is a closed subset of $Z_i$ (follows from Lemma \ref{orbit-lsc}).

By Lemma \ref{orbit-breaking}, each $\overline{Z_{\mathcal F_{i, j}}} \setminus Z_{\mathcal F_{i, j}}$ has a partition
$$\overline{Z_{\mathcal F_{i, j}}} \setminus Z_{\mathcal F_{i, j}} = (Z_{i, j}^{(0)})_1 \sqcup (Z_{i, j}^{(0)})_2 \sqcup\cdots\sqcup  (Z_{i, j}^{(0)})_{d_{i, j}}$$ such that for each $(Z_{i, j}^{(0)})_d$, there is a decomposition 
$$\mathcal F_{i, j} = \gamma_1 \mathcal F_{i, j}^{(1)} \sqcup \gamma_2 \mathcal F_{i, j}^{(2)} \sqcup \cdots\sqcup \gamma_{s_d} \mathcal F_{i, j}^{(s_d)}$$ with $\gamma_1, ..., \gamma_{s_d} \in \mathcal F_{i, j}$ such that
$$\mathrm{S}(x\gamma_s) = \mathcal F_{i, j}^{(s)},\quad s=1, 2, ..., s_d,\ x\in (Z_{i, j}^{(0)})_d,$$
and the decomposition of $\mathcal F_{i, j}$ is different for each $d$. 

For each $(Z_{i, j}^{(0)})_d$, $d=1, ..., d_{i, j}$, consider the map
$$ \bigoplus_{s=1}^{s_d}  \mathrm{M}_{\abs{\mathcal F_{i, j}^{(s)}}}(\mathrm{C}(\overline{Z_{\mathcal F_{i, j}^{(s)}}}))\ni (f_1, ..., f_{s_d}) \mapsto (x \mapsto \bigoplus_{s=1}^{s_d} f_s(x\gamma_i))  \in \mathrm{M}_{n_i}(\mathrm{C}((Z_{i, j}^{(0)})_d)),$$
and with all of these maps, one has the RSH algebra 
$$A:=\left[\cdots\left[\left[ C_0\oplus_{C_1^{(0)}} C_1 \right]\oplus_{C_2^{(0)}}C_2\right] \cdots\right]\oplus_{C_K^{(0)}} C_K,$$
with
$C_i = \mathrm{M}_{n_i}(\mathrm{C}(Z_i))$, $C_i^{(0)} = \mathrm{M}_{n_i}(\mathrm{C}(Z_i^{(0)}))$, and all the maps $A_i \to C_{i+1}^{(0)}$ are of diagonal type.

By Corollary \ref{connecting-map}, one has $\Phi(\textrm{C*}(\mathcal G)) \subseteq A$. One the other hand, the embedding map $\Phi$ induces a one-to-one correspondence between the irreducible representations of $\textrm{C*}(\mathcal G)$ and the irreducible representations of $A$. Thus $\textrm{C*}(\mathcal G)$ is a rich subalgebra of $A$, and hence $\Phi(\textrm{C*}(\mathcal G)) = A$ (see Theorem 11.1.6 of \cite{Dix-book}).

Consider the open set $E$ and the function $\varphi_E \in \mathrm{C}(X)$. Then, for each $x\in {Z_{\mathcal F_{i, j}}}$, 
\begin{eqnarray*}
\mathrm{rank}((\pi_{Z_{\mathcal F_{i, j}}} (\varphi_E))(x)) & = & \mathrm{rank}(\sum_{\gamma_g, \gamma_h \in \mathcal F_{i, j}} \varphi_E(x\gamma_h, \gamma^{-1}_h\gamma_g) E_{\gamma_g, \gamma_h}) \\
& = & \mathrm{rank}(\sum_{\gamma_g \in \mathcal F_{i, j}} \varphi_E(x\gamma_g, e) E_{\gamma_g, \gamma_g}) \\
& = & \abs{\{\gamma_g \in \mathcal F_{i, j}: x\gamma_g \in E \}} \\
& = & \abs{\{(x, \gamma_g) \in \mathcal G: x\gamma_g \in E \}} \\
& = & \abs{\mathrm{Orbit}_{\mathcal G}(x) \cap E},
\end{eqnarray*}
as desired.
\end{proof}

\section{Comparison of diagonal elements and comparison of open sets}

In this section, let us show that the Cuntz comparison of the diagonal elements in an RSH-algebra with diagonal maps is roughly determined by their ranks (see Theorem \ref{DSH-comp}), regardless of the dimensions of the base spaces. As a consequence, the Cuntz comparison of open subsets of $X$ is then determined by the partial orbit or by the measures if there exist small subgroupoids with arbitrarily large orbit (see Corollary \ref{orb-comp} and Corollary \ref{orb-comp-measure}). The proof of the main result of this section (Theorem \ref{DSH-comp}) follows closely the argument of Theorem 4.5 of \cite{EN-MD0}.

First, recall the following comparison theorem for general RSH algebras, .

\begin{thm}[Theorem 4.6 of \cite{Toms-Comp-DS}]\label{pre-comp-dim-rank}
Let $A$ be a recursive subhomogeneous C*-algebras with a decomposition 
$$\left[\cdots\left[\left[ C_0\oplus_{C_1^{(0)}} C_1 \right]\oplus_{C_2^{(0)}}C_2\right] \cdots\right]\oplus_{C_l^{(0)}} C_l,$$
where $$C_i=\mathrm{M}_{m_i}(\mathrm{C}(X_i))\quad \textrm{and}\quad C_i^{(0)} = \mathrm{M}_{m_i}(\mathrm{C}(X^{(0)}_i))$$ with $X^{(0)}_i$ a closed subset of $X_i$. Let $a, b\in A^+$ be positive elements satisfying
\begin{equation*}
\mathrm{rank}(a(x)) + \frac{\mathrm{dim}(X_i) - 1}{2} < \mathrm{rank}(b(x)) ,\quad x\in X_i\setminus X_i^{(0)},
\end{equation*} 
for any $i=0, 1, ..., l$.
Then $a \precsim b$.
\end{thm}


Using this theorem, one has the following comparison result.

\begin{prop}\label{comp-dim-rank}
Let $A$ be a separable RSH algebra with a fixed decomposition 
$$\left[\cdots\left[\left[ C_0\oplus_{C_1^{(0)}} C_1 \right]\oplus_{C_2^{(0)}}C_2\right] \cdots\right]\oplus_{C_l^{(0)}} C_l,$$ and let $a, b$ be positive elements of  $A$. Write $$C_i=\mathrm{M}_{n_i}(\mathrm{C}(X_i))\quad \textrm{and}\quad C_i^{(0)} = \mathrm{M}_{n_i}(\mathrm{C}(X^{(0)}_i))$$ with $X^{(0)}_i$ a closed subset of $X_i$. 

For each $i=0, 1,..., l$,
list $$\{\mathrm{rank}(b)(x): x\in X_i\} = \{r_{i, 0} < r_{i, 1} < \cdots < r_{i, s_i}\}$$ and set
$$W_{i, j} = \{x \in X_i: \mathrm{rank}(b(x)) \leq r_{i, j}\}$$ and $$Z_{i, j} = \{x \in X_i: \mathrm{rank}(b(x)) = r_{i, j}\},$$ $j=0, 1, ..., s_i.$

Assume that for each $i=1, 2, ..., l$ and each $j=0, 1, ..., s_i$, 
$$\mathrm{rank}(a(x)) + \frac{\mathrm{dim}(W_{i, j}) - 1}{2} < \mathrm{rank}(b(x)) = r_{i, j},\quad x\in Z_{i, j}.$$
Then $a \precsim b$ in $A$.
\end{prop}

\begin{proof}
Let us construct a new RSH decomposition of $A$: $$\left[\cdots\left[\left[ D_0\oplus_{D_1^{(0)}} D_1 \right]\oplus_{D_2^{(0)}}D_2\right] \cdots\right]\oplus_{D_s^{(0)}} D_s,$$
with $D_i=\mathrm{M}_{m_i}(\mathrm{C}(Y_i))$ and $D_i^{(0)} = \mathrm{M}_{m_i}(\mathrm{C}(Y^{(0)}_i))$ with $Y^{(0)}_i$ a closed subset of $Y_i$ such that
\begin{equation}\label{dim-gap-pre}
\mathrm{rank}(a(x)) + \frac{\mathrm{dim}(Y_i) - 1}{2} < \mathrm{rank}(b(x)) ,\quad x\in Y_i\setminus Y_i^{(0)}
\end{equation} 
in this new RSH decomposition. Then, it follows from Theorem \ref{pre-comp-dim-rank} that $a\precsim b$. 

 To building the new RSH decomposition, let us start with $$C_0=\mathrm{M}_{n_0}(\mathrm{C}(X_0)).$$ Since the rank function is lower semicontinuous, one has that 
$$W_{0, 0} \subseteq W_{0, 1} \subseteq \cdots\subseteq W_{0, s_0}$$ 
are closed. (Note that $W_{0, j} = Z_{0, 0}\cup\cdots\cup Z_{0, j}$.)
 
Set $$D_{0, 0}=\mathrm{M}_{n_0}(\mathrm{C}(W_{0, 0})),\quad D_{0, 1}=\mathrm{M}_{n_0}(\mathrm{C}(W_{0, 1})),\quad D^{(0)}_{0, 1}=\mathrm{M}_{n_0}(\mathrm{C}(W_{0, 0}))$$ with the map $D_{0, 0} \to D^{(0)}_{0, 1}$ being the identity map. Then it is clear that $$ D_{0, 0}\oplus_{D_{0, 1}^{(0)}} D_{0, 1} = \mathrm{M}_{n_0}(\mathrm{C}(W_{0, 1})).$$ 
Putting $$D_{0, 2} = \mathrm{M}_{n_0}(\mathrm{C}(W_{0, 2}))\quad \textrm{and} \quad D^{(0)}_{0, 2}=\mathrm{M}_{n_0}(\mathrm{C}(W_{0, 1})),$$ and repeating this construction, one obtains the new RSH decomposition of $C_0=\mathrm{M}_{n_0}(\mathrm{C}(X_0))$ by
$$\left[\cdots\left[\left[ D_{0, 0}\oplus_{D_{0, 1}^{(0)}} D_{0, 1} \right]\oplus_{D_{0, 2}^{(0)}}D_{0, 2}\right] \cdots\right]\oplus_{D_{0, s_0}^{(0)}} D_{0, s_0} = \mathrm{M}_{n_0}(\mathrm{C}(W_{0, s_0}))=\mathrm{M}_{n_0}(\mathrm{C}(X_0)),$$ where $$D_{0, j} = \mathrm{M}_{n_0}(\mathrm{C}(W_{0, j}))\quad\mathrm{and}\quad D_{0, j}^{(0)} = \mathrm{M}_{n_0}(\mathrm{C}(W_{0, j-1})).$$
It clearly satisfies \eqref{dim-gap-pre}.

Now, assume that one has a desired decomposition of 
$$A_i:=\left[\cdots\left[\left[ C_0\oplus_{C_1^{(0)}} C_1 \right]\oplus_{C_2^{(0)}}C_2\right] \cdots\right]\oplus_{C_{i}^{(0)}} C_{i};$$ let us construct the decomposition for $A_i\oplus_{C_{i+1}^{(0)}} C_{i+1}$ which satisfies \eqref{dim-gap-pre}.

Write $$b=(b_i, b_{i+1}) \in A_i\oplus_{C_{i+1}^{(0)}} C_{i+1},$$ and denote by $\rho: A_i \to C_{i+1}^{(0)}$ the homomorphism associated with the (original) RSH decomposition of $A$. 


Set $$D_{i+1, 0} = \mathrm{M}_{n_{i+1}}(\mathrm{C}(W_{i+1, 0})),\quad D^{(0)}_{i+1, 0} = \mathrm{M}_{n_{i+1}}(\mathrm{C}(X_{i+1}^{(0)} \cap W_{i+1, 0})),$$ and $$\eta_0: A_i \ni f \mapsto \rho(f)|_{X_{i+1}^{(0)}\cap W_{i+1, 0}} \in D^{(0)}_{i+1, 0}.$$ 
For the next stage, consider 
$$D_{i+1, 1} = \mathrm{M}_{n_{i+1}}(\mathrm{C}(W_{i+1, 1})),\quad D^{(0)}_{i+1, 1} = \mathrm{M}_{n_{i+1}}(\mathrm{C}((X_{i+1}^{(0)}\cap W_{i+1, 1})\cup W_{i+1, 0})),$$  
and 
$$\eta_1: A_i\oplus_{D^{(0)}_{i+1, 0}} D_{i+1, 0} \ni (f, g) \mapsto (\rho(f)|_{X_{i+1}^{(0)}\cap W_{i+1, 1}}, g) \in D_{i+1, 1}^{(0)},$$
where $(\rho(f)|_{X_{i+1}^{(0)}\cap W_{i+1, 1}}, g)$ is regarded as the (well defined) function on $$(X_{i+1}^{(0)}\cap W_{i+1, 1})\cup W_{i+1, 0},$$ which is $\rho(f)$ on $X_{i+1}^{(0)} \cap W_{i+1, 1}$ and is $g$ on $W_{i+1, 0}$. 

At $j$-th stage, set $$D_{i+1, j} = \mathrm{M}_{n_{i+1}}(\mathrm{C}(W_{i+1, j})),\quad D^{(0)}_{i+1, j} = \mathrm{M}_{n_{i+1}}(\mathrm{C}((X_{i+1}^{(0)} \cap W_{i+1, j})\cup W_{i+1, j-1})),$$ and 
\begin{eqnarray*}
\eta_j: (f, g_0, ..., g_{j-1}) & \mapsto & (\rho(f)|_{X_{i+1}^{(0)}\cap W_{i+1, j}}, g_0, ..., g_{j-1}) \in D_{i+1, j}^{(0)},
\end{eqnarray*}
where $$(f, g_0, ..., g_{j-1}) \in \left[\left[A_i\oplus_{D^{(0)}_{i+1, 0}} D_{i+1, 0}\right]\cdots\right] \oplus_{D_{i+1, j-1}^{(0)}} D_{i+1, j-1},$$ and $$(\rho(f)|_{X_{i+1}^{(0)}\cap W_{i+1, j}}, g_0, ..., g_{j-1})$$ is regarded as the function on $$(X_{i+1}^{(0)} \cap W_{i+1, j})\cup W_{i+1, j-1}$$ which is $\rho(f)$ on $X_{i+1}^{(0)}\cap W_{i+1, j}$ and is $g_{j-1}$ on $W_{i+1, j-1}$. It is well defined. 

In this new RSH-decomposition, note that the base spaces are $W_{i+1, j}$ together with the closed subset $(X_{i}^{(0)}\cap W_{i+1, j})\cup W_{i+1, j-1}$. Note that 
$$ W_{i+1, j} \setminus ( (X_{i}^{(0)}\cap W_{i+1, j})\cup W_{i+1, j-1} ) \subseteq W_{i+1, j} \setminus W_{i+1, j-1} = Z_{i+1, j},$$ and hence
$$\mathrm{rank}(b(x)) = r_{i+1, j},\quad x\in W_{i+1, j} \setminus ( (X_{i}^{(0)}\cap W_{i+1, j})\cup W_{i+1, j-1} ),$$
and it satisfies the condition \eqref{dim-gap-pre}.
\end{proof}

\begin{defn}
Let $A$ be a separable RSH algebra with a fixed decomposition 
$$\left[\cdots\left[\left[ C_0\oplus_{C_1^{(0)}} C_1 \right]\oplus_{C_2^{(0)}}C_2\right] \cdots\right]\oplus_{C_K^{(0)}} C_K.$$ Write $$C_k=\mathrm{M}_{n_k}(\mathrm{C}(X_k))\quad \textrm{and}\quad C_k^{(0)} = \mathrm{M}_{n_k}(\mathrm{C}(X^{(0)}_k))$$ with $X^{(0)}_k$ a closed subset of $X_k$. Let $\mathcal F\subseteq A$ be a finite set such that if $$f = (f_0, f_1, ..., f_K),\quad f\in\mathcal F,$$ then $$f_k \in \mathrm{M}_{n_k}(\mathrm{C}_\Real(X_k)),\quad k=0, 1, ..., K.$$ Then, for each $k=0, 1, ..., K$, 
 define the map $\Phi_{\mathcal F, k}: X_k \to \mathbb{R}^{n_k^2\abs{\mathcal F}}$ by
$$\Phi_{\mathcal F, k}: X_k \ni x \mapsto \bigoplus_{f\in \mathcal F}\bigoplus_{i, j=1}^{n_k} (f_k)_{i, j}(x) \in \bigoplus_{f\in \mathcal F}\bigoplus_{i, j=1}^{n_k} \Real = \mathbb{R}^{n_k^2\abs{\mathcal F}}.$$
\end{defn}

The following lemma shows that in an RSH algebra $A$ with diagonal maps, there always exist finitely many elements of $ A$ such that these elements span the whole $\mathrm{M}_n(\Comp)$ under any irreducible representation of $A$ but the topological dimension of the values of their entries is controlled by the dimension of the irreducible representations of $A$. These elements play the same role as that of the element $uH$ in the proof of Theorem 4.5 of \cite{EN-MD0}.

\begin{lem}\label{qsudo-system}
Let $A$ be a separable RSH algebra with a fixed decomposition 
$$\left[\cdots\left[\left[ C_0\oplus_{C_1^{(0)}} C_1 \right]\oplus_{C_2^{(0)}}C_2\right] \cdots\right]\oplus_{C_K^{(0)}} C_K.$$ Write $$C_i=\mathrm{M}_{n_i}(\mathrm{C}(X_i))\quad \textrm{and}\quad C_i^{(0)} = \mathrm{M}_{n_i}(\mathrm{C}(X^{(0)}_i))$$ with $X^{(0)}_i$ a closed subset of $X_i$. 
Assume
\begin{enumerate}
\item $n_0 < n_1 < \cdots < n_K$, and
\item all the maps $$\eta_i: A_i \to C_{i+1}^{(0)},\quad i=0, 1, ..., K-1,$$ are of diagonal type, where $$A_i = \left[\cdots\left[\left[ C_0\oplus_{C_1^{(0)}} C_1 \right]\oplus_{C_2^{(0)}}C_2\right] \cdots\right]\oplus_{C_i^{(0)}} C_i.$$
\end{enumerate}
Then there are finite sets $$\mathcal E_k \subseteq A_k,\quad k=0, 1, ..., K,$$
such that 
\begin{enumerate}
\item \begin{equation*}
         \eta_k(\mathcal E_{k-1}) \subseteq \pi_{X_k^{(0)}}(\mathcal E_k),\quad k=1, 2, ..., K;
         \end{equation*}
\item there are continuous functions $$\chi_i: X_i \to [0, 1],\quad i=0, 1, ..., K,$$  such that $\chi_i^{-1}(0) = X_i^{(0)}$ and 
         \begin{equation}\label{full-irrd}
         \{\chi_i e_{s, t}: 1 \leq s, t\leq n_i \} \subseteq \mathcal E_K|_{X_i},
        \end{equation}
where $\{e_{s, t}: 1\leq s, t \leq n_i\}$ are the standard matrix units of $\mathrm{M}_{n_i}(\Comp) \subseteq \mathrm{M}_{n_i}(\mathrm{C}(X_i))$;
\item and there are polyhedrons
$$
V_{k}\subseteq \Real^{n_k^2\abs{\mathcal E_k}},\quad 0\leq k \leq K,$$ 
such that 
\begin{equation}\label{V-dim}
\mathrm{dim}(V_{k}) \leq \frac{n_k-n_0}{n_0}
\end{equation}
and
\begin{equation}\label{V-cont}
\Phi_{\mathcal E_k, k}(X_{k}) \subseteq V_k
\end{equation}
for any $0\leq k\leq K$. 
\end{enumerate}
\end{lem}

\begin{proof}
Let us construct $\mathcal E_k$, $k=0, 1, ..., K$, recursively. Consider $$C_0\cong\mathrm{M}_{n_0}(\mathrm{C}(X_0)),$$ and put
$$\mathcal E_0=\{0, e_{i, j}^{(0)}: 1\leq i, j \leq n_0\},$$
where $e_{i, j}^{(0)}$ are matrix units of $\mathrm{M}_{n_0}(\Comp) \subseteq \mathrm{M}_{n_0}(\mathrm{C}(X_0))$. Since $\mathcal E_0$ consists of constant functions,  $\Phi_{\mathcal E_0, 0}$ is constant. Then it is clear that $$\mathcal E_0 \subseteq C_0\quad\mathrm{and}\quad V_0:=\{\mathrm{pt}\},$$ where $\mathrm{pt}$ is the constant value of $\Phi_{\mathcal E_0, 0}$, satisfy \eqref{full-irrd}, \eqref{V-dim}, and \eqref{V-cont} (with $\chi_0$ the constant function $1$).

Assume $$\mathcal E_{0}\subseteq A_0,\ \mathcal E_{1}\subseteq A_1,\ ...,\ \mathcal E_{k-1}\subseteq A_{k-1}\ \mathrm{and}\ V_0, V_1,\ ..., V_{k-1}$$ are constructed to satisfy \eqref{full-irrd}, \eqref{V-dim}, and \eqref{V-cont}. Consider $$A_k:=A_{k-1} \oplus_{C_k^{(0)}} C_k.$$

If $X_k^{(0)} = \varnothing$, then $A_k:=A_{k-1} \oplus C_k$. Define $$\mathcal E_k' = \{e_{i, j}^{(k)}: 1\leq i, j \leq n_k\}\subseteq C_k,$$
where $e_{i, j}^{(k)}$ are matrix units of $\mathrm{M}_{n_k}(\Comp) \subseteq \mathrm{M}_{n_k}(\mathrm{C}(X_k))=C_k$. Then $$\mathcal E_k:= \mathcal E_{k-1} \cup \mathcal E_{k}',$$ where $\mathcal E_{k-1}$ and $\mathcal E_k'$ are regarded as subsets of $A_{k-1}\oplus C_k$ naturally, and $$V_k:=\{\mathrm{pt}\},$$ 
where $\mathrm{pt}$ is the constant value of $\Phi_{\mathcal E_k, k}$ on $X_k$, satisfy \eqref{full-irrd}, \eqref{V-dim}, and \eqref{V-cont} (with $\chi_k$ the constant function $1$). 

Assume that $X_k^{(0)} \neq \varnothing$. Since the map $\eta$ is of diagonal type, there exist $l\in\mathbb N$ and a partition $$X_k^{(0)} = Y_1 \sqcup \cdots \sqcup Y_l $$
such that for each $Y_i$, there are continuous maps $$\lambda_1^{(i)}: Y_i \to X_{t_1},\  \lambda_2^{(i)}: Y_i \to X_{t_2},\ ...,\  \lambda_{s_i}^{(i)}: Y_i \to X_{t_{s_i}} $$ for some $s_i\in\mathbb N$ such that
$$\eta((f_0, f_1, ..., f_{k-1}))|_{Y_i} = 
\left(
\begin{array}{cccc}
f_{t_1} \circ \lambda_1^{(i)} & & & \\
& f_{t_2} \circ \lambda_2^{(i)} & & \\
& & \ddots & \\
& & & f_{t_{s_i}} \circ \lambda_{s_i}^{(i)}
\end{array}
\right),$$
where $(f_0, f_1, ..., f_{k-1}) \in A_{k-1}.$
Then it is clear that 
\begin{equation}\label{pre-cont}
\bigoplus_{e\in\mathcal E_{k-1}}\bigoplus_{i, j=1}^{n_k}\eta(e)_{i, j}(x) \subseteq V_{t_1} \oplus\cdots\oplus V_{t_{s_i}} =: V_{k, i}',\quad x\in Y_i.
\end{equation}
Let us estimate the dimension of $V_{k, i}'$. Let $(m_0, m_1, ..., m_{k-1})$ be the multiplicities of the map $\eta|_{Y_i}$. Note that
$$n_k=m_0n_0+m_1n_1+\cdots+m_{k-1}n_{k-1}$$
and
\begin{eqnarray*}
\mathrm{dim}(V_{k, i}') & = & \mathrm{dim}(V_{t_1}) + \cdots + \mathrm{dim}(V_{t_{s_i}}) \\
&=& m_0 \mathrm{dim}(V_0) + m_1 \mathrm{dim}(V_1) +  \cdots + m_{k-1} \mathrm{dim}(V_{k-1}).
\end{eqnarray*}
Since $n_k > n_{k-1}$, one has that
$$m_0 + m_1 + \cdots + m_{k-1} \geq 2.$$ 
Hence, for each $i=1, 2, ..., k$, one has
\begin{eqnarray*}
\mathrm{dim}(V_{k, i}') & = & m_0 \mathrm{dim}(V_0) + m_1 \mathrm{dim}(V_1) + \cdots + m_{k-1} \mathrm{dim}(V_{k-1})  \\
& \leq & (m_0 (n_0-n_0) + m_1 (n_1 - n_0) + \cdots + m_{i-1} ({n_{i-1} - n_0})) \cdot \frac{1}{n_0}  \\
& = & ((m_0 n_0 + m_1 n_1 + \cdots + m_{k-1} n_{k-1}) - \\
& & (m_0  + m_1 + \cdots + m_{k-1})n_0) \cdot \frac{1}{n_0} \\
&\leq & (n_k -2n_0) \cdot \frac{1}{n_0} = \frac{n_k}{n_0} - 2.
\end{eqnarray*}

Put $$V_k' = V'_{k, 1} \cup V'_{k, 2} \cup \cdots \cup V'_{k, l}.$$ By \eqref{pre-cont}, one has
$$\bigoplus_{e\in\mathcal E_{k-1}}\bigoplus_{i, j=1}^{n_k}\eta(e)_{i, j}(x) \subseteq V_k',\quad x\in X_k^{(0)}.$$
Also note that 
\begin{equation}\label{dim-control}
\mathrm{dim}(V'_k) \leq \frac{n_k}{n_0} - 2.
\end{equation}

For each element $e \in \mathcal E_{k-1}\subseteq A_{k-1}$, consider $\eta(e) \in C_k^{(0)}$. Since $V_{k}'$ is a polyhedron, it is a neighborhood retraction; hence there is an open set $U \supseteq X_k^{0}$ such that there is an extension of $\eta(e)$ to $U$, denoted by $e'$, such that 
\begin{equation}\label{contain}
\bigoplus_{e\in\mathcal E_{k-1}}\bigoplus_{i, j=1}^{n_k}e'_{i, j}(x) \subseteq V_k',\quad x\in U.
\end{equation}
Pick a continuous function $\chi_k: X_k \to [0, 1]$ such that $$\chi^{-1}(1) = X_k^{(0)}\quad \mathrm{and}\quad \chi^{-1}(0) =X_k\setminus U.$$ Define
$$\tilde{e}: X_k \ni x \mapsto \left\{ \begin{array}{ll} \chi(x) {e'}(x), & x\in U; \\ 0, & x\notin U. \end{array}\right.$$
Set $$\tilde{\mathcal E}_{k}=\{e\oplus \tilde{e} \in A_{k}: e\in \mathcal E_{k-1}\},$$
and
$$\mathcal E_k'=\{0 \oplus (1-\chi) e_{i, j}^{(k)}\in A_k: 1\leq i, j \leq n_k\},$$
where $e_{i, j}^{(k)}$ are matrix units of $\mathrm{M}_{n_k}(\Comp) \subseteq \mathrm{M}_{n_k}(\mathrm{C}(X_k))$.
Then
$$\mathcal E_k:=\tilde{\mathcal E}_{k} \cup  \mathcal E_k'$$ satisfies \eqref{full-irrd}, \eqref{V-dim}, and \eqref{V-cont} (for $A_k$).

Indeed, since $\mathcal E'_k\subseteq \mathcal E_k$, \eqref{full-irrd} is satisfied. 
Define $$V_k=\{(tv, (1-t)v) : t\in[0, 1],\ v\in V_k' \}.$$ Then, by \eqref{contain},
$$\Phi_{\mathcal E_k, k}(\mathcal E_k) \subseteq V_k$$ and by \eqref{dim-control},
$$\mathrm{dim}(V_k) = \mathrm{dim}(V_k') + 1 \leq \frac{n_k}{n_0} - 1 = \frac{n_k - n_0}{n_0}.$$
So, \eqref{V-dim} and \eqref{V-cont} are satisfied (for $A_k$). By the induction, the desired finite sets $\mathcal E\subseteq A$ exist. 
\end{proof}

\begin{rem}
The estimation which is needed later in Theorem \ref{DSH-comp} on the dimension of $V_i$ is actually $$\mathrm{dim}(V_{i}) \leq \frac{n_i}{n_0}.$$ But the stronger version \eqref{V-dim} is needed for the induction argument.
\end{rem}

\begin{lem}\label{irrd-rep}
Consider $n \times n$ matrices 
$$(a^{(1)}_{i, j})_{i, j=1}^n, ... , (a^{(m)}_{i, j})_{i, j=1}^n, \alpha e_{s, t},\quad 1\leq s, t \leq n,$$ 
and 
$$(b^{(1)}_{i, j})_{i, j=1}^n, ... , (b^{(m)}_{i, j})_{i, j=1}^n, \beta e_{s, t},\quad 1\leq s, t \leq n,$$
where $\alpha, \beta\in \Real\setminus\{0\}$ and $e_{s, t}$ is the matrix with $(s, t)$-entry $1$ and all other entries $0$. 

If there is a unitary $U\in \mathrm{M}_n(\Comp)$ satisfying
$$U^*(a^{(k)}_{i, j})U = (b^{(k)}_{i, j})\quad\mathrm{and}\quad U^*(\alpha e_{s, t})U = \beta e_{s, t},\quad 1\leq k \leq m, \ 1\leq s, t \leq n,$$
then
$$ a_{s, t}^{(k)} = b_{s, t}^{(k)}\quad\mathrm{and}\quad \alpha = \beta,\quad \quad 1\leq k \leq m, \ 1\leq s, t \leq n.$$
\end{lem}
\begin{proof}
Note that
$$\alpha I_n = U^*(\alpha I_n)U = U^*\sum_{i=1}^n\alpha e_{i, i}U = \sum_{i=1}^nU^*(\alpha e_{i, i})U = \sum_{i=1}^n\beta e_{i, i} = \beta I_n,$$ and hence $\alpha = \beta$. Consider a pair of matrices $(a_{i, j}^{(k)}), (b_{i, j}^{(k)})$. Then, for any $1\leq s, t\leq n$, 
\begin{eqnarray*}
&&(\alpha a_{s, t}^{(k)} \beta) e_{s, t} \\
 & = & \alpha a_{s, t}^{(k)} (U^* \alpha e_{s, t}U)  =  U^*(\alpha^2 a_{s, t}^{(k)} e_{s, t})U = U^*(\alpha e_{s, s}(a^{(k)}_{i, j})_{i, j=1}^n\alpha e_{t, t})U \\
& = & \beta e_{s, s}(b^{(k)}_{i, j})_{i, j=1}^n\beta e_{t, t} = (\beta^2 b_{s, t}^{(k)}) e_{s, t}.
\end{eqnarray*} 
Since $\alpha=\beta\neq 0$,  one has that $a_{s, t}^{(k)} = b_{s, t}^{(k)}$, as desired.
\end{proof}

\begin{lem}[Lemma 4.3 of \cite{EN-MD0}]\label{cont-field}
Let $X$ be a second countable locally compact Hausdorff space, and let $S$ be a sub-C*-algebra of $\mathrm{M}_n(\mathrm{C_0}(X))$.

Suppose that there exist a topological space $\Delta$ and a surjective continuous map $\xi: X\to\Delta$ such that 
\begin{enumerate}
\item\label{cond-sep}  for any $x_1, x_2\in X$, $$\xi(x_1)=\xi(x_2)$$ if and only if $$\textrm{$\pi_{x_1}|_S$ is unitarily equivalent to $\pi_{x_2}|_S$},$$ 
where $\pi_x$ is the standard irreducible representation of $\mathrm{M}_n(\mathrm{C_0}(X))$ at $x\in X$, 
\item\label{cond-cont} for any sequence $x_i, i=1, 2, ...,$ in $X$, any $x$ in $X$, and any $g\in S$, if $$\xi(x_i)\to\xi(x)\quad \textrm{as}\  i\to\infty,$$ then $$g(x_i) \to g(x)\quad\mathrm{as}\  i\to\infty,$$ and
\item\label{cond-full} $\pi_x(S)=\mathrm{M}_n(\mathbb C)$, for any $x\in X$.
\end{enumerate}

Then there is an isomorphism $\phi: \mathrm{M}_n(\mathrm{C}_0(\Delta)) \to S$. Moreover, under this isomorphism, if $X^{(0)}$ is a closed subset of $X$ and $\pi_{X^{(0)}}$ is the restriction map, there is a commutative diagram
$$
\xymatrix{
\mathrm{M}_{n}(\mathrm{C}_0(X)) \ar[d]^{\pi_{X^{(0)}}} & S\ar[d]^{\pi_{X^{(0)}}} \ar@{_{(}->}[l]  & \mathrm{M}_{n}(\mathrm{C}_0(\Delta)) \ar[d]^{\pi_{\Delta^{(0)}}}  \ar[l]_-\phi \\
\mathrm{M}_{n}(\mathrm{C}(X^{(0)})) & \pi_{X^{(0)}}(S) \ar@{_{(}->}[l]  & \mathrm{M}_{n}(\mathrm{C}(\Delta^{(0)})) \ar[l] 
},
$$
where $\Delta^{(0)} = \xi(X^{(0)})$ and $\pi_{\Delta^{(0)}}$ is the restriction map.

\end{lem}
\begin{proof}
For each $f\in S$, define a function $\tilde{f}: \Delta\to \mathrm{M}_n(\Comp)$ by
$$\tilde{f}(z)=f(x),\quad\textrm{if $\xi(x)=z$}.$$
By Condition (2), 
$\tilde{f}$ is well defined, and $\tilde{f}$ is continuous. Moreover, $\tilde{f}$ vanishes at infinity. To see this, note that, if $z_i\in\Delta$ with $z_i\to\infty$, then, since $\xi$ is surjective, there are $x_i\in X$ with $\xi(x_i)=z_i$. 
Then $x_i\to\infty$. Otherwise, there is a subsequence, say $(x_{i_k})$, converging to a point $x\in X$. Since $\xi$ is continuous, one has that $$z_{i_k}=\xi(x_{i_k})\to\xi(x),$$ which contradicts the assumption $z_i\to\infty$. Hence $$\tilde{f}(z_i)=f(x_i)\to 0,$$ and $\tilde{f}\in \mathrm{M}_n(\mathrm{C_0}(\Delta))$.

Moreover, it is clear that the map $f\mapsto\tilde{f}$ is an injective homomorphism, and thus one can regard $S$ as a sub-C*-algebra of $\mathrm{M}_n(\mathrm{C_0}(\Delta))$. It follows from Conditions (1) and (3) 
that $S$ is a rich sub-C*-algebra of $\mathrm{M}_n(\mathrm{C_0}(\Delta))$ in the sense of Dixmier (11.1.1 of \cite{Dix-book}), and therefore $S=\mathrm{M}_n(\mathrm{C}_0(\Delta))$ by Proposition 11.1.6 of \cite{Dix-book} (or by Theorem 7.2 of \cite{Kap-SCalg}).

Note that the above construction also induces an isomorphism $$\pi_{X^{(0)}}(S) \cong \mathrm{M}_n(\mathrm{C}(\Delta^{(0)}))$$ by $f \mapsto \tilde{f}$ with $\tilde{f}(\xi(x)) = f(x)$.   Also note that for any $f\in \mathrm{M}_{n}(\mathrm{C}(\Delta))$, and any $x\in X^{(0)}$, one has
$$\pi_{X^{(0)}}(\phi(f))(x) = \phi(f)(x) = f(\xi(x)) = \pi_{\Delta^{(0)}}(f)(\xi(x)).$$ Thus, the diagram commutes.
\end{proof}

The following is the main result of this section. Its proof is in the same line as that of Theorem 4.5 of \cite{EN-MD0}: One considers the sub-C*-algebra generated by the given elements $a$, $b$, and the finite set $\mathcal E_K$ obtained from Lemma \ref{qsudo-system}; then the resulting sub-C*-algebra actually has the dimension gap property of Proposition \ref{comp-dim-rank}, and the comparison between $a$ and $b$ follows. 

\begin{thm}\label{DSH-comp}
Let $A$ be a separable RSH algebra with a fixed decomposition 
$$\left[\cdots\left[\left[ C_0\oplus_{C_1^{(0)}} C_1 \right]\oplus_{C_2^{(0)}}C_2\right] \cdots\right]\oplus_{C_K^{(0)}} C_K,$$ and let $a, b$ be positive elements of  $A$. Write $$C_i=\mathrm{M}_{n_i}(\mathrm{C}(X_i))\quad \textrm{and}\quad C_i^{(0)} = \mathrm{M}_{n_i}(\mathrm{C}(X^{(0)}_i))$$ with $X^{(0)}_i$ a closed subset of $X_i$. 
Assume
\begin{enumerate}
\item $n_0 < n_1 < \cdots < n_K$,
\item $a$ and $b$ are diagonal matrices on each $X_i$,
\item all the maps $\eta_i: A_{i} \to C_{i+1}^{(0)}$, $i=0, ..., K-1$, are of diagonal type, where $$A_i = \left[\cdots\left[\left[ C_0\oplus_{C_1^{(0)}} C_1 \right]\oplus_{C_2^{(0)}}C_2\right] \cdots\right]\oplus_{C_i^{(0)}} C_i,$$
\item for each $i=0, 1, ..., K$, and each $x\in X_i\setminus X_i^{(0)}$, $$\mathrm{rank}(a(x)) < \frac{1}{4}\mathrm{rank}(b(x))\quad\mathrm{and}\quad \frac{1}{n_0} < \frac{\mathrm{rank}(b(x))}{4n_i}.$$
\end{enumerate}
Then $a\precsim b$.
\end{thm}
\begin{proof}
Write $$a=(a_0, a_1, ..., a_K)\quad\mathrm{and}\quad b=(b_0, b_1, ..., b_K)$$ with $a_i, b_i \in C_i$, $i=0, 1, ..., K$.

Let us construct sub-C*-algebras
$$D_i \cong \mathrm{M}_{n_i}(\mathrm{C}(\Delta_i)) \subseteq \mathrm{M}_{n_i}(\mathrm{C}(X_i)) =  C_i, \quad i=0, 1, ..., K, $$ 
together with closed subsets $\Delta_i^{(0)} \subseteq \Delta_i$ ($\Delta_0^{(0)} = \varnothing$) satisfying 
\begin{equation}\label{chain-cont}
\eta_i( \left[\cdots\left[\left[ D_0\oplus_{D_1^{(0)}} D_1 \right]\oplus_{D_2^{(0)}}D_2\right] \cdots\right]\oplus_{D_{i-1}^{(0)}} D_{i-1} )\subseteq \pi_{\Delta_{i}^{(0)}}(D_{i}),
\end{equation}
where $D_i^{(0)} = \mathrm{M}_{n_i}(\mathrm{C}(\Delta_i^{(0)}))$,
such that
\begin{equation}\label{element-in}
a_i, b_i\in D_i,\quad i=0, 1., ..., K,
\end{equation}
and
if $$\{\mathrm{rank}(b(x)): x\in \Delta_i\} = \{r_{i, 0} < r_{i, 1} < \cdots < r_{i, s_i}\}$$ and define
$$W_{i, j} = \{x \in \Delta_i: \mathrm{rank}(b(x)) \leq r_{i, j}\}\quad\mathrm{and}\quad Z_{i, j} = \{x \in \Delta_i: \mathrm{rank}(b(x)) = r_{i, j}\},$$ where $j=0, 1, ..., s_i,$ then
\begin{equation}\label{dim-gap}
\mathrm{rank}(a(x)) + \frac{\mathrm{dim}(W_{i, j}) - 1}{2} < \mathrm{rank}(b(x))=r_{i, j},\quad x\in Z_{i, j}.
\end{equation}
Then, by  \eqref{chain-cont} and \eqref{element-in},
$$ S:=\left[\cdots\left[\left[ D_0\oplus_{D_1^{(0)}} D_1 \right]\oplus_{D_2^{(0)}}D_2\right] \cdots\right]\oplus_{D_{K}^{(0)}} D_{K}$$ is a sub-C*-algebra of $A$ with $a, b \in S$; by \eqref{dim-gap} and Proposition \ref{comp-dim-rank}, one has $a \precsim b$, and this proves the theorem.

Let us construct the desired sub-C*-algebras $D_i$ recursively. By Lemma \ref{qsudo-system}, there exist finite sets $\mathcal E_k \subseteq A_k$, where 
$$A_k = \left[\cdots\left[\left[ C_0\oplus_{C_1^{(0)}} C_1 \right]\oplus_{C_2^{(0)}}C_2\right] \cdots\right]\oplus_{C_k^{(0)}} C_k,\quad k=0, 1, ..., K,$$ 
such that 
\begin{equation}\label{bd-ext}
\eta_k(\mathcal E_{k-1}) \subseteq \pi_{X_k^{(0)}}(\mathcal E_k),\quad k=1, 2, ..., K;
\end{equation}
\begin{equation}\label{full-irrd-in-use}
\textrm{C*}(\pi_x(\mathcal E_k)) = \mathrm{M}_{n_i}(\Comp),\quad x\in X_i\setminus X_i^{(0)},\ i=0, 1, ..., k;
\end{equation}
and there is a polyhedron 
$
V_{k}\subseteq \Real^{n_k^2\abs{\mathcal E_k}} 
$
such that 
\begin{equation}\label{V-dim-in-use}
\mathrm{dim}(V_{k}) \leq \frac{n_k-n_0}{n_0} 
\end{equation}
and
\begin{equation}\label{V-cont-in-use}
\Phi_{\mathcal E_k, k}(X_{k}) \subseteq V_k.
\end{equation}

Define 
$$D_0:=\textrm{C*}(\{a_0, b_0\}\cup \mathcal E_0 )\subseteq \mathrm{M}_{n_0}(\mathrm{C}(X_0)).$$ 
It clearly satisfies \eqref{element-in} (with $i=0$). \eqref{chain-cont} is also satisfied with $\Delta_0^{(0)} = \varnothing$.
Set $$\mathcal G_0=\{a, b\} \cup \mathcal E_0$$
and consider the map 
\begin{equation}\label{defn-PHI-0} 
\Phi_{\mathcal G_0, 0}: X_0 \ni x \mapsto (\bigoplus_{i=1}^{n_0}(a_0)_{i, i}(x))\oplus  (\bigoplus_{i=1}^{n_0}(b_0)_{i, i}(x)) \oplus  (\bigoplus_{e\in\mathcal E_0}\bigoplus_{i, j=1}^{n_0}(e)_{i, j}(x)) \in \Real^{2n_0+|\mathcal E_0|n^2_0}
\end{equation} 
Set $$\Delta_0 = \Phi_{\mathcal G_0, 0}(X_0).$$ Then,  for any $x\in X_0$, the restriction of $\pi_x$ to $D_0$ is still irreducible, and by Lemma \ref{irrd-rep}, one has that for any $x_1, x_2\in X_0$,
$$\textrm{$\pi_{x_1}|_{D_0}$ is unitarily equivalent to $\pi_{x_2}|_{D_0}$} \Longleftrightarrow \Phi_{\mathcal G_0, 0}(x_1) = \Phi_{\mathcal G_0, 0}(x_2).$$
By Lemma \ref{cont-field}, $$D_0\cong \mathrm{M}_{n_0}(\mathrm{C}(\Delta_0)).$$

Write $$W'_{0, j} = \{x\in X_0: \mathrm{rank}(b_0(x)) \leq r_{0, j}\}\quad\mathrm{and}\quad Z'_{0, j} = \{x \in X_0: \mathrm{rank}(b_0(x)) = r_{0, j}\},$$ where $j=0, 1, ..., s_0$, 
and $$W_{0, j} = \{x\in \Delta_0: \mathrm{rank}(b_0(x)) \leq  r_{0, j}\}\quad\mathrm{and}\quad Z_{0, j} = \{x \in \Delta_0: \mathrm{rank}(b_0(x)) = r_{0, j}\},$$ where $j=0, 1, ..., s_0.$ 
Then $$W_{0, j} = \Phi_{{\{a, b\}\cup\mathcal E_0}, 0}(W'_{0, j})\quad\mathrm{and}\quad Z_{0, j} = \Phi_{{\{a, b\}\cup\mathcal E_0}, 0}(Z'_{0, j}),\quad j=0, 1, ..., s_0.$$

Since $a_0$ and $b_0$ are diagonal and $$\mathrm{rank}(a(x)) < \frac{1}{4}\mathrm{rank}(b(x)),$$ by the construction of $\Phi_{{\{a, b\}\cup\mathcal E_0}, 0}$ (\eqref{defn-PHI-0}), one has
\begin{eqnarray*}
W_{0, j} & \subseteq & \{(y_1, ..., y_{n_0}): \textrm{at most $r_{0, j}$ many of coordinates are not $0$}\} \\
& & \times\{ (y_1, ..., y_{n_0}): \textrm{at most $\frac{r_{0, j}}{4}$ many of coordinates are not $0$} \} \times V_{0}.
\end{eqnarray*}
Hence
\begin{eqnarray}
\mathrm{dim}(W_{0, j}) & \leq &  r_{0, j} + \frac{1}{4} r_{0, j} + 0
\end{eqnarray}
and for any $x\in Z_{0, j}$,
\begin{eqnarray*}
 \mathrm{rank}(a(x)) + \frac{\mathrm{dim}(W_{0, j})-1}{2} 
& \leq & \frac{1}{4} r_{0, j} + \frac{r_{0, j} + \frac{1}{4} r_{0, j} -1}{2} \\
& = & \frac{7}{8}r_{0, j} -\frac{1}{2} < r_{0, j} = \mathrm{rank}(b(x)).
\end{eqnarray*}
Thus, the C*-algebra $D_0$ satisfies \eqref{dim-gap} (with $i=0$).

Let us assume that $D_0, D_1, ..., D_{k-1}$ are constructed to satisfy \eqref{chain-cont} (with $i=k-1$), \eqref{element-in} (with $i=k-1$), and \eqref{dim-gap} (with $i=k-1$), and also assume that
\begin{eqnarray}\label{generator-k-1}
&& \left[\cdots\left[ D_0\oplus_{D_1^{(0)}} D_1 \right] \cdots\right]\oplus_{D_{k-1}^{(0)}} D_{k-1} \\
& = & \textrm{C*}\{\{(a_0, ..., a_{k-1}), (b_0, ..., b_{k-1})\}\cup\mathcal E_{k-1}\}\subseteq A_{k-1}. \nonumber
 \end{eqnarray}

Let us construct $D_k$.
Define $$D_k=\textrm{C*}(\{a_k, b_k\}\cup \mathcal (\mathcal E_k|_{X_k}) \cup \{e_{s, t}: 1\leq s, t \leq n_k\}) \subseteq \mathrm{M}_{n_k}(\mathrm C(X_k)) = C_k,$$
where $e_{s, t}, 1\leq s, t \leq n_k$ are (constant) matrix units of $\mathrm{M}_{n_k}(\Comp) \subseteq \mathrm{M}_{n_k}(\mathrm C(X_k))$.
Set
$$\mathcal G_k=\{\{a_k, b_k\}\cup \mathcal (\mathcal E_k|_{X_k}) \cup \{e_{s, t}: 1\leq s, t \leq n_k\}\}$$
and consider
\begin{eqnarray}\label{defn-PHI-k} 
 \Phi_{\mathcal G_k, k}: X_k &\ni & x \nonumber \\
 &\mapsto& (\bigoplus_{i=1}^{n_k}(a_k)_{i, i}(x))\oplus  (\bigoplus_{i=1}^{n_k}(b_k)_{i, i}(x)) \oplus  (\bigoplus_{e\in\mathcal E_k|_{X_k}\cup \{e_{s, t}\}}\bigoplus_{i, j=1}^{n_k}(e)_{i, j}(x))\\
 &  \in & \Real^{2n_k+\abs{\mathcal E_k|_{X_k}\cup \{e_{s, t}\}}n^2_k} \nonumber 
\end{eqnarray} 
and
set
$$\Delta_k = \Phi_{\mathcal G_k, k} (X_k) \quad\mathrm{and}\quad \Delta_k^{(0)} = \Phi_{\mathcal G_k, k} (X^{(0)}_k).$$
Then, for any $x\in X_k$, the restriction of $\pi_x$ to $D_k$ is still irreducible, and by Lemma \ref{irrd-rep}, for any $x_1, x_2\in X_k$,
$$\textrm{$\pi_{x_1}|_{D_k}$ is unitarily equivalent to $\pi_{x_2}|_{D_k}$} \Longleftrightarrow \Phi_{\mathcal G_k, k}(x_1) = \Phi_{\mathcal G_k, k}(x_2).$$
By Lemma \ref{cont-field}, one has that
$D_k \cong \mathrm{M}_{n_k}(\mathrm{C}(\Delta_k)$
and
$$
\xymatrix{
\mathrm{M}_{n_k}(\mathrm{C}(X_k)) \ar[d]^{\pi_{X_k^{(0)}}} & D_k \ar[d]^{\pi_{X_k^{(0)}}} \ar@{_{(}->}[l] & \mathrm{M}_{n_k}(\mathrm{C}(\Delta_k)) \ar[d]^{\pi_{\Delta_k^{(0)}}}  \ar[l]_-{\cong} \\
\mathrm{M}_{n_k}(\mathrm{C}(X_k^{(0)})) & \pi_{X_k^{(0)}}(D_k) \ar@{_{(}->}[l] & \mathrm{M}_{n_k}(\mathrm{C}(\Delta_k^{(0)})) \ar[l]_-{\cong}
}
$$
By \eqref{bd-ext},
$$\eta_{k}(\mathcal E_{k-1}) \subseteq \pi_{X_k^{(0)}}(\mathcal E_k).$$
Then, by \eqref{generator-k-1} and $\mathcal E_k \in D_k$, 
one has
$$\eta_k( \left[\cdots\left[\left[ D_0\oplus_{D_1^{(0)}} D_1 \right]\oplus_{D_2^{(0)}}D_2\right] \cdots\right]\oplus_{D_{k-1}^{(0)}} D_{k-1} ) \subseteq \pi_{\Delta_k^{(0)}}(D_{k}),$$
and hence \eqref{chain-cont} is satisfied (with $i=k$).

Write $$W'_{k, j} = \{x\in X_k: \mathrm{rank}(b_k(x)) \leq r_{k, j}\}\quad\mathrm{and}\quad Z'_{k, j} = \{x \in X_k: \mathrm{rank}(b_k(x)) = r_{k, j}\},$$ where $j=0, 1, ..., s_k$ and $$W_{k, j} = \{x\in \Delta_k: \mathrm{rank}(b_k(x)) \leq r_{k, j}\}\quad\mathrm{and}\quad Z_{k, j} = \{x \in \Delta_k: \mathrm{rank}(b_k(x)) = r_{k, j}\},$$ where $j=0, 1, ..., s_k.$ Then $$W_{k, j} = \Phi_{\mathcal G_k, k}(W'_{k, j})\quad\mathrm{and}\quad Z_{k, j} = \Phi_{\mathcal G_k, k}(Z'_{k, j}),\quad j=0, 1, ..., s_k.$$

Since $a_k$ and $b_k$ are diagonal and $$\mathrm{rank}(a(x)) < \frac{1}{4}\mathrm{rank}(b(x)),$$ by the construction of $\Phi_{{\{a, b\}\cup\mathcal E_k\cup\{e_{s, t}\}}, k}$ (\eqref{defn-PHI-k}), one has
\begin{eqnarray*}
W_{k, j} & \subseteq & \{(y_1, ..., y_{n_k}): \textrm{at most $r_{k, j}$ many of coordinates are not $0$}\} \\
& & \times\{ (y_1, ..., y_{n_k}): \textrm{at most $\frac{r_{k, j}}{4}$ many of coordinates are not $0$} \} \\
& &  \times V_{k} \times \{v_e\}.
\end{eqnarray*}
Hence
\begin{eqnarray}
\mathrm{dim}(W_{k, j}) & \leq &  r_{k, j} + \frac{1}{4} r_{k, j} + \mathrm{dim}(V_{k})\leq \frac{5}{4} r_{k, j} + \frac{n_k-n_0}{n_0}
\end{eqnarray}
and for any $x\in Z_{k, j}$,
\begin{eqnarray*}
\mathrm{rank}(a(x)) + \frac{\mathrm{dim}(W_{k, j})-1}{2} & \leq & \frac{1}{4} r_{k, j} + \frac{\frac{5}{4} r_{k, j} + \frac{n_k}{n_0} -2}{2} \\
& \leq & \frac{7}{8}r_{k, j} + \frac{1}{8} r_{k, j} - 1 \\
& < & r_{k, j} =  \mathrm{rank}(b(x)).
\end{eqnarray*}
Thus, the C*-algebra $D_k$ satisfies \eqref{dim-gap} (with $i=k$). 

Put $$D_k^{(0)} = \mathrm{M}_{n_k}(\mathrm{C}(\Delta_k^{(0)})),$$ and consider $$ \left[\left[\cdots\left[ D_0\oplus_{D_1^{(0)}} D_1 \right] \cdots\right]\oplus_{D_{k-1}^{(0)}} D_{k-1} \right] \oplus_{D_k^{(0)}} D_k\subseteq A_k.$$ 
For the induction, we also need to show 
$$\left[ \left[\cdots\left[ D_0\oplus_{D_1^{(0)}} D_1 \right] \cdots\right]\oplus_{D_{k-1}^{(0)}} D_{k-1} \right ]\oplus_{D_k^{(0)}} D_k = \textrm{C*}(\{(a_0, ..., a_k), (b_0, ..., b_k)\} \cup \mathcal E_k).$$ 
(Note that $$\mathcal E_k\subseteq \left[ \left[\cdots\left[ D_0\oplus_{D_1^{(0)}} D_1 \right] \cdots\right]\oplus_{D_{k-1}^{(0)}} D_{k-1} \right] \oplus_{D_k^{(0)}} D_k,$$ since $\mathcal E_k|_{(X_0, ..., X_{k-1})} = \mathcal E_{k-1}$). 

Indeed, set 
$$B= \textrm{C*}(\{(a_0, ..., a_k), (b_0, ..., b_k)\} \cup \mathcal E_k),$$ and let us show that actually $$B=\left[ \left[\cdots\left[ D_0\oplus_{D_1^{(0)}} D_1 \right] \cdots\right]\oplus_{D_{k-1}^{(0)}} D_{k-1} \right]\oplus_{D_k^{(0)}} D_k.$$

Pick any point $$x\in X_i\setminus X_i^{(0)},\quad 0\leq i\leq k.$$ Since $\mathcal E_k \subseteq B$ and \eqref{full-irrd-in-use}, the restriction of $\pi_x$ to $B$ is still irreducible (with the same dimension), and hence any irreducible representation of $B$ or $$\left[ \left[\cdots\left[ D_0\oplus_{D_1^{(0)}} D_1 \right] \cdots\right]\oplus_{D_{k-1}^{(0)}} D_{k-1} \right] \oplus_{D_k^{(0)}} D_k$$ actually is a restriction of some $\pi_x$. In particular, the restriction of any  irreducible representation of $$\left[ \left[\cdots\left[ D_0\oplus_{D_1^{(0)}} D_1 \right] \cdots\right]\oplus_{D_{k-1}^{(0)}} D_{k-1} \right] \oplus_{D_k^{(0)}} D_k$$ to $B$ is still irreducible.


Let $\pi_{x_1}$ and $\pi_{x_2}$ be irreducible representations of $B$ which are not equivalent. If $x_1$ and $x_2$ are in different components, say, in $X_i\setminus X_i^{(0)}$ and $X_j\setminus X_j^{(0)}$, $i\neq j$, respectively,  then the the restrictions of $\pi_{x_1}$ and $\pi_{x_2}$  to $B$ are not equivalent since they have different dimensions. 

Assume that $x_1, x_2\in X_i\setminus X_i^{(0)}$ with $i< k$. In particular, $\pi_{x_1}$ and $\pi_{x_2}$ are irreducible representations of $$ \left[\cdots\left[ D_0\oplus_{D_1^{(0)}} D_1 \right] \cdots\right]\oplus_{D_{k-1}^{(0)}} D_{k-1} = \textrm{C*}\{\{(a_0, ..., a_{k-1}), (b_0, ..., b_{k-1})\}\cup\mathcal E_{k-1}\}.$$ Since $\pi_{x_1}$ and $\pi_{x_2}$ are not equivalent, one has that
$$ \Phi_{\{a, b\} \cup \mathcal E_{k-1}, i}(x_1) \neq \Phi_{\{a, b\} \cup \mathcal E_{k-1}, i}(x_2).$$ Since the restriction of $\mathcal E_k$ to $X_i$ is the same as the restriction of $\mathcal E_{k-1}$ to $X_i$, one has
$$ \Phi_{\{a, b\} \cup \mathcal E_{k}, i}(x_1) \neq \Phi_{\{a, b\} \cup \mathcal E_{k}, i}(x_2).$$ Hence, by Lemma \ref{irrd-rep}, the restrictions of $\pi_{x_1}$ and $\pi_{x_2}$ to $B$ are not equivalent. 

Assume $x_1, x_2\in X_k\setminus X_k^{(0)}$. If the restriction of $\pi_{x_1}$ and $\pi_{x_2}$ to $B$ are equivalent, then, by Lemma \ref{irrd-rep},
$$\Phi_{\{a, b\} \cup \mathcal E_k|_{X_k}, k} (x_1) = \Phi_{\{a, b\} \cup \mathcal E_k|_{X_k}, k} (x_2).$$
Since matrix units $\{e_{s, t}\}$ are constant functions, one then has
$$\Phi_{\{a, b\} \cup \mathcal E_k|_{X_k}\cup\{e_{s, t}\}, k} (x_1) = \Phi_{\{a, b\} \cup \mathcal E_k|_{X_k}\cup\{e_{s, t}\}, k} (x_2).$$ So, $\pi_{x_1}$ and $\pi_{x_2}$ are equivalent as representations of $D_k$ and hence are equivalent, which contradicts the assumption. Therefore, the restrictions of $\pi_{x_1}$ and $\pi_{x_2}$ to $B$ are not equivalent. Hence $B$ is a rich sub-C*-algebra, and therefore, $$B = \left[\left[\cdots\left[ D_0\oplus_{D_1^{(0)}} D_1 \right] \cdots\right]\oplus_{D_{k-1}^{(0)}} D_{k-1} \right] \oplus_{D_k^{(0)}} D_k$$
by Proposition 11.1.6 of \cite{Dix-book}.

Therefore, by induction, there are C*-algebras $D_0, D_1, ..., D_K$ satisfying \eqref{chain-cont}, \eqref{element-in} and \eqref{dim-gap}, and hence $a \precsim b$, as desired.
\end{proof}

The following are several corollaries of Theorem \ref{DSH-comp} 
\begin{cor}\label{comp-open-sets-Z}
Let $X$ be compact metrizable, and let $\sigma: X\to X$ be a minimal free homeomorphism. Suppose $E, F\subseteq X$ are open sets satisfying
\begin{equation}\label{mea-gap}
\mu(E) < \frac{1}{4}\mu(F), \quad \mu \in \mathcal M_1(X, \sigma).
\end{equation} 
Then
$$\varphi_E \precsim \varphi_F$$ in $\mathrm{C}(X) \rtimes_\sigma \Int$.
\end{cor}
\begin{proof}
Since $\sigma$ is minimal, there is $\delta>0$ such that $$\mu(F) > \delta,\quad \mu\in\mathcal M_1(X, \sigma).$$

Let $\eps>0$ be arbitrary.  One asserts that there is $N_0>\frac{4}{\delta}$ such that for any $x\in X$ and any $N>N_0$, one has
\begin{eqnarray}\label{pre-mea-contradic}
&& \abs{\{0 \leq n \leq N-1: (\varphi_E-\eps)_+(\sigma^n(x)) > 0\}} <  \frac{1}{4} \abs{\{0 \leq n \leq N-1: \varphi_F(\sigma^n(x)) > 0\}} \nonumber
\end{eqnarray}
and
$$\frac{\delta}{4} < \frac{1}{4N} \abs{\{0 \leq n \leq N-1: \varphi_F(\sigma^n(x)) > 0\}}.$$

Indeed, if the assertion were not true, there are $(k_n)\subseteq \mathbb N$ and $(x_{k_n})\subseteq X$ such that $k_n \to \infty$ as $n \to\infty$, and for all $n$, one has
\begin{eqnarray}\label{mea-contradic}
&& \frac{1}{k_n} \abs{\{0 \leq i \leq k_n-1: (\varphi_E-\eps)_+(\sigma^i(x_{k_n})) > 0\}} \\
&\geq & \frac{1}{4k_n} \abs{\{0 \leq i \leq k_n-1: \varphi_F(\sigma^i(x_{k_n})) > 0\}} \nonumber
\end{eqnarray}
or
\begin{equation}\label{frq-low-bd}
\frac{\delta}{4} \geq \frac{1}{4k_n} \abs{\{0 \leq i \leq k_n-1: \varphi_F(\sigma^i(x_{k_n})) > 0\}}.
\end{equation}

Consider the discrete probability measures $$\mu_n:=\frac{1}{k_n}\sum_{i=0}^{k_n-1} \delta_{\sigma^i(x_{k_n})}, \quad n=1, 2, ...,$$ where $\delta_x$ is the Dirac measure concentrated at $x$. Pick an accumulation point $\mu_\infty$ in the weak*-topology. Note that $\mu_\infty \in\mathcal{M}_1(X, \sigma)$. Passing to a subsequence, one assumes that $\mu_n \to\mu_\infty$.

Assume that \eqref{mea-contradic} holds for infinitely many $n$, pick a closed set $E'$ such that $$ \{x\in E: \varphi_{E}(x) \geq \eps\} \subseteq E' \subseteq E$$ and then
\begin{eqnarray*}
\mu_\infty(F) &\leq & \liminf_{n\to\infty} \mu_n(F)\quad\quad\quad \textrm{($F$ is open)} \\
&  \leq & 4 \liminf_{n\to\infty} \mu_n(E')\quad\quad\quad \textrm{(by \eqref{mea-contradic})}\\
& \leq & 4 \limsup_{n\to\infty} \mu_n(E') \\
& \leq & 4 \mu_\infty(E')\leq 4 \mu_\infty(E), \quad\quad\quad \textrm{($E'$ is closed)}
\end{eqnarray*}
which contradicts to \eqref{mea-gap}. 

Assume that \eqref{frq-low-bd} holds for infinitely many $n$. Since $F$ is open, then $$\mu_\infty(F) \leq \liminf_{k\to\infty} \mu_n(F) \leq \delta,$$ which contradicts the choice of $\delta$. This proves the assertion.

Consider the C*-algebra $A_Y\subseteq \mathrm{C}(X)\rtimes_\sigma\Int$, where $Y$ is a closed subset with nonempty interior. By Theorem \ref{Lin-Sub}, $A_Y$ is an RSH algebra with diagonal maps, and the canonical RSH decomposition of $A_Y$ satisfies Conditions (1), (2), and (3) of Theorem \ref{DSH-comp}. With $Y$ sufficiently small, one can assume that the heights of the Rokhlin towers in the decomposition of $A_Y$ are at least $N_0$ (so that $\frac{1}{N_0} < \frac{\delta}{4}$), and therefore, by the assertion,  Conditions (4) of Theorem \ref{DSH-comp} is also satisfied. Thus, it follows from Theorem \ref{DSH-comp} that $(\varphi_E-\eps)_+ \precsim \varphi_F$. Since $\eps$ is arbitrary, one has $\varphi_E \precsim \varphi_F$.
\end{proof}

\begin{cor}\label{main-thm-Z}
Let $X$ be a separable compact Hausdorff space, and let  $\sigma: X \to X$ be a minimal free homeomorphism. Then $$\mathrm{rc}(\mathrm{C}(X) \rtimes \Int) \leq\frac{1}{2}\mathrm{mdim}(X, \sigma).$$
\end{cor}
\begin{proof}
By Lemma \ref{LRT-Z}, $(X, \sigma)$ has the (URP). By Corollary \ref{comp-open-sets-Z}, $\mathrm{C}(X) \rtimes\Int$ has $(\frac{1}{4}, 1)$-Cuntz-comparison on open sets.  The statement then follows from Theorem \ref{main-thm}.
\end{proof}

\begin{rem}
Corollary \ref{main-thm-Z} is generalized  in \cite{Niu-MD-Zd} to $\Int^d$-actions.
\end{rem}

\begin{cor}\label{orb-comp}
Let $(X, \Gamma)$ be a free dynamical system, and let $\mathcal G \subseteq X \rtimes \Gamma $ be a small subgroupoid. Let $E, F\subseteq X$ be open sets such that $$\abs{\mathrm{Orbit}_{\mathcal G}(x)\cap E} < \frac{1}{4}\abs{\mathrm{Orbit}_{\mathcal G}(x)\cap F}, \quad x\in X,$$
and
$$\frac{1}{|\mathrm{Orbit}_{\mathcal G}(x_{0})|} < \frac{\abs{\mathrm{Orbit}_{\mathcal G}(x) \cap F)}}{4|\mathrm{Orbit}_{\mathcal G}(x)|},\quad x_0, x \in X.$$ 
Then, $\varphi_E \precsim \varphi_F$ in $\textrm{C*}(\mathcal G)$.
\end{cor}
\begin{proof}
This follows directly from Theorem \ref{diag-SHA} and Theorem \ref{DSH-comp}.
\end{proof}

\begin{cor}\label{orb-comp-measure}
Let $(X, \Gamma)$ is a minimal free dynamical system, where $\Gamma$ is amenable. Assume that $(X, \Gamma)$ has the property that for any finite set $K \subseteq \Gamma$ and any $\eps>0$, there is a small subgroupoid $\mathcal G \subseteq X \rtimes\Gamma$ such that $\mathrm{Orbit}_{\mathcal G}(x)$ is $(K, \eps)$-invariant for any $x\in X$. Then, for any open sets $E, F\subseteq X$ with 
\begin{equation}\label{mea-gap-G}
\mu(E) < \frac{1}{4} \mu(F),\quad \mu\in \mathcal M_1(X, \Gamma),
\end{equation} 
one has that $$\varphi_E \precsim \varphi_F\quad\mathrm{in}\quad \mathrm{C}(X) \rtimes \Gamma.$$  In other words, $\mathrm{C}(X) \rtimes \Gamma$ has $(\frac{1}{4}, 1)$-Cuntz-comparison of open sets.
\end{cor}

\begin{proof}
Without loss of generality, one may assume that $\abs{\Gamma} = \infty$. Otherwise, the C*-algebra $\mathrm{C}(X)\rtimes \Gamma$ is isomorphic to $\mathrm{M}_{\abs{\Gamma}}(\Comp)$, and the statement holds.

Since $\sigma$ is minimal, there is $\delta>0$ such that $$\mu(F) > \delta,\quad \mu\in\mathcal M_1(X, \Gamma).$$

Let $\eps'>0$ be arbitrary. With the same argument as that of Corollary \ref{comp-open-sets-Z}, it follows from \eqref{mea-gap-G} that there exists $(K, \eps)$ such that if $\Gamma_0\subseteq \Gamma$ is $(K, \eps)$-invariant, then for any $x\in X$, one has
$$\frac{1}{\abs{\Gamma_0}} \abs{\{\gamma\in\Gamma_0: (\varphi_E-\eps')_+(x\gamma) > 0\}} < \frac{1}{4\abs{\Gamma_0}} \abs{\{\gamma\in\Gamma_0: \varphi_F(x\gamma) > 0\}}$$
and
$$\frac{1}{\abs{\Gamma_0}} < \frac{\delta}{4} < \frac{1}{4\abs{\Gamma_0}} \abs{\{\gamma \in \Gamma_0: \varphi_F(x\gamma) > 0\}}.$$

If $\mathcal G \subseteq X \rtimes\Gamma$ is a small subgroupoid with all orbits $(K, \eps)$-invariant, it then follows from Corollary \ref{orb-comp} that $(\varphi_E-\eps')_+ \precsim \varphi_F$. Since $\eps'$ is arbitrary, one has $\varphi_E \precsim\varphi_F$.
\end{proof}

\section{Lower semicontinuous set-valued functions and small subgroupoids}

In this section, let us show that if there is an equivariant lower semicontinuous set-valued function on $X$ (in particular, if lower semicontinuous dynamical tiling exists), then there always exists a small subgroupoid associated to this function (Theorem \ref{Shape-to-groupoid}).

\begin{defn}\label{defn-shape-function}
Consider a topological dynamical system $(X, \Gamma)$. A shape function with domain $\Omega$, where $\Omega$ is an open subset of $X$, is a set-valued function $$S: \Omega \to 2^{\Gamma}$$ such that 
\begin{enumerate}
\item $e \in S(x)$, $x \in \Omega$,
\item $S(x)$ is uniformly bounded in the sense that there is a finite set $M \subseteq \Gamma$ such that $$S(x) \subseteq M,\quad x \in \Omega,$$ 
\item the function $S$ is lower semicontinuous in the sense that for any $x\in \Omega$, there is an open neighbourhood $U \ni x$ such that 
$$S(x) \subseteq S(y),\quad y\in U,$$ and
\item the function $S$ is equivariant in the sense that $$S(x\gamma) = S(x)\gamma^{-1},\quad \gamma \in S(x).$$
\end{enumerate}

\end{defn}

\begin{example}
Consider the dynamical system $(X, \sigma)$, where $\sigma: X \to X $ is a minimal homeomorphism. Let $Y\subseteq X$ be a closed set with non-empty interior. For each $x\in X$, define the positive first return time and negative first return time of $x$ to be 
$$J_+(x) =n,\quad \sigma(x), ..., \sigma^{n-1}(x) \notin Y\ \textrm{but}\ \sigma^n(x) \in Y$$ and
$$J_-(x) = n,\quad x,\ \sigma^{-1}(x), ..., \sigma^{-n+1}(x) \notin Y\ \textrm{but}\ \sigma^{-n}(x) \in Y.$$
Then
$$S(x) = \{-J_-(x), -J_-(x)+1, ..., J_+(x) -1  \}$$ is a shape function (with domain $X$).

Actually, let $$\{Z_1, \sigma(Z_1), ..., \sigma^{J_1-1}(Z_1)\}, ..., \{Z_k, \sigma(Z_k), ..., \sigma^{J_k-1}(Z_k)\}$$ be the Rokhlin partition associated to $Y$ (see Section \ref{RSH-Z}). Then $$ S(x) =  \{0, 1, ..., J_k-1\}-i,\quad\textrm{if $x\in \sigma^{i}(Z_k)$},\ 0\leq i\leq J_k-1.$$
\end{example}


\begin{thm}\label{Shape-to-groupoid}
Let $(X, \Gamma)$ be a topological dynamical system, and let $S$ be a shape function with domain $\Omega$. Then there is an open and relatively compact subgroupoid $\mathcal G \subseteq X \rtimes\Gamma$ such that the unit space is $\{(x, e): x \in\Omega\}$ and $$\mathrm{Orbit}_{\mathcal G}(x) = xS(x),\quad x\in \Omega.$$
\end{thm}

\begin{proof}
For each finite subset $\mathcal F \subseteq\Gamma$, define
$$Z_{\mathcal F} : = \{x \in\Omega: S(x) = \mathcal F\}.$$
Since $x \mapsto \mathcal S(x)$ is equivariant, there are finite subsets $F_1, F_2, ..., F_n\subseteq \Gamma$ and a partition of $\Omega$
$$\Omega = (\bigsqcup_{\gamma \in F_1^{-1}} Z_{F_1} \gamma) \sqcup (\bigsqcup_{\gamma \in F_2^{-1}} Z_{F_2} \gamma) \sqcup \cdots \sqcup (\bigsqcup_{\gamma \in F_n^{-1}} Z_{F_n} \gamma).$$
For each $\gamma \in \Gamma$, define
$$Y_\gamma = (\bigsqcup_{\gamma_1 \in F^{-1}_1\setminus F^{-1}_1\gamma} Z_{F_1}\gamma_1) \sqcup( \bigsqcup_{\gamma_2 \in F^{-1}_2\setminus F^{-1}_2\gamma} Z_{F_2}\gamma_2) \sqcup \cdots \sqcup (\bigsqcup_{\gamma_n \in F^{-1}_n\setminus F^{-1}_n\gamma} Z_{F_n}\gamma_n) \subseteq \Omega.$$
Note that
$$\Omega\setminus Y_\gamma = (\bigsqcup_{\gamma_1 \in F^{-1}_1\cap F^{-1}_1\gamma} Z_{F_1}\gamma_1) \sqcup( \bigsqcup_{\gamma_2 \in F^{-1}_2\cap F^{-1}_2\gamma} Z_{F_2}\gamma_2) \sqcup \cdots \sqcup (\bigsqcup_{\gamma_n \in F^{-1}_n\cap F^{-1}_n\gamma} Z_{F_n}\gamma_n).$$

Let us verify that $\Omega\setminus Y_\gamma$ is open. 

Assume that $x\in \Omega\setminus Y_\gamma$, then
$$x \in (\bigsqcup_{\gamma_1 \in F^{-1}_1\cap F^{-1}_1\gamma} Z_{F_1}\gamma_1) \sqcup( \bigsqcup_{\gamma_2 \in F^{-1}_2\cap F^{-1}_2\gamma} Z_{F_2}\gamma_2) \sqcup \cdots \sqcup (\bigsqcup_{\gamma_n \in F^{-1}_n\cap F^{-1}_n\gamma} Z_{F_n}\gamma_n),$$ and let us assume that
$$x \in Z_{F_i} \gamma_i $$ for some $\gamma_i \in F^{-1}_i \cap F^{-1}_i \gamma$.
That is $$S(x) = F_i\gamma_i.$$
Since $x \mapsto \mathcal{S}(x)$ is lower semicontinuous, there is an open set $U \ni x$ such that
$$S(x) \subseteq S(y),\quad y\in U.$$ Let us show that $U\subseteq \Omega\setminus Y_\gamma$, and hence $\Omega\setminus Y_\gamma$ is open.
 
 For each $y\in U$, there is $F_j$ and $c\in \Gamma$ such that
$$F_i\gamma_i = S(x) \subseteq S(y) = F_jc.$$ Since $e\in F_jc$, one has $c\in F_j^{-1}$. Also note that $$F_i\gamma_i\gamma^{-1} \subseteq F_jc\gamma^{-1}$$ and $e\in F_i\gamma_i\gamma^{-1}$ (since $\gamma_i\in F_i^{-1}\gamma$); one has $e \in F_jc\gamma^{-1}$ and hence $c\in F_j^{-1}\gamma$. That is $$c\in F_j^{-1} \cap F_j^{-1} \gamma.$$ In particular, $$y \in \bigsqcup_{\gamma_j\in F^{-1}_j\cap F_j^{-1}\gamma} Z_{F_j}\gamma_j\subseteq X\setminus Y_\gamma.$$ This shows that $\Omega\setminus Y_\gamma$ is open.

Define $$\mathcal G := \{(x, \gamma) \in \Omega \times\Gamma: x\gamma \in \Omega\setminus Y_\gamma\}.$$ Since $\Omega\setminus Y_\gamma$ is open, $\mathcal G$ is an open (and relatively compact) subset containing $(x, e)$, $x\in \Omega$. Let us show that $\mathcal G$ is actually a subgroupoid.

Let $(x, \gamma) \in \mathcal G$. Then there are $F_i$ and $\gamma_i\in F_i^{-1}\cap F_i^{-1}\gamma$ such that $$x\gamma \in Z_{F_i} \gamma_i \subseteq \Omega,$$ and hence $x\in Z_{F_i}\gamma_i\gamma^{-1}$. Since $\gamma_i \in F_i^{-1} \cap F_i^{-1}\gamma$, one has $$\gamma_i\gamma^{-1} \in F_i^{-1} \cap F_i^{-1}\gamma^{-1},$$ and therefore, $$(x, \gamma)^{-1} = (x\gamma, \gamma^{-1}) \in\mathcal G$$ (note that $x\gamma \in \Omega$).

Let $x\in X$, $\gamma_1, \gamma_2 \in\Gamma$ with $$(x, \gamma_1),\ (x\gamma_1, \gamma_2) \in\mathcal G.$$ Then, there are $$F_{i_1},\  F_{i_2},\  \gamma_{i_1}\in F_{i_1}^{-1} \cap F_{i_1}^{-1}\gamma_1,\ \mathrm{and}\ \gamma_{i_2} \in  F_{i_2}^{-1} \cap F_{i_2}^{-1}\gamma_2$$ such that 
$$x\gamma_1 \in  Z_{F_{i_1}}\gamma_{i_1}\subseteq Z_{F_{i_1}} F_{i_1}^{-1}  \quad\mathrm{and}\quad x\gamma_1\gamma_2 \in Z_{F_{i_2}}\gamma_{i_2}.$$ Since $\gamma_{i_2} \in F_{i_2}^{-1}\gamma_2$, one has
$$x\gamma_1\gamma_2 \in Z_{F_{i_2}} F_{i_2}^{-1}\gamma_2$$ and hence $$x\gamma_1 \in Z_{F_{i_2}} F_{i_2}^{-1}.$$ Noting that $$Z_{F_{i_1}}F_{i_1}^{-1} \cap Z_{F_{i_2}}F_{i_2}^{-1} = \varnothing\quad \mathrm{if}\ F_{i_1} \neq F_{i_2}$$ and $$x\gamma_1 \in Z_{F_{i_1}}F_{i_1}^{-1} \cap Z_{F_{i_2}}F_{i_2}^{-1},$$ one has that $F_{i_1} = F_{i_2}$. So, let us denote both $F_{i_1}$ and $F_{i_2}$ by $F_i$.

Since $$ x\gamma_1\gamma_2 \in Z_{F_{i}} (F_{i}^{-1} \cap F_i^{-1}\gamma_2),$$ one has $$ x\gamma_1 \in Z_{F_{i}} (F_{i}^{-1} \gamma_2^{-1} \cap F_i^{-1}).$$ Since $$x\gamma_1 \in Z_{F_i} (F_i^{-1}\cap F_i^{-1}\gamma_1),$$ one has $$x\gamma_1 \in Z_{F_i} (F_i^{-1}\cap F_i^{-1}\gamma_1\cap F_i^{-1}\gamma_2^{-1}),$$ and hence
$$x\gamma_1\gamma_2 \in Z_{F_i} (F_i^{-1} \cap F_i^{-1} \gamma_2 \cap F_i^{-1}\gamma_1\gamma_2) \subseteq Z_{F_i} (F_i^{-1} \cap F_i^{-1}\gamma_1\gamma_2).$$
Therefore, $$x\gamma_1\gamma_2 \in \Omega\setminus Y_{\gamma_1\gamma_2} \quad \mathrm{and}\quad (x, \gamma_1)(x\gamma_1, \gamma_2) \in \mathcal G.$$ Thus $\mathcal G$ is a subgroupoid of $X\rtimes\Gamma$.

It is clear that the unit space of $\mathcal G$ is $\{(x, e): x\in\Omega\}$. Let $x \in \Omega$, and let us calculate the $\mathcal G$-orbit of $x$. Assume that $x\in Z_{F_i}$ for some $F_i$. Then $$\mathrm{Shape}(x) = F_i.$$ If $(x, \gamma) \in \mathcal G$ for some $\gamma \in\Gamma$, then $$x\gamma \in \Omega\setminus Y_\gamma = (\bigsqcup_{\gamma_1 \in F^{-1}_1\cap F^{-1}_1\gamma} Z_{F_1}\gamma_1) \sqcup( \bigsqcup_{\gamma_2 \in F^{-1}_2\cap F^{-1}_2\gamma} Z_{F_2}\gamma_2) \sqcup \cdots \sqcup (\bigsqcup_{\gamma_n \in F^{-1}_n\cap F^{-1}_n\gamma} Z_{F_n}\gamma_n).$$
Assume that $x\gamma \in Z_{F_j}\gamma_j$ for some $\gamma_j \in F_j^{-1} \cap F_j^{-1}\gamma$. Then, $x \in Z_{F_j}\gamma_j\gamma^{-1}$. Note $$\gamma_j\gamma^{-1} \in F_j^{-1}\gamma^{-1} \cap F_j^{-1}\subseteq F_j^{-1},$$ one has that $x\in Z_{F_j}F_j^{-1}$, and therefore $$F_j = F_i\quad \mathrm{and}\quad \gamma = \gamma_j \in F_i.$$

On the other hand, for any $\gamma \in F_i$, it is clear that $(x, \gamma) \in \mathcal G$. Therefore 
$$\mathcal G = \{(x, \gamma) \in \Omega \times \Gamma: \gamma \in S(x)\},$$ and hence
$$\mathrm{Orbit}_{\mathcal G}(x) = x S(x).$$
\end{proof}

\begin{rem}
By considering 
$$\overline{\mathcal G} = \mathcal G \cup \{(x, e): x \in X\},$$ one extends $\mathcal G$ to a (open and relatively compact) subgroupoid with unit space $\{(x, e): x \in X\}$ without changing orbit of each $x\in\Omega$.
\end{rem}

\begin{defn}\label{defn-tiling-property}
A free dynamical system $(X, \Gamma)$ is said to have lower semicontinuous dynamical tiling property (LscT) if for any finite set $K\subseteq\Gamma$ and any $\eps>0$, there is a partition valued function $\mathcal T: X \to \mathcal{P}(\Gamma)$ such that
\begin{enumerate}
\item[(1)] there are finite sets $F_1, F_2, ..., F_n\subseteq \Gamma$ such that for each $x\in X$, the partition $\mathcal T(x)$ is a tiling of $\Gamma$ with tiles $F_1, F_2, ..., F_n$, and $F_1, F_2, ..., F_n\subseteq \Gamma $ are $(K, \eps)$ invariant,
\item[(2)] the map $\mathcal T$ is equivariant, i.e., $$\mathcal T(x\gamma) = \mathcal T(x)\gamma^{-1},\quad \gamma\in \Gamma,$$
\item[(3)] the function $\mathcal T$ is lower semicontinuous, i.e., for any finite set $F \subseteq \Gamma$ and any $x\in X$, there is a neighbourhood $U\ni x$ such that $$\mathcal P(x)|_F\ \mathrm{refines}\ \mathcal P(y)|_F,\quad y\in U.$$
\end{enumerate}

The system $(X, \Gamma)$ is said to have continuous dynamical tiling property (CT) if Condition (3) is strengthen to
\begin{enumerate}
\item[(3')]  the function $\mathcal T$ is continuous, i.e., for any finite set $F \subseteq \Gamma$ and any $x\in X$, there is a neighbourhood $U\ni x$ such that $$\mathcal P(x)|_F = \mathcal P(y)|_F,\quad y\in U,$$ where $\mathcal P(x)|_F$ and $\mathcal P(y)|_F$ denote  partitions of $F$ induced by $\mathcal P(x)$ and $\mathcal P(y)$, respectively.
\end{enumerate}
 
\end{defn}

The following lemma is straightforward.
\begin{lem}\label{LscT-ext}
If a free dynamical system $(Y, \Gamma)$ has (LscT) (or (CT)), and if $(X, \Gamma)$ is an extension of $(Y,  \Gamma)$, then $(X, \Gamma)$ has (LscT) (or (CT)).
\end{lem}

Once the dynamical system $(X, \Gamma)$ has the (LscT), then arbitrarily invariant shape functions (hence small subgroupoids) exist:
\begin{cor}\label{tiling-to-groupoid}
Let $(X, \Gamma)$ be a free dynamical system with (LscT). Then, for any finite set $K\subseteq\Gamma$ and any $\eps>0$, there is a small subgroupoid $\mathcal G \subseteq X \rtimes\Gamma$ such that for any $x\in X$, the $\mathcal G$-orbit $\mathrm{Orbit}_{\mathcal G}(x)$ is $(K, \eps)$-invariant.
\end{cor}
\begin{proof}
For each $x\in X$, define $$S(x) = (\mathcal T(x))_e,$$ where $(\mathcal T(x))_e$ is the tile of $\mathcal T(x)$ containing $e$. Then $x \mapsto S(x)$ is a shape function with domain $X$ in the sense of Definition \ref{defn-shape-function}, and the statement follows directly from Theorem \ref{Shape-to-groupoid}.
\end{proof}

Together with Corollary \ref{orb-comp-measure}, one has
\begin{cor}\label{tiling-to-comp}
Let $(X, \Gamma)$ be a free dynamical system with (LscT). Then for any open sets $E, F \subseteq  X$ with 
$$\mu(E) < \frac{1}{4} \mu(F),\quad \mu\in \mathcal M_1(X, \Gamma),$$ one has $$\varphi_E \precsim \varphi_F\quad \mathrm{in}\quad \mathrm{C}(X) \rtimes \Gamma.$$

\end{cor}

The topological-dynamical version of subequivalence and comparison have a long history, see, for example, \cite{GPS-Cantor} and \cite{GW-CM}. The following definitions can be found in \cite{Kerr-D} and \cite{DZ-comparison}.
\begin{defn}\label{defn-dyn-comp}
Let $G$ be a countable amenable group.
\begin{enumerate}
\item Let $G$ act on a zero-dimensional compact metric space $X$. For two clopen sets $A, B \subseteq X$, we say that $A$ is subequivalent to $B$ (and write $A \precsim B$), if there exists a finite partition $A = \bigsqcup_{i=1}^k A_i$ of $A$ into clopen sets and there are elements $g_1, g_2, ..., g_k$ of $G$ such that $g_1(A_1)$, $g_2(A_2)$, ..., $g_k(A_k)$ are disjoint subsets of $B$. We say that the action admits comparison if for any pair of clopen subsets $A, B$ of $X$, the condition that for each invariant measure $\mu$ on $X$ we have $\mu(A) < \mu(B)$, implies $A \precsim B$.
\item If every action of $G$ on any zero-dimensional compact metric space admits comparison then we will say that $G$ has the comparison property.
\end{enumerate}

\end{defn}

Also recall 
\begin{thm}[Theorem 5.11 and Theorem 6.2 of \cite{DZ-comparison}]\label{DZ-C}
Any amenable group with subexponential growth has the comparison property. Any free Cantor system $(\Omega, \Gamma)$ with $\Gamma$ an amenable group with the comparison property has (CT).
\end{thm}

Then, for extensions of a Cantor system, one has the following Cuntz comparison of open sets.
\begin{cor}\label{comp-open-set-ext}
If $\Gamma$ has comparison property in sense of Definition \ref{defn-dyn-comp}  (in particular, if $\Gamma$ has subexponential growth) and $(X, \Gamma)$ has a free Cantor factor, then, for any open sets $E, F \subseteq X$ with 
$$\mu(E) < \frac{1}{4} \mu(F),\quad \mu\in \mathcal M_1(X, \Gamma),$$ one has that $$\varphi_E \precsim \varphi_F\quad \textrm{in}\quad \mathrm{C}(X) \rtimes \Gamma.$$ In other words, $\mathrm{C}(X) \rtimes \Gamma$ has $(\frac{1}{4}, 1)$-Cuntz-comparison of open sets.
\end{cor}
\begin{proof}
Assume that $(X, \Gamma)$ is an extension of a free Cantor system $(\Omega, \Gamma)$. By Theorem \ref{DZ-C}, the factor $(\Omega, \Gamma)$ has the (CT), and then, by Lemma \ref{LscT-ext}, $(X, \Gamma)$ has the (CT) (in particular, the (LscT)). The statement follows from Corollary \ref{tiling-to-comp}.
\end{proof}

\begin{cor}\label{main-thm-G}
Let $(X, \Gamma)$ be a minimal free topological dynamical system which is an extension of a free Cantor system. If $\Gamma$ has the comparison property in sense of Definition \ref{defn-dyn-comp} (in particular, if $\Gamma$ has subexponential growth), then 
$$\mathrm{rc}(\mathrm{C}(X) \rtimes \Gamma) \leq\frac{1}{2}\mathrm{mdim}(X, \Gamma).$$
\end{cor}
\begin{proof}
By Corollary \ref{URP-ext-SBP}, $(X, \Gamma)$ has the (URP). Since $\Gamma$ is assumed to have the comparison property (Definition \ref{defn-dyn-comp}), it follows from Corollary \ref{comp-open-set-ext} that $\mathrm{C}(X) \rtimes\Gamma$ has $(\frac{1}{4}, 1)$-Cuntz-comparison on open sets.  Then the statement follows from Theorem \ref{main-thm}.
\end{proof}

\bibliographystyle{plainurl}
\bibliography{operator_algebras}

\end{document}